\documentclass[11pt, leqno]{article}
\usepackage[numbers]{natbib}

\usepackage[utf8]{inputenc}
\usepackage{amsmath,amssymb}
\usepackage[export]{adjustbox}
\usepackage{hyperref}
\usepackage{bbm}
\usepackage{bm}
\usepackage{geometry}
\usepackage{comment}
\usepackage[ruled,vlined]{algorithm2e}
\usepackage[toc,page]{appendix}
\usepackage{subfig}
\usepackage{authblk}
\usepackage{tabularx}
\usepackage{booktabs}
\usepackage[labelfont=bf,format=plain,justification=raggedright,singlelinecheck=false]{caption}
\usepackage{fullpage}
\usepackage[symbol]{footmisc}

\title{Convergence rates of non-stationary and deep Gaussian process regression}
\author[1]{Conor Osborne\thanks{conor.osborne@ed.ac.uk}}
\author[1]{Aretha L. Teckentrup\thanks{a.teckentrup@ed.ac.uk}}
\affil[1]{School of Mathematics and Maxwell Institute for Mathematical Sciences, University of Edinburgh, King’s Buildings, Edinburgh EH9 3FD, UK}

\date{}

\newtheorem{theorem}{Theorem}[section]
\newtheorem{corollary}[theorem]{Corollary}
\newtheorem{lemma}[theorem]{Lemma}
\newtheorem{proposition}[theorem]{Proposition}

\newtheorem{definition}[theorem]{Definition}

\newtheorem{remark}[theorem]{Remark}
\newtheorem{assumption}{Assumption}
\newenvironment{proof}{\paragraph{Proof:}}{\hfill$\square$}

\DeclareMathOperator{\warp}{warp}
\DeclareMathOperator{\mix}{mix}  
\DeclareMathOperator{\conv}{conv} 

\DeclareMathOperator{\inv}{inv}

\DeclareMathOperator{\MS}{MS}
\DeclareMathOperator*{\esssup}{ess \, sup}

\def\bE{\mathbb{E}}
\def\bN{\mathbb{N}}
\def\bP{\mathbb{P}}

\def\bR{\mathbb{R}}

\def\R{\bR}
\def\N{\bN}
\def\P{\bP}
\def\E{\bE}
\def\C{\mathbb{C}} 
\def\GP{\mathcal{GP}} 
\def\TGP{\mathcal{TGP}} 
\def\H{\mathcal{H}}

\def\M{\text{Mat}}
\def\a{\alpha}
\def\s{\sigma}
\def\O{\Omega}

\def\low{\text{low}}
\def\up{\text{up}}

\usepackage{natbib}
\usepackage{graphicx}
\usepackage{xcolor}

\usepackage{geometry}
\usepackage{xcolor}
\usepackage{amsmath}
\usepackage[some]{background}

\definecolor{titlepagecolor}{cmyk}{1,.60,0,.40}

\backgroundsetup{
scale=1,
angle=0,
opacity=1,
contents={\begin{tikzpicture}[remember picture,overlay]
 \path [fill=titlepagecolor] (current page.west)rectangle (current page.north east); 
 \draw [color=white, very thick] (5,0)--(5,0.5\paperheight);
\end{tikzpicture}}
}
\providecommand{\keywords}[1]
{
  \small	
  \textit{AMS 2020 subject classifications:} #1
}

\makeatletter                   
\def\printauthor{%
    {\large \@author}}          
\makeatother

\begin{document}

\maketitle

\begin{abstract}
        The focus of this work is the convergence of non-stationary and deep Gaussian process regression. More precisely, we follow a Bayesian approach to regression or interpolation, where the prior placed on the unknown function $f$ is a non-stationary or deep Gaussian process, and we derive convergence rates of the posterior mean to the true function $f$ in terms of the number of observed training points. In some cases, we also show convergence of the posterior variance to zero. The only assumption imposed on the function $f$ is that it is an element of a certain reproducing kernel Hilbert space, which we in particular cases show to be norm-equivalent to a Sobolev space. Our analysis includes the case of estimated hyper-parameters in the covariance kernels employed, both in an { empirical Bayes} setting and the particular hierarchical setting constructed through deep Gaussian processes. We consider the settings of noise-free or noisy observations on deterministic or random training points. We establish general assumptions sufficient for the convergence of deep Gaussian process regression, along with explicit examples demonstrating the fulfilment of these assumptions. Specifically, our examples require that the H\"older or Sobolev norms of the penultimate layer are bounded almost surely.

\end{abstract}
{\keywords{62G08, 62G20, 60G17, 65C20, 68T07}}

\section{Introduction}
Given a model, such as a simulator of a real world physical process, we are interested in forming an approximation, or {\it emulator}, of this model. Following the Bayesian framework \cite{kennedy2001bayesian,sacks1989design}, we construct a \emph{prior distribution} over the space of potential models. This prior distribution encompasses our initial knowledge of the model. Given some observations, which typically come in the form of input-output pairs, we then condition the prior distribution on the observations to construct a more informative probability distribution, known as the \emph{posterior distribution}. This posterior distribution then encompasses our initial knowledge, as well as the observations. An advantage of the Bayesian framework is that the posterior distribution allows for uncertainty quantification in our approximation.

Mathematically, the model of interest corresponds to a function, mapping inputs to corresponding outputs. A popular choice of prior distributions on functions are Gaussian processes (GPs), see e.g. \cite{Rasmussen06gaussianprocesses,kennedy2001bayesian,helin2023introduction} for an introduction to the topic.
A GP is completely described by its mean and covariance functions, and these functions can then be designed to describe features we expect to see in our model, such as typical length scales and amplitudes, or a prescribed regularity. In this work, we assume a zero mean in the prior distribution, and focus on sophisticated covariance kernels to capture the features of the model. The posterior mean, which is a linear combination of kernel evaluations, then inherits the features from the covariance kernel.

When the model being recovered is non-stationary, here meaning that it behaves differently in different parts of its domain, 
standard choices of covariance kernels such as the Mat\'ern family \cite{Rasmussen06gaussianprocesses,matern} fail to capture essential properties of the model, and as a result, the accuracy of the emulator is often poor. Many modern applications involve non-stationary models (see e.g. \cite{Ping2020,Nabney1996,Nason2006}), and a 
considerable amount of work has been done over the last decades to design non-stationary covariance kernels (see e.g. the recent survey \cite{sauer2023nonstationary}).
However, although the superior performance of these methods has been observed in practice, convergence properties of these methods are often still less understood than for classical, stationary constructions.

In the last decade, the use of deep GPs (DGPs) has further gained traction as a method to create flexible priors, including capturing non-stationary behaviour \cite{damianou2013deep, dunlop2018deep, Sauer2022, finocchio2023posterior}. DGPs are no longer GPs themselves, but are defined through a hierarchy of stochastic processes that are conditionally Gaussian. In particular, the stochastic process at layer $n$ of the hierarchy is used to define the covariance structure of stochastic process at layer $n+1$, and conditioned on the $n$-th layer, the distribution of the $n+1$-th layer is Gaussian. DGPs can be seen as a particular example of a hierarchical statistical model.
Although DGPs have been shown to work much better than standard stationary GPs in a  wide variety of applications (see e.g. \cite{Muir2023,Desai_2023,Soltanpour2023}), 
the theoretical understanding of DGPs, including convergence properties, is still limited.

The focus of this work is to study convergence properties of regression (in the case of noisy training data) or interpolation (in the case of noise-free training data) with non-stationary and deep GP priors. More precisely, we follow a Bayesian approach where the prior placed on the unknown function $f$ is a non-stationary or deep Gaussian process and the data comes in the form of noisy or noise-free point evaluations of $f$, and we derive convergence rates of the posterior mean to the true function $f$ in terms of the number of observed training points. In some cases, we also show convergence of the posterior variance to zero. Whether or not the training data is assumed noisy depends on the context. If the training data consists of real data, the assumption that the observed function values contain noise is common. If on the other hand the training data is simulated data, obtained from simulating a mathematical model on a computer, assuming no noise in the function values may be more {appropriate.}


A large body of work exists exploring the convergence analysis of GP regression (or interpolation) under various assumptions on the training data and the true function $f$, see e.g. the seminal works  \cite{Choi2007OnPC, Stein1999,  van_der_Vaart_2008,Vaart2011,sss13} as well as the more recent works \cite{hoffmann2015adaptive, karvonen2020maximum,wynne2021convergence, nieman2022contraction}. However, the vast majority of results known in the literature are often either proven for specific, stationary kernels, or they are proven for general kernels under assumptions that have so far only been verified for stationary kernels. 
In this work, our primary interest is in the case where the training data is noise-free, and we hence use proof techniques that give optimal convergence rates in this case. The same approach was already used in for example \cite{sss13,karvonen2020maximum,Teckentrup2019}, and uses tools mostly from the scattered data approximation literature \cite{wendland_2004}. Although we do obtain convergence rates also in the case of noisy data, the rates we obtain are not optimal. We consider the case where the locations of the training data are fixed, given e.g. by a rule with known theoretical properties, or sampled randomly from a specified distribution.

In practice, GP and DGP regression is often used in a black-box context, where properties of the function $f$ to be approximated are largely unknown. Therefore, in the majority of our results the only assumption that we impose on $f$ is that it belongs to a Sobolev space $H^\beta(\O)$ (or a corresponding tensor product version), for a bounded domain $\O \subseteq \R^d$ and some $\beta \geq d/2$. This translates into assuming that $f$ is continuous and that it possesses a certain number of square-integrable weak derivatives, which is often reasonable. Considering this setting allows us to (i) justify the use of the methodology in a black-box setting, and (ii) conclude that the methodology is robust to misspecification of the ``non-stationarity" in the sense that our results do not require the non-stationary structure of $f$ and that of the prior distribution to match. 

As such, this work complements the recent works \cite{finocchio2023posterior,Bachoc2024,castillo2024deep,abraham2023deep}, which show that for functions $f$ with a certain compositional structure, a DGP prior based on compositions can be constructed which exploits the structure and gives faster convergence rates. We believe it is possible to extend the analysis presented here to include further structural assumptions on $f$ and show faster convergence of non-stationary GP and DGP regression compared to stationary counterparts, but this is not straight-forward and has therefore been left for future work. 



Finally, we note here that as in \cite{finocchio2023posterior,Bachoc2024,castillo2024deep,abraham2023deep}, our final results on the convergence of DGP regression for specific constructions require a truncation of the (conditionally) Gaussian process on the penultimate layer, such that appropriate H\"older or Sobolev norms of the penultimate layer are uniformly bounded almost surely. Although we do present general convergence results for DGP regression that do not require this truncation, we were unable to verify the assumptions of these results without enforcing the truncation. 




\subsection{Our contributions}
In this paper we make the following novel contributions to the analysis of non-stationary and deep GP regression, of which the first two are also of independent interest:
\begin{enumerate}
    \item We show that for two typical non-stationary kernels, namely the warping and mixture kernels introduced in section \ref{subsubsec:nonstationary}, the reproducing kernel Hilbert space (RKHS) can be found explicitly, and we prove that under given conditions the RKHS is norm-equivalent to a Sobolev space. These results are given in Theorem \ref{lemma:sob equiv warping} and \ref{lemma: sob equiv mixture}.
    \item We present results on the regularity properties of sample paths of a non-stationary Gaussian process for general kernels, as given in Lemmas \ref{lem:old school sob regularity method} and \ref{lemma: Sullivan regularity}. Additionally, specific applications to the warping and mixture kernels  are discussed in Corollaries \ref{cor:sample path old school} and \ref{cor: sample path sullivan}.
    \item We prove convergence rates for non-stationary GP regression in terms of the number of training points, under various assumptions on the true function $f$ being approximated and the training data used. 
    
    For noise-free data, these results are largely based on the known results for general kernels given in Propositions \ref{thm:Wend11.13} and \ref{thm:Convergence_in_sob}, with the required assumptions proven here for the warping and mixture kernels. The main final results are given in Corollaries \ref{cor:convergence in the native space}, \ref{cor:conv_sob}, \ref{cor:conv_rkhs_est}, \ref{cor:conv_sob_est}, \ref{cor:conv_sep} and \ref{cor:conv_random}. Corollaries \ref{cor:convergence in the native space} and \ref{cor:conv_rkhs_est} also apply to the convolution kernel from section \ref{subsubsec:nonstationary}.
    
    For noisy data, results for general kernels are given in Theorem \ref{thm:conv_sob_noisy} (adapted from \cite[Theorem 2]{wynne2021convergence}) and \cite[Theorem 2]{Vaart2011}, and results for specific kernels are given in  Corollaries \ref{cor:conv_sob_noisy}, \ref{cor:conv_sob_est_noisy} and \ref{cor:conv_random w/ noise}.
    \item We prove convergence rates for deep GP regression in terms of the number of training points used, under general assumptions on the true function $f$ being approximated and the training data used. The results for general kernels and constructions is given in Theorems \ref{thm:exp convergence rate} and \ref{thm:exp convergence rate_wgp}.
    
   Explicit constructions of (constrained) deep GPs based on the warping and mixture kernels, and satisfying the assumptions of Theorem \ref{thm:exp convergence rate} or \ref{thm:exp convergence rate_wgp}, are provided in Lemmas \ref{lem:DGPkernelWarp}, \ref{lem:DGPkernelMix} and  \ref{lem:DGPWideGP}.
\end{enumerate}

To the best of our knowledge, the error bounds proven in items 3 and 4 are the first to apply in the settings considered in this work. In particular, the analysis in section \ref{sec:EADGPR} is the first to show convergence of deep GP regression under the general assumption that the true function $f$ belongs to a Sobolev space, rather than imposing further structural assumptions (as in e.g. \cite{finocchio2023posterior,Bachoc2024,castillo2024deep,abraham2023deep}).

\subsection{Paper structure}

This paper is structured as follows. 
In section \ref{sec:GPR}, we introduce the general set-up of GP regression, together with standard choices of stationary and non-stationary covariance kernels and the estimation of related hyper-parameters. In section \ref{sec:nativespaces}, we introduce and construct native spaces for the non-stationary covariance kernels, and show that these are norm-equivalent to Sobolev spaces in special cases. Section \ref{sec:sample path regularity} gives results on the sample path regularity of non-stationary Gaussian processes. Sections \ref{sec:EAGPR} and \ref{sec:EADGPR} are devoted to the error analysis on non-stationary and deep GP regression, respectively. Finally, in section \ref{sec:numsim}, we present simple numerical simulations that illustrate the theory. Section \ref{sec:conclusion} offers some conclusions and discussion.

\subsection{Preliminaries and notation}
Throughout this work, $\O\subset\R^d$, for $d\in\N$, will be a bounded Lipschitz domain that satisfies an interior cone condition \cite{wendland_2004}.
For any multi-index \(q=(q_1,\ldots,q_n)\in\N^{d}_0\) with length \(|q|=q_1+\ldots+q_n\) we denote the \(q\)-th derivative of a function \(g:\O \rightarrow \R\) by 
\[
    D^{q}g=\frac{\partial^{|q|}}{\partial^{ q_1}\ldots\partial^{q_d}}g.
\]
In the case $d=1$, we use the simpler notation $g' := D^1 g$ and $g^{(q)}:= D^q g$.
{ When \(g\) is defined on \(\O \times \O\), the operator \(D^{q_1, q_2}\) computes the \(q_1\)-th partial derivative with respect to the first input and the \(q_2\)-th partial derivative with respect to the second input, where \((q_1, q_2) \in \mathbb{N}_0^{2d}\).
For \(p\geq0\) we use the notation \(g\in C^{p}(\O)\) if for all multi-indices \(|q|\leq\lfloor p\rfloor\) the derivative \(D^{q}g\) exists and is \(p-\lfloor p\rfloor\)-H\"older continuous.}
The norm on \(C^{p}(\O)\) is given by
\begin{equation*}
    \|g\|_{C^{p}(\O)} = \sum_{|\beta|\leq \lfloor p\rfloor}\sup_{u\in\O}|D^{\beta}g(u)|+\sum_{|\beta|= \lfloor p\rfloor}\sup_{u,u'\in\O,u\neq u'}\frac{|D^{\beta}g(u)-D^{\beta}g(u')|}{\|u-u'\|_2^{p-\lfloor p\rfloor}}.
\end{equation*}
When \(p\) is an integer this norm is the standard norm on continuously differentiable functions. The space of smooth functions is defined as \(C^{\infty}(\O)=\cap_{p\geq 0}C^{p}(\O)\).
Additionally, we use the notation \(k\in C^{p,p}(\O\times\O)\) for \(p\geq 0\)  when \(k(u_1,\cdot)\in C^{p}(\O)\) and \(k(\cdot,u_2)\in C^{p}(\O)\) for all \(u_1, u_2\in \O \). Note that \(C^{2p}(\O\times\O)\subset C^{p,p}(\O\times\O)\) for all \(p \geq 0\).

For \(\beta\in\N\), the Sobolev space of functions with square integrable weak derivatives up to the \(\beta\)-th order is denoted \(H^{\beta}(\O)\).
For \(\beta\in\R_+\setminus\N\), with $1 \leq \beta < \infty$, we define \(H^{\beta}(\O)\) as a fractional order Sobolev space known as a Sobolev–Slobodeckij space (see, e.g. \cite{Aronszajn1955}), 
    equipped with the norm
    \[
        \|g\|_{H^{\beta}(\O)}:= \sqrt{ \sum_{|q|\leq \lfloor \beta \rfloor}\int_{\O}|D^{q}g(u)|^2 \mathrm{d} u} + \sup_{|q| = \lfloor \beta \rfloor} \sqrt{\int_\O\int_\O\frac{|D^{q}g(u)-D^{q}g(v)|^2}{|u-v|^{2 (\beta - \lfloor \beta \rfloor)+1}}\mathrm{d}u \mathrm{d}v}.
    \]
    The special case $\beta=0$ is denoted by $L^2(\O) := H^0(\O)$.
For \(\beta\in\N\), \(W^{\beta,\infty}(\O)\) denotes the space of functions with essentially bounded weak derivatives up to the \(\beta\)-th order, equipped with the norm 
\[
\|g\|_{W^{\beta,\infty}(\O)} = \sum_{|\beta|\leq \lfloor p\rfloor}\esssup_{u\in\O}|D^{\beta}g(u)|.
\]

For vector spaces \(V_1,V_2\) the notation \(V_1\hookrightarrow V_2\) denotes the continuous embedding of \(V_1\) into \(V_2\) and \(V_1\cong V_2\) denotes that \(V_1\) and \(V_2\) are equal as vector spaces and their respective norms are equivalent. For any function \(g:V_1\to V_2\) we denote by \(\inv(g):V_2\to V_1\) the inverse function of \(g\), whereas $1/g$ refers to its reciprocal function.

We have further summarised the notation used in this work in Table \ref{table:notation}
.

\begin{table}
\caption{Notation}
\begin{tabularx}{\textwidth}{@{}cl@{}}
\toprule
    \(\O\)  & Bounded Lipschitz domain satisfying an interior cone condition \\ 
  \(d\) & Dimension of domain \(\O\subset\R^d\) \\
  $N$ & Number training points \(U_N\subset\O \) \\
  \(L\) & Number of components in mixture kernel \(k_{\mix}^{\{\s_\ell,k_\ell\}_{\ell=1}^L}\)\\
  \(D\) & Depth of a deep GP \(\{f^n\}_{n=0}^{D-1}\)\\
  \(\GP\) & A Gaussian process\\
  \(k_{\M(\nu)}\) & Mat\'ern kernel \\
  \(\s^2\) &  Marginal variance of Mat\'ern kernel \\
  \(\lambda\) & Correlation length of Mat\'ern kernel\\
  \(\nu\) & Smoothness parameter of Mat\'ern kernel \\
  \(k_{\warp}^{w,k_s}\) & Warping kernel, see \eqref{eq:warpingkernel}\\
  \(k_{\mix}^{{\{\s_\ell\}}_{\ell=1}^L}\) & Mixture kernel, see \eqref{eq:kernel_mixture}\\
   \(k_{\conv}^{\lambda,k_i}\) & Convolution kernel, see \eqref{eq:convolution_kernel}\\
    \(\H_k(\O)\) & The native space relating to kernel \(k\) on \(\O\)\\
  \(V_1 \cong V_2\) & \(V_1\) and \(V_2\) are equal as vector spaces and their respective norms are equivalent\\
  \(V_1\hookrightarrow V_2\) & \(V_1\) is continuously embedded into \(V_2\)\\
    \(C_{\low}\), \(C_{\up}\) & Constants for norm equivalences\\
  \(\inv(g)\) & The inverse function of a function \(g\)\\
  \(1/g\) & The reciprocal function of a function \(g\)\\
  \(g{'}\) & the \(1\)-st derivative of a univariate function \(g\)\\
  \(g^{(r)}\) & the \(r\)-th derivative of a function \(g\)\\
  \(\lceil x\rceil\) & the ceiling of \(x\in\R\), the smallest integer greater than or equal to \(x\)\\
  \(\lfloor x\rfloor\) & the floor of \(x\in\R\), the greatest integer less than or equal to \(x\)\\
  \(\N_0\) & Set of natural numbers \(\N\) and \(0\)\\
  \(\R_+\) & Set of positive real numbers\\
   \(\mathbbm{1}_{A}\) & Indicator function on a set \(A\)\\ 
   \(\mathbf{I}_N\) & Identity matrix of size \(N\)\\
   \(H^\beta(\O)\) & Sobolev space \(W^{\beta,2}(\O)\)\\
\bottomrule
\end{tabularx} \label{table:notation}
\end{table}


\section{GP regression}\label{sec:GPR}

We want to use GP regression to derive approximations of functions \(f:\O\subset\R^d\to\R\). We will focus on the convergence of these approximations as the number of training points \(N\) tends to infinity, but first give an introduction to the general methodology.

\subsection{Set up}
Let $f:\O\to\R$ be an arbitrary function, with $\O\subset\R^d$ a bounded Lipschitz domain that satisfies an interior cone condition \cite{wendland_2004}.  Many domains \(\Omega\) satisfy an interior cone condition; for instance, the unit cube \([0,1]^{d}\) is a typical example. Specifically, this assumption is necessary in Propositions \ref{thm:Wend11.13} and \ref{thm:Convergence_in_sob}, Theorems \ref{thm:conv_sob_noisy}, \ref{lem:cond on estimates for Sob bnds}, \ref{thm:exp convergence rate} and \ref{thm:exp convergence rate_wgp}, as well as all associated lemmas and corollaries. 

Denote by $U_N:=\{u_n\}_{n=1}^N\subset \O$ a set of $N\in\N$ distinct training points where $f(u_n)$ is observed, possibly with noise. Collectively, we denote this data as 
$y_N := \{u_n, f(u_n)+\varepsilon_n\}_{n=1}^N$, where $\varepsilon_n \sim N(0, \delta^2)$ { i.i.d for $\delta^2 \geq 0$. Note that the i.i.d. assumption can be relaxed in the misspecified setting where the assumption that $\varepsilon_n \sim N(0, \delta^2)$ is not satisfied. This is discussed in section \ref{sec:conv_sobolev_noisy}.} In our analysis, we will be interested in the setting of {\em noisy data}, where $\delta^2 >0$, as well as the setting of {\em noise-free data}, where $\varepsilon_n \equiv 0$ and formally $\delta^2=0$.

To recover $f$ from $y_N$ in the Bayesian framework, we first assign a GP \emph{prior} to $f$, denoted by 
\begin{equation}\label{eq:GP empirical prior}
f_0\sim\mathcal{GP}(0, k(u,u')).
\end{equation}
Here, 
$k:\O\times \O\to\R$ is a {positive semi-definite covariance function (or covariance kernel), that is, for any \(n\in\N\), \(c \in\R^n\setminus\{0\}\) and  \(u_1\ldots,u_n\in \O\), we have that\footnote{We use the convention of \cite{wendland_2004} whereby a non-strict inequality is used to define positive semi-definiteness. The kernel is referred to as positive definite if strict inequality holds for pairwise distinct \(u_1\ldots,u_n\in \O\). 
}
\begin{equation*} \sum_{i=1}^{n}\sum_{j=1}^nc_ic_jk(u_i,u_j) \geq 0.
\end{equation*}}
For ease of presentation, we have chosen the mean function in \eqref{eq:GP empirical prior} to be zero, but all results in this paper extend to the case where a non-zero mean function is used, under suitable assumptions (see, e.g. \cite{Teckentrup2019}).

The prior distribution encapsulates our prior knowledge of the function $f$, and in particular is independent of the data \(y_N\). Intuitively, it should give a higher probability to the types of functions we expect to see, and this is typically reflected in the choice of mean and covariance function. Typical choices for the covariance function $k$ are discussed in section \ref{subsec:kernels}.

{We condition the prior on the training data \(y_N\) to obtain the \emph{posterior} distribution
\begin{equation}\label{eq:gp_posterior}
f_N := f_0 |y_N \sim\mathcal{GP}(m_{N,\delta^2}^f(u), k_{N,\delta^2}(u,u')),
\end{equation}
with mean and the covariance function given by (see e.g. \cite{Rasmussen06gaussianprocesses})
\begin{align}
\label{eq:gp_mean}
m_{N,\delta^2}^f(u)&=k(u,U_N)^T (K(U_N,U_N) + \delta^2 \mathrm{I})^{-1}f^{\varepsilon}(U_N), \\
\label{eq:gp_variance}
    k_{N,\delta^2}(u,u')&=k(u,u')-k(u,U_N)^T (K(U_N,U_N) + \delta^2 \mathrm{I})^{-1}k(u',U_N),
\end{align}
where $f^{\varepsilon}(U_N)=\left[f(u_1) + \varepsilon_1,\ldots,f(u_N)+\varepsilon_N\right]\in\R^N$, $k(u,U_N)=\left[k(u,u_1),\ldots,k(u,u_N)\right]\in\R^N$ and $K(U_N,U_N)\in\R^{N\times N}$ is the matrix with $ij$\textsuperscript{th} entry $k(u_i,u_j)$. Note that $K(U_N,U_N) + \delta^2 \mathrm{I}$ is invertible for $\delta^2 > 0$ since we have assumed that $k$ is positive semi-definite. In the noise-free case, in which the above equations hold with $\delta^2=0$, we require $k$ to be positive definite.}

One of the major advantages of using the Bayesian framework outlined above is that it allows us to perform uncertainty quantification. We are not given a point estimate for the unknown function $f$ given the data $y_N$, but rather a distribution over a suitable function space. Calculating variances, and hence error estimates, is therefore possible, and this can be crucial in applications of GP regression in computational pipelines (see, e.g. \cite{Stuart2016}).

\subsection{Covariance kernels}\label{subsec:kernels}
We begin our discussion on covariance kernels by introducing the widely used family of stationary kernels known as Mat\'ern kernels. Following this, we discuss various methods to construct non-stationary kernels from stationary ones.

\subsubsection{Mat\'ern kernel}\label{subsubsec:mat}
  The family of Mat\'ern kernels (see e.g. \cite{Rasmussen06gaussianprocesses}) is given by
\[
    k_{\M(\nu)}(u,u')=\s^2\frac{2^{1-\nu}}{\Gamma(\nu)}\left(\sqrt{2\nu}\frac{\|u-u'\|_2}{\lambda}\right)K_{\nu}\left(\sqrt{2\nu}\frac{\|u-u'\|_2}{\lambda}\right),
\]
where \(\Gamma\)  is the gamma function, \(B_{\nu}\) is the modified Bessel function of the second kind and \(\s^2, \lambda, \nu>0\) are positive parameters. The parameter \(\s^2\) is the marginal variance, \(\lambda\) is the correlation length scale and \(\nu\) is the smoothness parameter. 
Figures \ref{fig:station prior lam = 0.5} and \ref{fig:station prior lam = 3} illustrate the effect of changing the correlation length \(\lambda\) on sample paths from a Mat\'ern kernel with \(\nu=5/2\). Notice that with a smaller length scale of \(0.5\) the samples fluctuate much more rapidly than with a larger length scale of \(3\).

\begin{figure}[htp]%
\vspace{-1.5cm}
\centering
\begin{minipage}{0.5\textwidth}
\includegraphics[width=\textwidth]{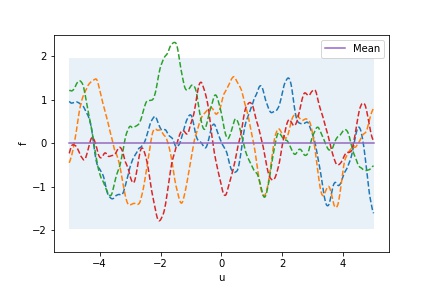}
\caption{Sample paths from \(\GP(0,k_{\M(5/2)})\), with \(\{\s^2, \lambda\}=\{1, 0.5\}\).}
\label{fig:station prior lam = 0.5}
\end{minipage}\hfill
\begin{minipage}{0.5\textwidth}
\includegraphics[width=\textwidth]{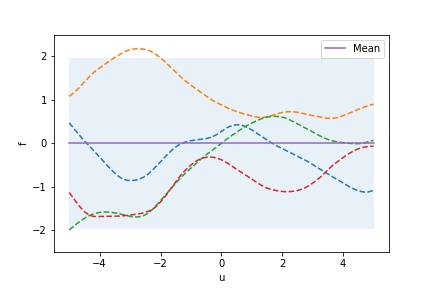}
\caption{Sample paths from \(\GP(0,k_{\M(5/2)})\), with \(\{\s^2, \lambda\}=\{1, 3\}\).}
\label{fig:station prior lam = 3}
\end{minipage}\par
\vskip\floatsep
\begin{minipage}{0.45\textwidth}
\includegraphics[width=\textwidth]{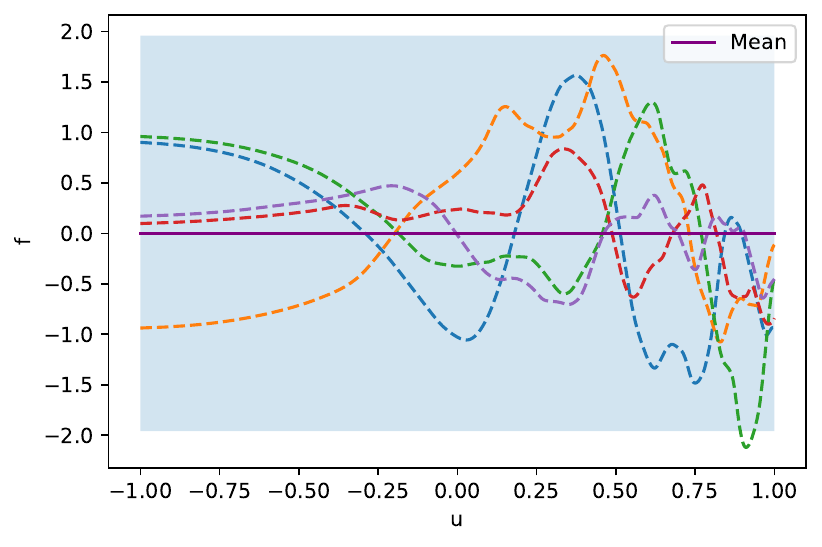}
\caption{Sample paths from \(\GP(0, k_{\warp}^{w,k_s})\), with stationary kernel \(k_{\M(5/2)}\) and \(w(u)=(u+1.5)^4\).}
\label{fig:warping_prior}
\end{minipage}\hfill
\begin{minipage}{0.45\textwidth}
\includegraphics[width=\textwidth]{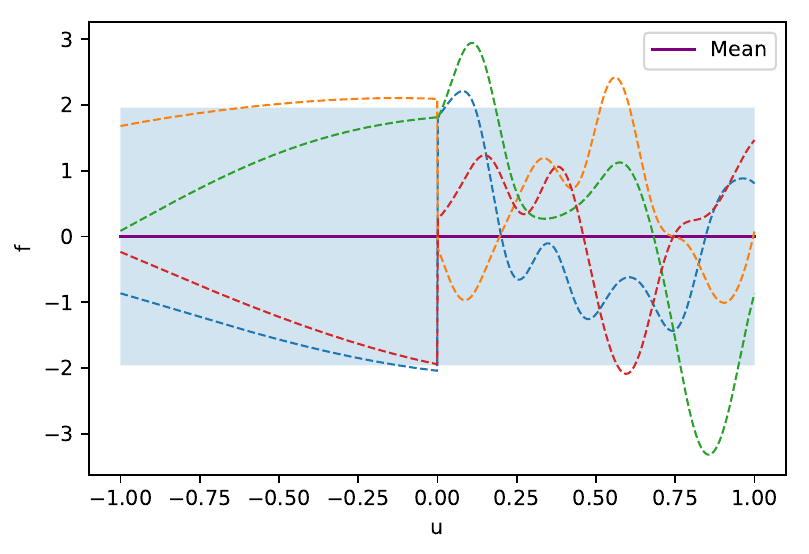}
\caption{Sample paths from \(\GP(0, k_{\mix}^{\{\s_\ell,k_{\ell}\}_{\ell=1}^2})\), with stationary kernels \(k_1=k_{\M(\infty)}\) with \(\lambda = 1\) and \(k_2=k_{\M(\infty)}\) with \(\lambda = 0.1\), and \(\s_1(u)=\mathbbm{1}_{\{u<0\}}\), \(\s_2(u)=\mathbbm{1}_{\{u\geq0\}}\).}
\label{fig:mixture_prior}
\end{minipage}\par
\vskip\floatsep
\begin{minipage}{0.45\textwidth}
\includegraphics[width=\textwidth]{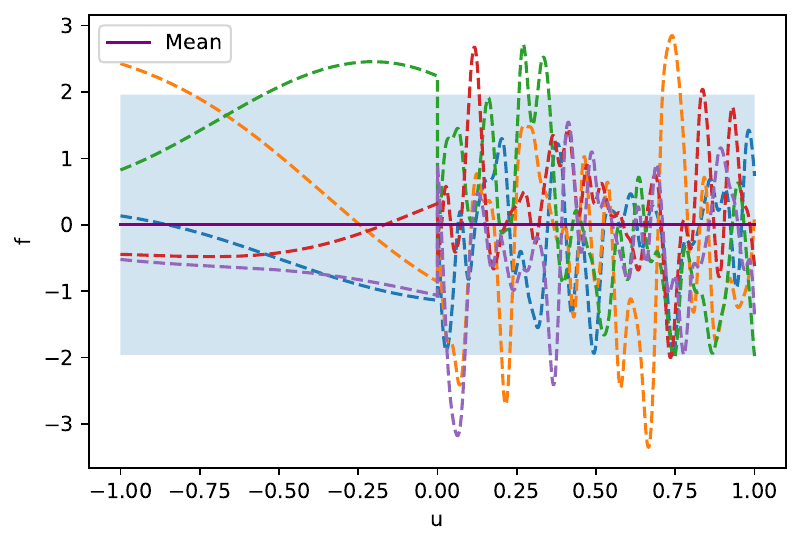}
\caption{Samples from \(\GP(0, k_{\conv}^{{\lambda_{a}},k_i})\)  with stationary Gaussian kernel and \({\lambda_{a}}(u)=\mathbbm{1}_{\{u<0\}}+\mathbbm{1}_{\{u\geq0\}}/100\).}
\label{fig:conv_prior}
\end{minipage}\hfill
\begin{minipage}{0.5\textwidth}
\includegraphics[width=\textwidth]{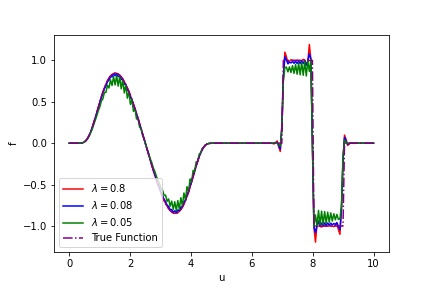}
\caption{Posterior mean with prior kernel a Mat\'ern kernel with \(\nu=1/2\), with \({\lambda} = 0.8, 0.08, 0.05\).}
\label{fig:post_mult}
\end{minipage}\par
\end{figure}

When \(\nu = p+1/2\) with \(p\in\N_0\), the Mat\'ern kernel can be written as a product of an exponential and a polynomial of order \(p\) \cite{abramowitz1965handbook}, greatly simplifying the expression. 
As \(\nu\to\infty\) the Mat\'ern covariance kernel converges to the Gaussian kernel
\[
    k_{\M(\infty)}(u,u')=\s^2\exp\left(-\frac{\|u-u'\|_2^2}{2\lambda^2}\right).
\]
Mat\'ern kernels are often the standard choice for GP regression, but can behave poorly when the observations exhibit non-stationary trends. Figure \ref{fig:post_mult} gives a simple example of  when non-stationary kernels are much better suited. The function to be approximated is continuous but not continuously differentiable, and it varies smoothly in the left hand side of the domain, whereas it has very steep gradients in the right hand side. Figure \ref{fig:post_mult} shows the posterior mean $m_{N,0}^f$, varying the length scale $\lambda$ in the prior with a Mat\'ern kernel with \(\nu=1/2\). We can see that no single length scale works for the entire function.  We desire a large length scale when the function is very flat, and a small length scale at the three jump points where the gradient is very large. Note that it is typically unknown a-priori which parts of the domain require large, or respectively small, length scales.

\subsubsection{Non-stationary covariance kernels}\label{subsubsec:nonstationary}
We now introduce three specific constructions for non-stationary kernels that can model different length scales in the observations. There are of course many more methods known to create non-stationary kernels, see for instance \cite{Remes2017, Sethian2021, Paciorek2003phd} and the references therein, but we believe the constructions covered in this paper are prototypical of approaches that efficiently model separate length scales in the data. It may be possible to analyse alternative methods using techniques similar to those presented in this paper. 

Before we go on, we introduce here precise definitions of stationarity, non-stationarity and anisotropy. A kernel \(k:\O\times\O\to\R\) is {\em stationary} if there exists a function \(k_s:\O \to\R\) such that that \(k(u,u)=k_s(u-u')\). In other words, the covariance between the GP at $u$ and $u'$ does not depend on the location of $u$ and $u'$. A kernel is thus non-stationary if no such \(k_s\) exists. We refer to a GP as (non-)stationary if it has a (non-)stationary covariance kernel. A stationary kernel is {\em isotropic} if there exists a function \(k_i:[0,\infty)\to\R\) such that that \(k(u,u)=k_i(\|u-u'\|)\). 
In an abuse of notation, we shall for simplicity often write \(k(u,u')=k(u-u')=k(\|u-u'\|)\) when referring to a stationary and/or isotopic kernel.

\subsubsection*{Warping kernel}
The warping kernel \cite{damianou2013deep,Sampson1992} uses a homeomorphism to warp the input space of a GP. 
For a warping function \(w:\O\to\O\) and a stationary kernel $k_s$, the warping kernel is defined by
\begin{equation}\label{eq:warpingkernel}
    k_{\warp}^{w,k_s}(u,u')=k_s\left({w(u)-w(u')}\right).
\end{equation}
The positive (semi-)definiteness of this kernel follows from a simple argument.
{
\begin{proposition}
If $k_s$ is positive semi-definite, then \(k_{\warp}^{w,k_s}\) is positive semi-definite. If \(w\) is injective and $k_s$ is positive definite, then \(k_{\warp}^{w,k_s}\) is positive definite.
\end{proposition}
\begin{proof}
We have
\[
    \sum_{i=1}^{n}\sum_{j=1}^nc_ic_jk_{\warp}^{w,k_s}(u^i,u^j)=\sum_{i=1}^{n}\sum_{j=1}^nc_ic_jk_s(w(u^i)-w(u^j)) = \sum_{j=1}^nc_ic_jk_s(\tilde{u}^i-\tilde{u}^j),
\]
where $\tilde{u}^i = w(u^i)$, for $i=1,\dots,n$. The claim then follows.
\end{proof}
}

See Figure \ref{fig:warping_prior} for an example of sample paths of a GP with a warping kernel. Notice that the samples fluctuate much more in the positive than the negative part of the domain. This happens because the negative part is ``pushed together" by the warping function \(w(u)=(u+1.5)^4\), while the positive part of the domain is ``pulled apart". For example, \(|w(0.75)-w(0.5)|>|0.25|>|w(-0.5)-w(-0.75)|\), and so for this non-stationary kernel the correlation between \((u,u')=(0.75, 0.5)\) will be lower than for \((u,u')=(-0.5, -0.75)\).

\subsubsection*{Kernel mixture}
The kernel mixture formulation \cite{Fuentes2001,Volodina2020} in general allows the regularity, variance and length scales to change over the domain. 
For \(L\in\N\), the non-stationary covariance kernel mixture is defined as
\begin{equation}\label{eq:kernel_mixture}
    k_{\mix}^{\{\s_\ell,k_\ell\}_{\ell=1}^L}(u,u'):=\sum_{\ell=1}^L\sigma_\ell(u)\sigma_\ell(u')k_\ell\left(u,u'\right),
\end{equation}
where the functions \(\{\s_\ell, k_\ell\}_{\ell=1}^L\) represent mixture component coefficients and stationary kernels, respectively. The kernel mixture is positive (semi-)definite under mild assumptions.
{
\begin{proposition}
Suppose \(\sigma_\ell:\O\to\R\) satisfies \(\s_\ell(u)\geq0\) for all \(u\in \O\) and \(\ell =1,\dots, L\). If $k_\ell$ is positive semi-definite for all \(\ell =1,\dots, L\), then \(k_{\mix}^{\{\s_\ell,k_\ell\}_{\ell=1}^L}\) is positive semi-definite. If there exists \(j\in\{1,\ldots, L\}\) such that \(\s_j(u)>0\) for all \(u\in \O\) and \(k_j\) is positive definite, then \(k_{\mix}^{\{\s_\ell,k_\ell\}_{\ell=1}^L}\) is positive definite.
\end{proposition}
\begin{proof}
We have
\[     \sum_{i,j=1}^{n}c_ic_jk_{\mix}^{\{\s_\ell,k_\ell\}_{\ell=1}^L}(u^i,u^j)=\sum_{i,j=1}^{n}c_ic_j\sum_{\ell=1}^L\s_\ell(u^i)\s_\ell(u^j)k_\ell(u^i,u^j)= \sum_{\ell=1}^L  \sum_{i,j=1}^{n} \tilde{c}_{i,\ell} \tilde{c}_{j,\ell} k_\ell(u^i,u^j)
\]
where \(\tilde{c}_{i,\ell} = c_i\s_\ell(u^i)\), for $i=1,\dots,n$ and $\ell=1,\dots,L$. The claim then follows.
\end{proof}
}

To model separate length scales, the kernel mixture could be constructed such that each \(k_\ell\) is a Mat\'ern kernel with a different length scale $\lambda_\ell$.
{See Figure \ref{fig:mixture_prior} for an example of a GP with a mixture kernel. In this example, \(L=2\) with two stationary kernels \(k_1\) and \(k_2\), both equal to \(k_{\M(\infty)}\). For \(k_1\), we set \(\lambda = 1\), and for \(k_2\), we set \(\lambda = 0.1\). The non-stationary functions \(\s_1\) and \(\s_2\) are defined as \(\s_1(u)=\mathbbm{1}_{\{u<0\}}\), \(\s_2(u)=\mathbbm{1}_{\{u\geq0\}}\). Notice how the fluctuation of the samples changes at \(u=0\). This behaviour is expected because, 
for \(u, u'>0\), the only kernel not taking non-zero values is the kernel with a lower length scale \(\lambda = 0.1\), and similarly for \(u,u'<0\) the only kernel taking non-zero values has a higher length scale \(\lambda = 1\).}

In the case $L=1$, the mixture kernel is no longer a mixture, but can instead be interpreted as a version of the stationary kernel $k_1$ with non-stationary marginal variance {$\sigma_1(u)^2k_1(0)$}. 

\subsubsection*{Convolution kernel}

Paciorek \cite{Paciorek2003} gives a general method to construct anisotropic versions of isotropic covariance kernels, based on the convolution of two Gaussian kernels.
The resulting convolution kernel allows non-stationarity in the form of different length scales in the domain.  
For $k_i$ an isotropic covariance function and \({\lambda_{a}}:\O\to\R_{+}\) representing the length scale, we define the non-stationary convolution covariance kernel by
\begin{equation}\label{eq:convolution_kernel}
    k_{\conv}^{{\lambda_{a}},k_i}(u,u')= \frac{2^{d/2}{\lambda_{a}}(u)^{d/4}{\lambda_{a}}(u')^{d/4}}{({\lambda_{a}}(u)+{\lambda_{a}}(u'))^{d/2}}k_i\left(\frac{\|u-u'\|_2}{\sqrt{({\lambda_{a}}(u)+{\lambda_{a}}(u'))/2)}}\right).
\end{equation}
Crucially, \({\lambda_{a}}\) is a function rather than a scalar parameter as in the Mat\'ern covariance kernel.
{
\begin{proposition}Suppose \({\lambda_{a}}:\O\to\R\) satisfies \({\lambda_{a}}(u)>0\) for all \(u\in \O\). If $k_i$ is positive (semi-) definite, then \(k_{\conv}^{{\lambda_{a}},k_i}\) is positive (semi-)definite.
\end{proposition}
\begin{proof} This follows from \cite[Theorem 1]{Paciorek2003} and \cite[Propostion 1]{dunlop2018deep}, since the matrix ${\lambda_{a}}(u) \mathrm{I}$ is positive definite for all \(u\in \O\).
\end{proof}}

Figure \ref{fig:conv_prior} gives an example for the convolution kernel being used as the kernel for a GP prior. The isotropic kernel in this example is the Gaussian kernel, and the non-stationary function \({\lambda_{a}}:[-1,1]\to\R\) is given by \({\lambda_{a}}(u)=\mathbbm{1}_{\{u<0\}}+\mathbbm{1}_{\{u\geq0\}}/100\). For negative \(u\) we see that \({\lambda_{a}}(u)=1\), meaning that there is a high correlation between any two points \(u,u'<0\), whereas for positive \(u\) we see that \({\lambda_{a}}(u)=1/100\) and so there is a low correlation between any two points \(u,u'>0\). 

\subsection{Estimating hyper-parameters}\label{subsec:Estimating hyper-paramters}
To obtain good practical performance, it is crucial to use suitable choices of the functions $w$, $\{\sigma_\ell\}_{\ell=1}^L$ and $\lambda_a$ defining the non-stationary structure of the warping, mixture and convolution kernels, respectively. Good choices are typically not known a-priori, and these functions are usually learnt from the observations $y_N$. Including the hyper-parameters appearing in the stationary and/or isotropic kernels used in the constructions \eqref{eq:warpingkernel}-\eqref{eq:convolution_kernel}, we then need to estimate a set of hyper-parameters $\theta$, containing both scalars and functions. 


In the following, we discuss two options for dealing with the hyper-parameters. Firstly, we present an empirical Bayesian approach, whereby a point-estimate of the hyper-parameters is computed, and then ``plugged-in" to the GP prior before we perform the standard prior to posterior update. Secondly, we discuss a particular hierarchical Bayesian approach, where a GP prior is placed on the functional hyper-parameters resulting in a DGP model. We are then interested in the marginal posterior on the unknown function $f$ given $y_N$.

For a fixed value of the hyper-parameters $\theta$, recall the prior \eqref{eq:GP empirical prior} on $f$ given by
\[
    f_0|\theta\sim\GP(0, k(\cdot,\cdot;\theta)),
\]
where we have now explicitly indicated the dependence on $\theta$.
In the empirical Bayesian approach, we use the available data \(y_N\) to estimate suitable values for the hyper-parameters, denoted by \(\widehat\theta_N\), and then use the conditional posterior
\begin{equation}\label{eq:gp_posterior_est}
    f_N(\widehat \theta_N) := f_0|\widehat\theta_N, y_N \sim \GP(m_{N,\delta^2}^f(\cdot;\widehat\theta_N), k_{N,\delta^2}(\cdot,\cdot;\widehat\theta_N))
\end{equation}
to recover $f$. Here, $m_{N,\delta^2}^f$ and $k_{N,\delta^2}$ are as in \eqref{eq:gp_mean}-\eqref{eq:gp_variance}, with the dependence on $\theta$ made explicit. Note that this distribution is given in closed form, and in particular, its mean and covariance functions are known analytically. This approach is discussed in more detail in section \ref{subsec:eagr_empirical}.

In a hierarchical approach, we instead place a prior distribution on $\theta$. A particular example is a DGP, which is constructed using a sequence of random fields $\{f^n\}_{n\in\N_0}$ that are conditionally Gaussian \cite{damianou2013deep,dunlop2018deep}:
\begin{align*}
\nonumber f^0&\sim\mathcal{GP}(0, k_0(u,u')),\\
f^{n+1}|f^n&\sim\mathcal{GP}(0, k_n(u,u';f^n)).
\end{align*}
We denote by $f^{D}$ a DGP of depth \(D+1\). 
In the context of constructing priors for regression, we can then use the DGP $f^{D}$ as a prior on $f$. The penultimate layer $f^{D-1}$ will be used to construct the non-stationarity in $f^{D}$, for example by defining the warping function {\(w=f^{D-1}\) or the length scale function \({\lambda_{a}} = F(f^{D-1})\) for some positive function \(F:\O\to\R_+\)}. The remaining layers define the prior on $f^{D-1}$. In contrast to standard GP regression, the posterior $f^{D}|y_N$ can no longer be found explicitly, and one typically uses Markov chain Monte Carlo (MCMC) methods for sampling \cite{dunlop2018deep,monterrubio2020posterior}. DGPs are discussed in more detail in section \ref{sec:EADGPR}. 

\section{Reproducing kernel Hilbert spaces}\label{sec:nativespaces}
A crucial ingredient in the analysis of GP and DGP regression are reproducing kernel Hilbert spaces (RKHSs) \cite{Aronszajn1950, paulsen_raghupathi_2016}, also referred to as native spaces \cite{wendland_2004} or Cameron-Martin spaces \cite{da2014stochastic} (in the context of Gaussian measures). This section is devoted to studying the RKHSs corresponding to the non-stationary kernels $k^{w,k_s}_\mathrm{warp}$ and $k_{\mix}^{\{\s_\ell,k_\ell\}_{\ell=1}^L}$ introduced in section \ref{subsubsec:nonstationary}. In particular, we show that these are norm-equivalent to Sobolev spaces in special cases. We were unable to find an explicit expression for the RKHS of the convolution kernel $k^{{\lambda_{a}},k_i}_\mathrm{conv}$.

Let us start with the definition of an RKHS.
\begin{definition}\label{def:RKHS}
Let \(\mathcal{H}(\O)\) be a real Hilbert space of functions \(f:\O\to\R\). \(\mathcal{H}(\O)\) is the RKHS of a kernel \(k:\O\times\O\to\R\) if
\begin{itemize}
    \item[(i)] \(k(\cdot, u)\in\mathcal{H}(\O)\) for all \(u\in\O\),
    \item[(ii)] \(f(u)=\left(f, k(\cdot,u)\right)_{\mathcal{H}}(\O)\) for all \(f\in\mathcal{H}(\O)\) and all \(u\in\O\).
\end{itemize}
\end{definition}

In the remainder of this work, we shall use $\mathcal{H}_k(\O)$ to denote the RKHS corresponding to a kernel $k$. It is not always easy to find the RKHS of a specific kernel in closed form. A notable exception is the family of Mat\'ern covariance kernels defined in section \ref{subsubsec:mat}, for which the RKHS can be characterised by its norm-equivalence with a Sobolev space \cite{wendland_2004,Teckentrup2019}. 
\begin{proposition}[{\cite[Lemma 3.4]{Teckentrup2019}}]\label{prop:sobolev equivnorms}
We have \(\H_{k_{\M(\nu)}}(\Omega) \cong H^{\nu+d/2}(\Omega)\). In particular,
\[
     C_\low(k_{\M(\nu)})\|g\|_{\H_{k_{\M(\nu)}}(\O)} \leq \|g\|_{H^{\nu+d/2}(\O)}\leq C_\up(k_{\M(\nu)}) \|g\|_{\H_{k_{\M(\nu)}}(\O)},
\]
for all \(g\in H^{\nu+d/2}(\Omega)\), where
{\begin{align*}
    C_\low(k_{\M(\nu)})&=\frac{\s\Gamma(\nu+d/2)^{1/2}\lambda^{d/2}}{\pi^{d/4}\Gamma(\nu)^{1/2}}\min\{1,\lambda^{-1}\}, \\
    C_\up(k_{\M(\nu)})&=\frac{\s\Gamma(\nu+d/2)^{1/2}\lambda^{d/2}}{\pi^{d/4}\Gamma(\nu)^{1/2}}\max\{1,\lambda^{-1}\}.
\end{align*}
}
\end{proposition}

Note that the only hyper-parameter that affects the order of the Sobolev space is $\nu$, whereas  $\lambda$ and $\sigma^2$ also affect the constants in the norm equivalence. 

\subsection{Operations on kernels}
For the warping kernel $k^{w,k_s}_\mathrm{warp}$ and the kernel mixture $k_{\mix}^{\{\s_\ell,k_\ell\}_{\ell=1}^L}$, we can use algebraic operations on kernels to construct the corresponding RKHSs explicitly. 
We recall the following results from  \cite{Aronszajn1950,paulsen_raghupathi_2016}. 
\begin{proposition}\label{prop:Compo}
{\cite[Theorem 5.7]{paulsen_raghupathi_2016}}
Let \(k:\O \times \O\to\R\) and \(w:\O\to\O\) be given. Then the RKHS corresponding to the kernel
\(k^w(u,u')=k\left({w(u),w(u')}\right)\)
is given by
\begin{equation*}
     \H_{k^{w}}(\Omega)=\left\{h\circ w: h\in \H_{k}(\Omega)\right\}
\end{equation*}
equipped with the norm $\|u\|_{\H_{k^{w}}}=\min\{\|h\|_{\H_{k}}:u = h\circ w\}$. 
\end{proposition}

\begin{proposition}\label{prop:NonStatCov}
{\cite[Proposition 5.20]{paulsen_raghupathi_2016}}
Let \(k:\O \times \O\to\R\) and \(\s:\O\to\R\) be given, and define
\[
    \H_{k}^0(\Omega):=\{h\in \H_{k}(\Omega):\sigma h \equiv 0\}.
\]
Then the RKHS corresponding to the kernel $k^{\sigma}(u,u')=\sigma(u) \sigma(u')k(u,u')$ is given by 
\begin{equation*}
     \H_{k^{\s}}(\O)=\left\{\sigma h:h\in{\left(\H_{k}^0(\O)\right)}^\perp\right\},
\end{equation*}
equipped with the norm
\(  \|\s h\|_{\H_{k^{\s}}} = \| h\|_{\H_{k}}\).
\end{proposition}

\begin{remark}\label{rem:NonStatConv}
    Note that if we in Proposition \ref{prop:NonStatCov} additionally assume \(\s(u)> 0\) for all $u \in \O$, then \(\s h\equiv0\) if and only if \(h\equiv0\) and hence \(\H_{k^{\s}}(\O)=\{\s h:h\in\H_{k}(\O)\}\).
\end{remark}


%
\begin{proposition}\label{prop:addition_kernels}
{\cite[Theorem 5.4]{paulsen_raghupathi_2016}}
Let \(k_1:\O \times \O\to\R\) and \(k_2:\O \times \O\to\R\) be given. Then the RKHS corresponding to the kernel \(k=k_1+k_2\) is given by
\[
    \H_k(\O):=\{u_1+u_2:u_i\in \H_{k_i}(\O)\, i=1,2\}
\]
equipped with the norm  
\[
    \|u\|^2_{\H_{k}(\O)}=\min\{\|u_1\|^2_{\H_{k_1}(\O)}+\|u_2\|^2_{\H_{k_2}(\O)}:u =u_1+u_2, u_i\in \H_{k_i}(\O), i=1,2\}.
\]
\end{proposition}

\subsection{RKHSs of non-stationary kernels as Sobolev spaces} \label{sec:nativespaces_nonstationary}
The results presented in the previous section allow us to derive the RKHSs of the non-stationary kernels $k^{w,k_s}_\mathrm{warp}$ and $k_{\mix}^{\{\s_\ell,k_\ell\}_{\ell=1}^L}$ explicitly in terms of the RKHS of the stationary/isotropic kernel used to construct them and the hyper-parameters $w$ or $\{\s_\ell\}_{\ell=1}^L$, respectively. By Proposition \ref{prop:sobolev equivnorms}, we know that the native space associated with the Mat\'ern kernel can be characterised by its equivalence with a Sobolev space, and we use this fact to prove the following equivalences.


\begin{theorem}[RKHS for warping kernel]\label{lemma:sob equiv warping} 
Let $\Omega \subset \R$.
    Suppose $\{w,k_s\}$ are such that $(i)$ \(k_s=k_{\M(\nu)}\) with \(\nu=p+1/2\) for some \(p\in\N_0\), $(ii)$ \(w\in C^{\beta}(\O)\) 
    for \(\beta=\nu+1/2\), and $(iii)$  \(|w{'}| \geq c>0\) and \(|\inv(w){'}|\geq c'>0\) a.e. in $\O$. 
 Then \(\H_{k_{\warp}^{w,k_s}}(\O) \cong H^{\beta}(\O)\). In particular,
\[
     C_\low(k_{\warp}^{w,k_s})\|g\|_{\H_{k_{\warp}^{w,k_s}(\O)}} \leq \|g\|_{H^{\beta}(\O)}\leq C_\up(k_{\warp}^{w,k_s}) \|g\|_{\H_{k_{\warp}^{w,k_s}(\O)}}
\]
for all \(g\in H^{\beta}(\Omega)\), where {for a constant \(C(\beta)\) depending only \(\beta\) we have 
\begin{align*}
  C_\low({k_{\warp}^{w,k_s}}) &=  C_\low(k_s) C(\beta)^{-1} \sqrt{c'} \max\left\{1,\|\inv(w)\|_{C^{\beta}(\O)}^{\beta} \right\}^{-1} \\
    C_\up({k_{\warp}^{w,k_s}}) &= C_\up(k_s) C(\beta)\frac{1}{\sqrt{c}}\max\left\{1,\|w\|_{C^{\beta}(\O)}^{\beta} \right\}.
\end{align*}
}
\end{theorem}

{
\begin{theorem}[RKHS for mixture kernel]\label{lemma: sob equiv mixture}
    Suppose $\{\s_\ell,k_\ell\}_{\ell=1}^L$ are such that $(i)$ \(k_\ell=k_{\M(\nu)}\), $(ii)$ \(\s_\ell\in H^{\beta}(\O)\) for \(\beta = \nu + d/2\)  for $\ell=1,\dots,L$, and $(iii)$ \(1/\s_{j}\in H^{\beta}(\O)\) for some \(j\in\{1,\ldots,L\}\). Then \(\H_{k_{\mix}^{{\{\s_\ell,k_\ell\}}_{\ell=1}^L}}(\O) \cong H^{\beta}(\O)\). In particular,
\[
     C_\low(k_{\mix}^{{\{\s_\ell,k_\ell\}}_{\ell=1}^L}) \|g\|_{\H_{k_{\mix}^{{\{\s_\ell,k_\ell\}}_{\ell=1}^L}}(\O)} \leq \|g\|_{H^{\beta}(\O)}\leq C_\up(k_{\mix}^{{\{\s_\ell,k_\ell\}}_{\ell=1}^L}) \|g\|_{\H_{k_{\mix}^{{\{\s_\ell,k_\ell\}}_{\ell=1}^L}}(\O)}
\]
for all \(g\in H^{\beta}(\Omega)\), where
\begin{align*}
     C_\low(k_{\mix}^{{\{\s_\ell,k_\ell\}}_{\ell=1}^L})&= C_{\ref{prop:banachalg}}^{-1} C_\low(k_j)   \|1/\sigma_j\|_{H^{\beta}}^{-1},\\
    C_\up(k_{\mix}^{{\{\s_\ell,k_\ell\}}_{\ell=1}^L})&= \sqrt{2} C_{\ref{prop:banachalg}} \max_{{1 \leq \ell \leq L}} \{C_\up(k_\ell) \|\s_\ell\|_{H^{\beta}(\O)} \}.
\end{align*}
\end{theorem}
}
Since this work focuses on modelling variable length scales, we have assumed that \(\nu_\ell \equiv \nu\) for \(\ell = 1, \ldots, L\) in Theorem \ref{lemma: sob equiv mixture}, and that the kernels $k_\ell$ are characterised by different length scales $\lambda_\ell$. However, Theorem \ref{lemma: sob equiv mixture} can be extended to the case of non-constant \(\nu_\ell\) without much effort,  see \cite{MoriartyOsborne2024} for more details.

\begin{remark}\label{rmk:on ess bnd w}
    In Theorem \ref{lemma:sob equiv warping}, it is possible to replace the assumption that \(w\in C^{\beta}(\O)\) with \(w\in W^{\beta,\infty}(\O)\). If additionally \(\inv(w)\in W^{\beta,\infty}(\O)\), the conclusions of the Theorem still hold, with the obvious replacement of norms.
\end{remark}

{\label{rmk: on suff cond for 1/sig is sob}
\begin{remark}
    In Theorem \ref{lemma: sob equiv mixture}, it might be possible to change the assumption that \(\s_j\geq\s_{\min} > 0\) (and hence $1/\s_j \in H^\beta(\O)$ by Lemma \ref{lemma:reciprocal has derivatives}) to assuming that for each \(u\in\O\) there exists \(\ell\in\{1,\ldots,L\}\) such that \(\s_\ell(u)\geq\s_{\min}\). This has been not been investigated for brevity. Note however, that we would still require that each \(\s_\ell\in H^{\beta}(\O)\), and thus, that each \(\s_{\ell}\) is continuous. 
    This would preclude examples where \(\s_{\ell}\) are piecewise constant. 
\end{remark}}

Theorems \ref{lemma:sob equiv warping} and \ref{lemma: sob equiv mixture} provide assumptions on the non-stationary hyper-parameters $w$ and $\{\s_\ell\}_{\ell=1}^L$ such that the RKHS of the non-stationary kernel is identical (up to norm-equivalence) to the RKHS of the Mat\'ern kernel used to construct it. 
Choosing $w$ and $\{\s_\ell\}_{\ell=1}^L$ more regular will not change the RKHS, since we are limited by the regularity of $k_s$ and $\{k_\ell\}_{\ell=1}^L$. Choosing $w$ and $\{\s_\ell\}_{\ell=1}^L$ less regular on the other hand is unfortunately not within the framework of this theory, and the RKHS is no longer a Sobolev space. It is a subspace of a Sobolev space, dependent on $w$ or $\{\s_\ell\}_{\ell=1}^L$, respectively. To see this suppose, for instance, that \(L=1\),  \(k_{1}=k_{\M(\nu_1)}\) for \(\beta=\nu_1+d/2\) such that $\H_{k_1}(\O) \cong H^\beta(\O)$, and \(\s_1\in H^{\a}(\O)\) with \(\a<\beta\). By Proposition \ref{prop:NonStatCov}, we know that every element of $\H_{k_{\mix}^{\s_1,k_1}}$ can be written in the form $\s_1 h$ for $h \in H^\beta(\O)$, which is not true for all elements of $H^\alpha(\O)$.

It is possible to extend Theorem \ref{lemma:sob equiv warping} to the case of general $\nu >0 $ and $\O \subset \R^d$ following the same proof technique, but this becomes very technical and has been omitted for brevity. In the statement of Theorem \ref{lemma:sob equiv warping}, we have further implicitly used the following result, which allows us to conclude that \(\inv(w)\in C^{\beta}(\O)\) under the assumptions of Theorem \ref{lemma:sob equiv warping}.

\begin{lemma}\label{lemma:inverse has derivatives}
Suppose $\O \subseteq \mathbb{R}$, \(w\in C^{\beta}(\O)\) for \(\beta\in\N\), and \(w'(u)\geq w'_{\min}\) for all \(u\in\O\), for some \(1\geq w'_{\min}>0\). Then \(\inv(w)\in C^{\beta}(\O)\) and
\[
    \|\inv(w)\|_{C^{\beta}(\O)}\leq \left(1+\frac{1}{(w'_{\min})^{2\beta}}\right)(\beta+1)(2\beta-1)!B_{2\beta-1}\left(1+\left\|w\right\|^{\beta-1}_{C^{\beta}(\O)}\right),
\]
where \(B_{2\beta-1}\) is the Bell number, the count of the possible partitions of the set \(\{1,\ldots,2\beta-1\}\).
\end{lemma}
\begin{proof}
    By the inverse function rule (see, e.g. \cite{Warren2002}), we see that
    \begin{equation}\label{eq:inverse func rule}
        \inv(w)'(u)=\frac{1}{w'(\inv(w)(u))}.
    \end{equation}
    Then \(w'\in C^{0}(\O)\),  \(\inv(w)\in C^{0}(\O)\) and $w'(u)\geq w'_{\min}$ together imply that \(\inv(w) \in C^{1}(\O)\). 
    Using this argument inductively we see that \(\inv(w)\in C^{\beta}(\O)\).

    For the second claim, note that an explicit expression the \(\a\)-th derivative of the inverse function is given by
    \[
        \inv(w)^{(\a)}(u)=\sum_{i=0}^{\a-1}\frac{1}{(w'(\inv(w)(u)))^{\a+i}}(-1)^i\frac{1}{i!}\sum_{\substack{b_1+\dots+b_{i}=\a+i-1\\b_j\geq2}}\binom{\a+i-1}{b_1,\ldots,b_{\a-1}}\prod_{1\leq i\leq \a-1}w^{(b_i)}(u),
    \]
    for \(u\in\O\),
    where \(b_1,\ldots,b_i\) are the sizes of the blocks of all partitions of size \(i\) of the set \(\{1,\ldots,\a+i-1\}\) \cite{Warren2002}. First notice that maximum size of any individual block over all \(i=1,\ldots, \a-1\) and over all \(i\)-tuples \((b_1,b_2,\ldots,b_i)\) is \(\a\). 
    This occurs when \(i=1\) so that \(b_1=\a-i+1=\a\). Hence we can see that for any \(b_i\) in any such \(i\)-tuple we have that
    \begin{align}\label{eq:bound on f warping DGP}
        \|w^{(b_i)}\|_{C^{0}(\O)} \leq \sup_{\tilde\a\leq\a}\|w^{(\tilde\a)}\|_{C^{0}(\O)}=\left\|w\right\|_{C^{\a}(\O)}.
    \end{align}
    Using the fact that \(w'\) is bounded away from zero, the triangle inequality and \eqref{eq:bound on f warping DGP} we see that
    \begin{align*}
    \|\inv(w)^{(\a)}\|_{C^{0}(\O)}
    &\leq \left(1+\frac{1}{(w'_{\min})^{2\a}}\right)(2\a-1)!\sum_{\substack{b_1+\dots+b_{i}=\a+i-1\\b_j\geq2}}\prod_{1\leq i\leq \a-1}\left\|w^{(b_i)}\right\|_{C^{0}(\O)}\\
    &\leq \left(1+\frac{1}{(w'_{\min})^{2\a}}\right)(2\a-1)!\left\|w\right\|^{\a-1}_{C^{\a}(\O)}\sum_{\substack{b_1+\dots+b_{i}=\a+i-1\\b_j\geq2}}1\\
    &\leq \left(1+\frac{1}{(w'_{\min})^{2\a}}\right)(2\a-1)!B_{2\a-1}(1+\left\|w\right\|^{\a-1}_{C^{\a}(\O)}),
    \end{align*}
    where \(B_{2\a-1}\) is the Bell number which gives the count of the possible partitions of the set \(\{1,\ldots,2\a-1\}\).
    Since \(w'_{\min}\leq 1\) the following then achieves the desired result,
    \begin{align*}
        \|\inv(w)\|_{C^{\beta}(\O)}&=\sum_{\a=0}^{\beta}\|\inv(w)^{(\a)}\|_{C^{0}(\O)}
        \leq\left(1+\frac{1}{(w'_{\min})^{2\beta}}\right)(\beta+1)(2\beta-1)!B_{2\beta-1}\left(1+\left\|w\right\|^{\beta-1}_{C^{\beta}(\O)}\right).
    \end{align*}
\end{proof}


{In Theorem \ref{lemma: sob equiv mixture}, to ensure that \(1/\sigma_j\in H^{\beta}(\O)\) it is sufficient that
\(\s_j\) is uniformly bounded away from zero, \(\s_j\geq\s_{min}>0\).}
 The following Lemma proves this claim along with two other results that will be used in section \ref{sec:EADGPR}. In the context of the mixture kernel $k_{\mix}^{{\{\s_\ell,k_\ell\}}_{\ell=1}^L}$, where $\sigma_\ell$ represents a mixture coefficient, the assumption that $\s_{\ell}$ is non-negative is natural.

{
\begin{lemma}\label{lemma:reciprocal has derivatives}
    Suppose \(\O\subset\R^d\). If \(\s\in H^{\beta}(\O)\) for an integer \(\beta>d/2\) and \(F\in C^{\infty}(\R)\) such that \(F\geq F_{min}>0\) there exists constants \(C_{\ref{lemma:reciprocal has derivatives}}'<\infty\) and \(C_{\ref{lemma:reciprocal has derivatives}}''=C_{\ref{lemma:reciprocal has derivatives}}''(F_{\min})\) such that
    \begin{align*}
       \|F(\s)\|_{H^{\beta}(\O)}&\leq C_{\ref{lemma:reciprocal has derivatives}}'\left(1+\|\s\|^{\beta}_{H^{\beta}(\O)}\right),\\ 
       \|1/F(\s)\|_{H^{\beta}(\O)}&\leq C_{\ref{lemma:reciprocal has derivatives}}''\left(1+\|\s\|^{\beta}_{H^{\beta}(\O)}\right). 
    \end{align*}
    In particular, suppose \(\s\geq \s_{\min}>0\), then there exists a constant \(C_{\ref{lemma:reciprocal has derivatives}} = C_{\ref{lemma:reciprocal has derivatives}}(\s_{\min})<\infty\) such that
    \begin{align*}
    \|1/\s\|_{H^{\beta}(\O)}\leq C_{\ref{lemma:reciprocal has derivatives}}\left(1+\|\s\|^{\beta}_{H^{\beta}(\O)}\right).
    \end{align*}
\end{lemma}
\begin{proof}
    Since \(F\in C^{\infty}(\R)\) and that for all \(k\geq 1\) we have that \(\|F^{(k)}\|_{C^0{([F_\mathrm{min},\infty))}}<\infty\) the first result follows from
    {\cite[Lemma 29]{Nickl2020}} by restricting the domain to \(\O\).
    To show the second claim notice that \(1/F(\s)\in C^{\infty}([F_{\min},\infty])\) as a function of \(F(\s)\), and further \(\|(1/F(\s))^{(k)}\|_{C^0{([F_{\min},\infty))}}<\infty\) for all \(k\geq 1\). The second result follows again from
    {\cite[Lemma 29]{Nickl2020}} by restricting the domain to \([F_{\min},\infty)\).
     The third claim follows directly from the first two cases.
\end{proof}}

The remainder of this section is devoted to the proofs of Theorems \ref{lemma:sob equiv warping} and \ref{lemma: sob equiv mixture}.
The following proposition will be used in the proof of Theorem \ref{lemma: sob equiv mixture}.
\begin{proposition}{\cite[Theorem 4.39]{RAAdams_JJFFournier_2003}}\label{prop:banachalg}
The Sobolev space \(H^{\beta}(\O)\), \(\O\subset\R^d\), is a Banach algebra whenever \(2\beta>d\). In particular, there exists a constant $C_{\ref{prop:banachalg}}$ such that for all $g,h \in H^\beta(\O)$ we have
\[
\|gh\|_{H^\beta(\O)} \leq C_{\ref{prop:banachalg}}\|g\|_{H^\beta(\O)} \|h\|_{H^\beta(\O)}.
\]
\end{proposition}

\begin{proof}[Of Theorem \ref{lemma:sob equiv warping}]
We provide the proof for \(\beta=2\), which corresponds to \(k_s=k_{\M(3/2)}\). The general case can be proved similarly using Fa\`a di Bruno's formula \cite{jones_1951}. 

We first show some initial results that will help us later.
We use the change of variables \(w(x)=y\), so that \(dx=\frac{dx}{dy}dy = (w'(\inv(w)(y)))^{-1} dy\). Since \(w(\O)=\O\) (which follows from $\inv(w)$ being well-defined on $\Omega$) and \(|w'|\geq c>0\) a.e., we have
\begin{align*}
     \|u\circ w\|_{L^2(\O)}^2&=\int_{w(\Omega)}(u(y))^2\frac{dy}{w'(\inv(w)(y))}
     \leq \frac{1}{c}\|u\|_{L^2(\O)}^2, \\
    \|(u'\circ w)w'\|_{L^2(\O)}^2
    &=\int_{w(\Omega)}(u'(y)w'(\inv(w)(y)))^2\frac{dy}{w'(\inv(w)(y))}
    \leq \frac{1}{c} \|w\|^2_{W^{1,\infty}(\O)}\|u'\|_{L^2(\Omega)}^2, \\
    \|(u''\circ w)(w')^2\|_{L^2(\O)}^2 &\leq \frac{1}{c}\|w\|^4_{W^{1,\infty}(\O)}\|u''\|_{L^2(\Omega)}^2, \\
    \|(u'\circ w)w''\|_{L^2(\O)}^2 
    &=\int_{w(\Omega)}((u'(y))w''(x))^2\frac{dy}{w'(\inv(w)(y))}
    \leq \frac{1}{c}{\|w\|^2_{W^{2,\infty}(\O)}}\|u'\|_{L^2(\Omega)}^2.
\end{align*}

First, we show that \(\H_{k_{\warp}^{w,k_s}}(\Omega)\hookrightarrow H^2(\Omega)\), where $\hookrightarrow$ denotes a continuous embedding.
By Proposition \ref{prop:Compo}, for arbitrary \(g\in \H_{k_{\warp}^{w,k_s}}(\O)\) there exists \(u\in\H_{k_{s}}(\O) \cong H^2{(\Omega)} \) such that \(g = u\circ w\) and $\|g\|_{\H_{k_{\warp}^{w,k_s}}(\Omega)} = \|u\|_{\H_{k_s}(\O)}$. 
Using the chain rule, the bounds proven above, and Proposition \ref{prop:sobolev equivnorms}, we see that
\begin{align*}
    \|g\|^2_{H^2(\Omega)}&=\|u\circ w\|_{H^2(\O)}^2 \\
    &\leq \frac{1}{c}\max\left\{1,\|w\|_{W^{2,\infty}(\O)}^{2} \right\}^2 \|u\|_{H^2(\O)}^2 \\
    &\leq C_\up(k_s)^2 \frac{1}{c}\max\left\{1,\|w\|_{W^{2,\infty}(\O)}^{2} \right\}^2 \|u\|_{\H_{k_s}(\O)}^2 \\
    &= C_\up(k_s)^2 \frac{1}{c}\max\left\{1,\|w\|_{W^{2,\infty}(\O)}^{2} \right\}^2 \|g\|_{\H_{k_{\warp}^{w,k_s}}(\Omega)}^2. 
\end{align*}

Conversely, to show that that \(H^2(\Omega)\hookrightarrow \H_{k_{\warp}^{w,k_s}}(\Omega)\) we simply replace \(w\) with \(\inv(w)\) in the bounds. 
By Proposition \ref{prop:Compo} and since $w$ is invertible, we then have
\begin{align*}
    \|u\|^2_{\H_{k_{\warp}^{w,k_s}}(\O)}&=\min\left\{\|\tilde u\|^2_{H_{k_{s}}(\O)}:u=\tilde u\circ w\right\}\\
    &= \| u\circ \inv(w)\|^2_{\H_{k_{s}}(\O)}\\
    &\leq (C_\low(k_{s}))^{-2}\| u\circ \inv(w)\|_{H^2(\O)}\\
    &\leq (C_\low(k_s))^{-2} \frac{1}{c'}\max\left\{1,\|\inv(w)\|_{W^{2,\infty}(\O)}^{2} \right\}^2 \|u\|_{H^2(\Omega)}^2.
\end{align*}
This finishes the proof, with
\begin{align*}
  C_\low({k_{\warp}^{w,k_s}}) &= C_\low(k_s) \sqrt{c'} \max\left\{1,\|\inv(w)\|_{W^{2,\infty}(\O)}^{2} \right\}^{-1}, \\
    C_\up({k_{\warp}^{w,k_s}}) &= C_\up(k_s) \frac{1}{\sqrt{c}}\max\left\{1,\|w\|_{W^{2,\infty}(\O)}^{2} \right\}. 
\end{align*}

\end{proof}

\begin{proof}[Of Theorem \ref{lemma: sob equiv mixture}]
We first consider the case \(L=1\). 
To show \(\H_{k_{\mix}^{\s_1,k_1}}(\Omega)\hookrightarrow H^\beta(\Omega)\), let \(g\in \H_{k_{\mix}^{\s_1,k_1}}(\Omega)\) be arbitrary and note that by Proposition \ref{prop:NonStatCov} 
there exists \(u\in H^\beta(\Omega)\) such that \(g=\s_1 u\) and $\|g\|_{\H_{k_{\mix}^{\s_1,k_{1}}}(\Omega)} = \|u\|_{\H_{k_1}(\Omega)}$. Propositions \ref{prop:banachalg} and \ref{prop:sobolev equivnorms} then give
\begin{align*}
    \|g\|_{H^\beta(\Omega)}
    &\leq C_{\ref{prop:banachalg}} \|\s_1\|_{H^\beta(\Omega)}\|u\|_{H^\beta(\Omega)}\\
    &\leq C_{\ref{prop:banachalg}} C_\up(k_1) \|\s_1\|_{H^\beta(\Omega)}\|u\|_{\H_{k_1}(\Omega)}\\
    &= C_{\ref{prop:banachalg}}  C_\up(k_1) \|\s_1\|_{H^\beta(\Omega)}\|g\|_{\H_{k_{\mix}^{\s_1,k_{1}}}(\Omega)}.
\end{align*}
To conversely show that \( H^\beta(\Omega)\hookrightarrow{\H_{k_{\mix}^{\s_1,k_{1}}}(\Omega)}\), let \(u\in H^\beta(\Omega)\) be arbitrary. 
Propositions \ref{prop:NonStatCov}, \ref{prop:sobolev equivnorms}, and \ref{prop:banachalg} then give
\begin{align*}
    \|{u}\|_{\H_{k_{\mix}^{\s_1,k_{1}}}(\Omega)} 
&= \|u/\s_1\|_{\H_{k_1}(\Omega)} \\
                &\leq (C_\low(k_1))^{-1}\|{u}/\s_1\|_{H^{\beta}(\Omega)}\\
                        &\leq C_{\ref{prop:banachalg}} (C_\low(k_1))^{-1} \|1/\s_1\|_{H^\beta(\Omega)}\|{u}\|_{H^{\beta}(\Omega)}.
\end{align*}
This finishes the proof for $L=1$. {For all cases where \(L\geq 2\), we  assume without loss of generality that \(j=1\), meaning that \(1/\s_1\in H^{\beta}(\O)\). } For the case \(L=2\), 
  since \(\nu_1=\nu_2\) we have that 
\[
    H^{\nu_2+d/2}(\O)\cong H^{\nu_1+d/2}(\O) \cong H^{\beta}(\O).
\]
We first show that \(\H_{k^{\{\s_\ell,k_\ell\}_{\ell=1}^2}_{\mix}}(\O)\hookrightarrow H^{\beta}(\O)\). Let \(g\in\H_{k^{\{\s_\ell,k_\ell\}_{\ell=1}^2}_{\mix}(\O)}\) be arbitrary. By Propositions \ref{prop:addition_kernels} and \ref{prop:NonStatCov}, we can write \(g = \sigma_1g_1+\sigma_2g_2\) for \(g_i\in \H_{k_i}(\Omega) \) such that $\|g\|^2_{\H_{k^{\{\s_\ell,k_\ell\}_{\ell=1}^2}_{\mix}}(\O)} = \|g_1\|^2_{\H_{k_1}(\Omega)}+\|g_2\|_{\H_{k_2}(\Omega)}^2$.
Then the triangle inequality in $H^\beta(\O)$, together with Propositions \ref{prop:banachalg} and \ref{prop:NonStatCov}, gives 
\begin{align*}
\|g\|_{H^{\beta}(\O)}^2&=\|\sigma_1g_1+\sigma_2g_2\|_{H^{\beta}(\O)}^2\\
    &\leq 2 C_{\ref{prop:banachalg}}^2 \left(\|\s_1\|^2_{H^{\beta}(\O)} \|g_1\|^2_{H^{\beta}(\O)}+\|\s_2\|^2_{H^{\beta}(\O)} \|g_2\|_{H^{\beta}(\O)}^2\right)\\
     &\leq 2 C_{\ref{prop:banachalg}}^2 \max \{C_\up(k_1)^2 \|\s_1\|^2_{H^{\beta}(\O)},C_\up(k_2)^2 \|\s_2\|^2_{H^{\beta}(\O)}\} \left(\|g_1\|^2_{\H_{k_1}(\O)}+\|g_2\|_{\H_{k_2}(\O)}^2\right)\\
     &=2 C_{\ref{prop:banachalg}}^2 \max \{C_\up(k_1)^2 \|\s_1\|^2_{H^{\beta}(\O)},C_\up(k_2)^2 \|\s_2\|^2_{H^{\beta}(\O)}\}\|g\|^2_{\H_{k^{\{\s_\ell,k_\ell\}_{\ell=1}^2}_{\mix}}(\O)}.
\end{align*}

To conversely show that \(H^{\beta}(\O)\hookrightarrow \H_{k^{\{\s_\ell,k_\ell\}_{\ell=1}^2}_{\mix}}(\O)\), let \(g\in H^{\beta}(\O)\) be arbitrary. 
 Then Propositions \ref{prop:addition_kernels}, \ref{prop:NonStatCov}, \ref{prop:sobolev equivnorms} and \ref{prop:banachalg} give 
 {\begin{align*}
     \|g\|^2_{\H_{k^{\{\s_\ell,k_\ell\}_{\ell=1}^2}_{\mix}}(\O)}&=\min\{\|g_1\|^2_{\H_{k_1}(\O)}+\|g_2\|^2_{\H_{k_2}(\O)}:g =g_1+g_2, g_i\in\H_{k_i}(\O), i=1,2\}\\
     &\leq \|g\|_{\H_{k_1}(\O)}^2 \\
     &\leq (C_\low(k_1))^{-2} \|g/\sigma_1\|^2_{H^{\beta}(\O)} \\
     &\leq C_{\ref{prop:banachalg}}^2 (C_\low(k_1))^{-2} \|1/\sigma_1\|^2_{H^{\beta}} \|g\|^2_{H^{\beta}(\O)}.
 \end{align*}}

This finishes the proof for $L=2$. For \(L\geq3\), we can use an inductive argument. 
Suppose we have already shown that \(\H_{k^{\{\s_\ell, k_\ell\}_{\ell=1}^{L-1}}_{\mix}}(\O) \cong H^{\beta}(\O)\). Then notice we can write
\begin{align*}
    {k^{\{\s_\ell, k_\ell\}_{\ell=1}^{L}}_{\mix}}(u,u') 
    &=\tilde{\sigma}(u)\tilde{\sigma}(u'){k^{\{\s_\ell, k_\ell\}_{\ell=1}^{L-1}}_{\mix}}(u,u') + \sigma_L(u)\s_L(u')k_L(u,u'),
\end{align*}
with \(\tilde{\sigma}\equiv 1\). We then use the argument shown in this proof for when \(L=2\) to show that \(\H_{k^{\{\s_\ell, k_\ell\}_{\ell=1}^{L}}_{\mix}}(\O) \cong H^{\beta}(\O)\). { In particular, since any \(j\in\{1,\ldots,L\}\) can be such that \(1/\s_j\in H^{\beta}(\O)\), we obtain \begin{align*}
    C_\low(k_{\mix}^{{\{\s_\ell,k_\ell\}}_{\ell=1}^L})&= C_{\ref{prop:banachalg}}^{-1} C_\low(k_j)   \|1/\sigma_j\|_{H^{\beta}(\O)}^{-1}, \\
    C_\up(k_{\mix}^{{\{\s_\ell,k_\ell\}}_{\ell=1}^L})&= \sqrt{2} C_{\ref{prop:banachalg}} \max_{{1 \leq \ell \leq L}} \{C_\up(k_\ell) \|\s_\ell\|_{H^{\beta}(\O)} \}.
\end{align*}
}
\end{proof}

\section{Sample path regularity}\label{sec:sample path regularity}

In this section, we explore the { sample }path properties of a Gaussian process, particularly focusing on the warping and mixture kernels introduced in section \ref{subsubsec:nonstationary}. Similar results hold for the convolution kernel, but are omitted for brevity since they are not used in this work. These will play a crucial role in the convergence analysis throughout the rest of this paper. We begin by recalling known results on both the classic regularity of M\'atern kernels and the Sobolev regularity properties of a Gaussian process.
\begin{proposition}[{\cite[Lemma C.1]{Nobile2015}}]\label{prop:regularity of matern kernel lemma 11}
    Let \(k=k_{\M(\nu)}\). Then \(k\in C^{2\nu}(\O\times\O)\) if \(\nu\) is not an integer and \(k\in C^{\a}(\O\times\O)\) for any \(\a<2\nu\) if \(\nu\) is an integer.
\end{proposition}
\begin{proposition}[{\cite[Theorem 2]{SCHEUERER20101879}}]\label{prop:regularity of GP samples}
    Let \(k\in C^{p,p}(\O\times\O\)) 
    and \(g\sim\mathcal{GP}(0,k)\). Then $g \in H^{p}(\O)$ almost surely. 
\end{proposition}
The remainder of this section establishes results regarding the classic regularity properties of a Gaussian process. 
In particular, we present general outcomes regarding the H\"older regularity properties based solely on the condition given in Assumption \ref{assump: sample path regularity} below. These results are detailed in Lemma \ref{lem:old school sob regularity method}, using the Sobolev embedding theorem, and in Lemma \ref{lemma: Sullivan regularity}, using mean-square properties. We provide more specific statements for the non-stationary kernels discussed in section \ref{subsubsec:nonstationary} in Corollaries \ref{cor:sample path old school} and \ref{cor: sample path sullivan}. 

While it may appear counterintuitive to employ two distinct methods for presenting sample properties based on the dimension under consideration, the Sobolev embedding theorems exhibit dimension dependence, while the mean-square properties do not. As a result, we choose Sobolev methods when \(d=1\) and rely on mean-square properties when \(d\geq 2\).

A Gaussian process \(g \sim \GP(0,k)\) is said to be {\it sample continuous} if \(g\) is continuous on \(\O\) almost surely, and {\it sample \(q\)-partial differentiable} if \(D^{q}g\) exists 
 on \(\O\) almost surely for \(q\in\N_0^d\). 
\begin{assumption}\label{assump: sample path regularity}
    We have \(k\in C^{p,p}(\O\times\O)\) for some \(p>0\).
\end{assumption}

\begin{lemma}\label{lem:old school sob regularity method}
    Suppose that Assumption \ref{assump: sample path regularity} is satisfied and \(g\sim\GP(0,k)\).  Then \(g\in C^{\a-d/2}(\O)\) almost surely, where  \(\a=p\) if \(p\) is not an integer and \(\a<p\) otherwise. 
\end{lemma}
\begin{proof}
    By Proposition \ref{prop:regularity of GP samples}, we see that \(g\in H^{p}(\O)\) almost surely. 
    The result then follows from the Sobolev embedding theorem (see, e.g. {\cite[Theorem 4.12]{RAAdams_JJFFournier_2003}}).
\end{proof}

An application of Lemma \ref{lem:old school sob regularity method} then gives the following result.
\begin{corollary}\label{cor:sample path old school}
    Suppose that either:
\begin{itemize}
    \item $k=k_{\warp}^{w,k_{\M(\nu)}}$ with \(w\in C^{p}(\O)\) such that \(p\leq \nu\), or
    \item{  \(k = k_{\mix}^{\{\s_\ell,k_{\M(\nu)}\}_{\ell=1}^L}\) with \(\s_\ell\in C^{p}(\O)\) for \(\ell=1,\ldots,L\) such that \(p\leq\nu\).
    }
\end{itemize} 
Then for \(g\sim\GP(0,k)\) we have that \(g\in C^{\a-d/2}(\O)\) almost surely, where  \(\a=p\) if \(p\) is not an integer and \(\a<p\) otherwise. 
\end{corollary}
\begin{proof}
    Assumption \ref{assump: sample path regularity} is satisfied due to Proposition \ref{prop:regularity of matern kernel lemma 11} and the chain rule or product rule. 
\end{proof}

 Sample regularity properties are crucial for our analytical framework, but establishing the mean-square properties of a Gaussian process is often more straightforward. Our objective is to infer conclusions about sample regularity properties through an analysis of mean-square properties. We follow the methodology outlined in \cite[section 3]{Wang2021BayesianNM}, where H\"older sample path properties for Gaussian processes with a M\'atern kernel are established through the use of mean-square properties.

We start by introducing the concepts of mean-square continuity and differentiability. We say that \(g \sim \GP(0,k)\) is
\begin{itemize}
    \item[(i)] \emph{mean-square continuous} on $\O$,  if for all \(u\in\O\)
    \[
        \E_{g\sim\GP(0,k)}[g(u)]<\infty\quad\text{ and } \lim_{u'\to u}\E_{_{g\sim\GP(0,k)}}[|g(u)-g(u')|^2]=0.
    \]
\item[(ii)]\emph{mean-square \(\bar{q}\)-partial differentiable} on $\O$ for \(\bar q\in\N^{d}_0\) with \(\|\bar q\|_{1}=1\), if  there exits a random field \(D^{\bar q}_{\MS}g\) with finite second moments, we have for all \(u\in\O\)
\[
    \lim_{h\to0}\E_{_{g\sim\GP(0,k)}}\left[\left|D^{\bar q}_{\MS}g-\frac{g(u)-g(u+{\bar q}h)}{h}\right|^2\right]=0.
\]
 For general \(q\in\N^{d}_0\), we denote \(D^{q}_{\MS}g\) as the \emph{mean-square \(q\)-partial  derivative of \(g\)}, and this is defined analogously. We refer to {\cite[section 3.1]{Wang2021BayesianNM}} for a more comprehensive definition and discussion.
\end{itemize}
{A {\it modification} of \(g \sim {\GP}(0,k)\) is a random process \(\tilde{g}\) such that for each \(u\in\O\) we have \(\P[g(u)=\tilde g(u)]=1\).} A Gaussian process is completely defined by its finite dimensional distributions, thus allowing us to work with a modification, as the laws of both \(g\) and \(\tilde g\) are the same.

For a mean-square continuous Gaussian process \(g \sim \GP(0,k)\) with \(k\in C^{|q|,|q|}(\O\times\O)\), we have that (see e.g. \cite{Wang2021BayesianNM})
\[
 D_{\MS}^{q}g \sim \GP\left(0,D^{q,q}k\right).
\]
We now prove the main result  on the H\"older regularity of samples from a Gaussian process through a sequence of lemmas. The final result is given in Lemma \ref{lemma: Sullivan regularity}. 

\begin{lemma}\label{lemma:sullival Prop1}
    Suppose $k$ satisfies Assumption \ref{assump: sample path regularity}. Then \(g \sim \GP(0,k)\) is mean-square $i$-partial differentiable and \(D_{\MS}^{i}\GP(0,k)\) is mean-square continuous for every \(i\in\N_0^d\) with \(0\leq |i|\leq \lceil p-1\rceil\).
\end{lemma}
\begin{proof}
    Our proof follows the proof of {\cite[Proposition 1]{Wang2021BayesianNM}}.  Fix \(i\in\N_0^d\) with \(0\leq |i|\leq \lceil p-1\rceil\). Mean-square  $i$-partial differentiability follows from \cite[section 2]{Stein1999} since \(k\in C^{p,p}(\O\times\O)\) by Assumption \ref{assump: sample path regularity}.
    Mean-square continuity of \(D_{\MS}^{i}\GP(0,k)\) follows since
    \[\E_{g\sim\GP(0,k)}[|D_{\MS}^{i}g(u)-D_{\MS}^{i}g(u')|^2]=D^{i,i}(k(u,u)+k(u',u')-2k(u,u'))\to 0 \text{ as } u\to u',\]
    since \(k\in C^{p,p}(\O\times\O)\).
\end{proof}
\begin{lemma}\label{lemma:sullival Prop2}
    Suppose $k$ satisfies Assumption \ref{assump: sample path regularity}. Then for all \(i\in\N_0^d\) with \(0\leq |i|\leq \lceil p-1\rceil\), \(D_{\MS}^{i}\GP(0,k)\) has a modification that is sample continuous.
\end{lemma}
\begin{proof}
    Since \(D^{i,i}k\in C^{p-|i|,p-|i|}(\O\times\O)\) and since \(p-|i|> \lceil p-1\rceil-|i|\geq0\) we see that \(D^{i,i}k\) is \(p-|i|-\)H\"older continuous on \(\O\times\O\) and so for \(u,u',\tilde{u}\in\O\) we have
    \[
        \left|D^{i,i}k(u,u')-D^{i,i}k(u,\tilde{u})\right|\leq C(k)\|u'-\tilde{u}\|_2^{p-|i|}.
    \]
    Hence, for any \(u,u'\in\O\)
     \[
       D^{i,i}k(u,u)+D^{i,i}k(u',u')-2D^{i,i}k(u,u')\leq C(k)\|u-u'\|_2^{p-|i|}.
    \]
    Following the proof of {\cite[Proposition 2, A.3]{Wang2021BayesianNM}} we see that this inequality, via Kolmogrov's Continuity Theorem (see e.g., {\cite[Theorem 3]{Wang2021BayesianNM}}), proves the result.
\end{proof}
\\
\begin{proposition}{\cite[Corollary 1]{Wang2021BayesianNM}}\label{prop:sullival cor 1}
    Let \(k\in C^{\lceil p-1\rceil,\lceil p-1\rceil}(\O\times\O)\), for some $p > 0$. Suppose that  for all \(i\in\N_0^d\) with \(0\leq |i|\leq \lceil p-1\rceil\), \(g \sim \GP(0,k)\) has mean-square derivatives \(D^{i}_{\MS}g\), and \(D^{i}_{\MS}g\) is mean-square continuous and sample path continuous. Then \(g\) has continuous sample derivatives \(D^{i}g\) which satisfy \(D^{i}g=D_{\MS}^{i}g\) almost surely, for all \(0 \leq |i|\leq \lceil p-1\rceil\).
\end{proposition}

\begin{lemma}\label{lemma: Sullivan regularity}
Suppose Assumption \ref{assump: sample path regularity} is satisfied.
    Then there exists a modification \(\tilde{g}\) of \(g \sim \GP(0, k)\) such that \(\tilde{g}\in C^{\lceil p-1\rceil}(\O)\) almost surely.
\end{lemma}
\begin{proof} By Assumption \ref{assump: sample path regularity} we have \(k\in C^{p,p}(\O\times\O)\) and so also \(k\in C^{\lceil p-1\rceil,\lceil p-1\rceil}(\O\times\O)\).
    By Lemma \ref{lemma:sullival Prop1}, \(g \sim \GP(0,k)\) is mean-square \(\lceil p-1\rceil\)-partial  differentiable and \(D^{i}_{\MS}g\) is mean-square continuous for \(0\leq |i|\leq \lceil p-1\rceil\). By Lemma \ref{lemma:sullival Prop2}, we can work with a modification \(\tilde{g}\) of \(g\) such that \(D^i_{MS}\tilde{g}\) is sample continuous for each \(0\leq |i|\leq \lceil p-1\rceil\). The result then follows from Proposition \ref{prop:sullival cor 1}.
\end{proof}


Notice that in Lemma \ref{lem:old school sob regularity method}, the regularity of sample paths decreases as \(d\) increases. In contrast, Lemma \ref{lemma: Sullivan regularity} is dimension-independent, and therefore will often be the choice we will use in dimensions \(d \geq 2\). Thus, in general, we shall only use Lemma \ref{lem:old school sob regularity method} when \(d = 1\).
An application of Lemma \ref{lemma: Sullivan regularity} gives the following result.
\begin{corollary}\label{cor: sample path sullivan}
    Suppose that either:
\begin{itemize}
    \item $k=k_{\warp}^{w,k_{\M(\nu)}}$ with \(w\in C^{p}(\O)\) such that \(p\leq \nu\) if \(\nu\) is not an integer and \(p<\nu\) otherwise, or
    {\item  \(k = k_{\mix}^{\{\s_\ell,k_{\M(\nu)}\}_{\ell=1}^L}\) with \(\s_\ell\in C^{p}(\O)\) for \(\ell=1,\ldots,L\) such that \(p\leq\nu\).}
\end{itemize} 
Then there exists a modification \(\tilde{g}\) of \(g \sim \GP(0, k)\) such that 
\(\tilde{g}\in C^{\lceil p-1\rceil}(\O)\) almost surely. 
\end{corollary}
\begin{proof}
    Assumption \ref{assump: sample path regularity} is satisfied due to Proposition \ref{prop:regularity of matern kernel lemma 11} and the chain rule or the product rule. The result then follows from Lemma \ref{lemma: Sullivan regularity}.
\end{proof}

\begin{remark}\label{rem:sob_vs_class} (Sobolev vs H\"older regularity) For \(k\in C^{p,p}(\O\times\O)\), the results in this section give sample path regularity $H^p(\O)$ and $C^{\tilde p}(\O)$, for $\tilde p < p$. There are results in the literature (see e.g. \cite[Theorem 1(ii)]{Dunlop2016}) which show that the Sobolev and H\"older sample path regularity of some Gaussian processes is the same, i.e. $\tilde p = p$. However, the proof of these results relies on properties of the eigendecomposition of the covariance operator of the Gaussian process, and to the best of our knowledge, there are no such results for general covariance kernels. Hence, the H\"older regularity proved in Corollaries \ref{cor:sample path old school} and \ref{cor: sample path sullivan} may be sub-optimal.  
\end{remark}


\section{Error analysis of non-stationary GP regression}\label{sec:EAGPR}
In this section we discuss the convergence of the GP emulator \(f_N\), as defined in \eqref{eq:gp_posterior} with fixed hyper-parameters $\theta$ and \eqref{eq:gp_posterior_est} with estimated hyper-parameters $\widehat \theta_N$, to the true function \(f\). 
In particular, we show that the mean function \(m_{N,\delta^2}^f\) converges to the true function \(f\) and, in the case of noise-free data, the covariance function \(k_{N,0}\) converges to 0. This type of convergence is required for example when we want to use GP emulators as surrogate models in inverse problems \cite{Stuart2016, Teckentrup2019, helin2023introduction}. 
Convergence rates will be given in terms of the fill distance $h_{U_N,\O}$, which is the maximum distance any point in \(\O\) can be from \(U_N\):
\[
h_{U_N,\O}:=\sup_{u\in \O}\inf_{u_n\in U_N}\|u-u_n\|_2.
\]
This can be translated into convergence rates in terms of the number of training points $N$ for specific point sets $U_N$, see, e.g. the discussion in \cite{Teckentrup2019} and the references therein. The fastest possible rate of { decay} for $h_{U_N,\O}$ is $N^{-1/d}$, and this is obtained for example for a tensor grid based on equispaced points.

We start by analysing the convergence of $m_{N,0}^f$ to $f$ (i.e. the noise-free case) in sections \ref{sec:conergence in native space} and \ref{sec:conv_sobolev}, under different assumptions on the function $f$ and kernel $k$ used in the GP prior \eqref{eq:GP empirical prior}. We analyse the convergence of $m_{N,\delta^2}^f$ to $f$ (i.e. the noisy case) in section \ref{sec:conv_sobolev_noisy}. We discuss the effect of estimating the hyper-parameters $\theta$ by $\widehat \theta_N$ in an empirical Bayesian approach in section \ref{subsec:eagr_empirical}. In section \ref{subsec:eagr_ext}, we present various extensions to the analysis presented in the previous sections, including convergence of $k_{N,0}$ to 0, faster convergence of $m_{N,0}^f$ and $k_{N,0}$ for separable kernels, and convergence of $m_{N,0}^f$, $m_{N,\delta^2}^f$ and $k_{N,0}$ in the case of randomly distributed training points $U_N$.


\subsection{Convergence of $m_{N,0}^f$ in the RKHS}\label{sec:conergence in native space}
In this section we are interested in the general setting where the native space $\H_k$ corresponding to the prior covariance kernel $k$ is not necessarily known explicitly. A benefit to this is that the results are applicable to a wide variety of kernels, but this also means that showing that \(f\) satisfies the required assumption that it is an element of \(\H_k(\O)\) is often difficult to verify. 
A general convergence result, based only on the general requirements given in Assumption \ref{assump:kernel1}, is presented in Proposition \ref{thm:Wend11.13} below. 

\begin{assumption}\label{assump:kernel1}
We have \(k\in C^{2p}(\Omega\times\Omega)\), for some \(p\in\N\), and \(f\in\mathcal{H}_{k}(\Omega)\).
\end{assumption}

\begin{proposition}[Convergence in mean for $f \in \H_k(\O)$ {\cite[Theorem 11.13]{wendland_2004}}]\label{thm:Wend11.13}
Suppose Assumption \ref{assump:kernel1} holds with $p \in \N$. Then there exist constants $h_0$ and $C$, independent of $f$, $k$ and $N$, such that
\begin{equation*}
    \|D^{\a}f-D^{\a}m_{N,0}^f\|_{C^0(\O)}\leq C C_{\ref{thm:Wend11.13}}(k)h^{p-|\a|}_{U_N,\Omega}\|f\|_{\H_k(\Omega)},
\end{equation*}
for all $\a\in\N^d_0$ with $|\a|<p$ and $h_{U_N,\Omega}\leq h_0$.
The constant $C_{\ref{thm:Wend11.13}}$ is given by
\[
    C_{\ref{thm:Wend11.13}}(k):= \max_{\substack{(\a_1,\a_2)\in\N^{2d}_0\\|\a_1|+|\a_2|=2p}}\max_{u_1,u_2\in\O}\left|D^{\a_1,\a_2}k(u_1,u_2)\right|.
\]
\end{proposition}

In Assumption \ref{assump:kernel1}, the regularity assumption on $k$ can often be easy to verify, and this is done for specific examples of the warping, mixture and convolution kernels from section \ref{subsubsec:nonstationary}. When the RKHS $\H_k(\O)$ is not known explicitly, the regularity assumption on $f$ can be much more difficult to verify, and we do not address this assumption further in this work. When the RKHS is known explicitly, for example given by a Sobolev space as in section \ref{sec:conv_sobolev}, then faster convergence rates than those in Proposition \ref{thm:Wend11.13} can sometimes be obtained by using the specific structure of the RKHS. Note that even though the RKHS is not known explicitly, some general properties can be deduced from Assumption \ref{assump:kernel1}. For example, we know that we have $\H_k(\O) \hookrightarrow C^p(\O)$ by \cite[Theorem 10.45]{wendland_2004}.

In Lemmas \ref{lemma:exp_inputwarping constant}, \ref{lemma:exp_kernel mixture} and \ref{lemma:exp_convolutionker constant} below, we verify the regularity assumption on $k$ for special cases of the non-stationary kernels presented in section \ref{subsubsec:nonstationary}. The resulting convergence rates are summarised in Corollary \ref{cor:convergence in the native space}. For the stationary kernel used in the construction of the non-stationary kernels, we focus on the Gaussian kernel $k_{\M(\infty)}$, for which the RKHS is defined by rather technical conditions (see e.g. \cite[Theorem 1]{Minh2010}).
Since the Gaussian kernel is in \(C^{\infty}(\O \times \O)\), the regularity of the non-stationary kernel depends only on the regularity of the hyper-parameters \(w, \{\sigma_\ell\}_{\ell=1}^L\) and \({\lambda_{a}}\). If we use a stationary kernel with finite regularity, then the regularity will depend both on the hyper-parameters \(w, \{\sigma_\ell\}_{\ell=1}^L\) and \({\lambda_{a}}\) and the stationary kernel. For instance, for the warping kernel we obtain regularity \(2p=\min\{2q,2\tilde{q}\}\) for \(w\in C^{2q}(\O)\) and  \(k_s\in C^{2\tilde{q}}(\O\times\O)\).



\begin{lemma}[Warping kernel]\label{lemma:exp_inputwarping constant}\label{lemma:mat3/2_inputwarping constant}
    Let $\O \subset \R$. Suppose $(i)$ \(k_s=k_{\M(\infty)}\) and $(ii)$ \(w\in C^{2p}(\O)\) for some \(p\in\N\). Then $k_{\warp}^{w,k_s} \in C^{2p}(\O \times \O)$, and 
    
    \[
    C_{\ref{thm:Wend11.13}}(k_{\warp}^{w,k_s})=C_0\s^2\sqrt{(2p)!}B_{2p}\|w\|_{C^{2p}(\Omega)},
    \]
     where \(C_0\leq 1.0866\) is a constant and \(B_{2p}\) is the Bell number, which represents the number of possible partitions of the set \(\{1,\ldots,2p\}\).
\end{lemma}
\begin{lemma}[Mixture kernel] \label{lemma:exp_kernel mixture}
Let \(\O\subset\R\). Suppose $\{\s_\ell,k_\ell\}_{\ell=1}^L$ are such that for all \(\ell=1,\ldots,L\) we have $(i)$ \(k_\ell = k_{\M(\infty)}\), and $(ii)$ \(\s_\ell\in C^{2p}(\O)\) for some \(p\in\N\). Then $k_{\mix}^{\{\s_\ell,k_\ell\}_{\ell=1}^L} \in C^{2p}(\O \times \O)$, and with \(C_0\leq 1.0866\) is a constant,
    \[
        C_{\ref{thm:Wend11.13}}(k_{\mix}^{\{\s_\ell,k_\ell\}_{\ell=1}^L})=C_02p2^{4p}\sqrt{(2p)!}\max_{1 \leq \ell \leq L} \|\s_\ell\|^2_{C^{2p}(\Omega)}.
    \]
\end{lemma}

\begin{lemma}[Convolution kernel]\label{lemma:exp_convolutionker constant}
    Let \(\O\subset\R\). Suppose $(i)$ $k_i = k_{\M(\infty)}$, $(ii)$  \({\lambda_{a}}(u) \geq c>0\) for all $u \in \O$, and $(iii)$ \({\lambda_{a}}\in C^{2p}(\O)\) for some \(p\in\N\). Then \(k_{\conv}^{{\lambda_{a}},k_i} \in C^{2p}(\O \times \O)\), and 
    \begin{align*}
     C_{\ref{thm:Wend11.13}}(k_{\conv}^{{\lambda_{a}},k_i})= 2^{1/2}\s^2&\left(|5/4 - 2p|^{2p}\max\{\|{\lambda_{a}}\|_{C^{0}(\Omega)}, 1\}^{1/4 - 2p}
    \|{\lambda_{a}}\|_{C^{2p}(\Omega)}\right)^2\nonumber\\&
\Bigg((4p)!2C_{0}\sqrt{(2p!)} \|{\lambda_{a}}\|^2_{C^{2p}(\Omega)}
(2p)!(2c)^{-1-2p})\|{\lambda_{a}}\|_{C^{2p}(\Omega)}\Bigg)B_{2p}^5
\end{align*}
where \(C_0\leq 1.0866\) is a constant and \(B_{2p}\) is the Bell number, which represents the number of possible partitions of the set \(\{1,\ldots,2p\}\).
\end{lemma}

The proof of Lemmas \ref{lemma:exp_inputwarping constant}, \ref{lemma:exp_kernel mixture} and \ref{lemma:exp_convolutionker constant} can be found in the appendix. For brevity we only consider the case $\O \subset \R$, but the general case can be proved using similar techniques. An application of Proposition \ref{thm:Wend11.13} then gives the following result.

\begin{corollary}\label{cor:convergence in the native space}
   Let $k \in \{k_{\warp}^{w,k_s}, k_{\mix}^{\{\s_\ell,k_\ell\}_{\ell=1}^L}, k_{\conv}^{{\lambda_{a}},k_i}\}$. Suppose  $(i)$ the assumptions of Lemma \ref{lemma:exp_inputwarping constant}, \ref{lemma:exp_kernel mixture} or \ref{lemma:exp_convolutionker constant} hold with \(p \in \N\), respectively, and $(ii)$ $f \in \H_k(\O)$. Then there exists a constant $h_0>0$ such that
   \begin{equation*}
    \|D^{\mu}f-D^{\mu}m_{N,0}^f\|_{C^0(\O)}\leq C_{\ref{thm:Wend11.13}}(k)h^{p-|\mu|}_{U_N,\Omega}\|f\|_{\H_k(\Omega)},
\end{equation*}
for all $\mu\in\N^d_0$ with $|\mu|<p$ and $h_{U_N,\Omega}\leq h_0$. 
\end{corollary}

\subsection{Convergence of $m_{N,0}^f$ in a Sobolev space with noise-free data}\label{sec:conv_sobolev}
The main drawback of the results in section \ref{sec:conergence in native space} is the lack of transparency concerning the assumption that \(f\) needs to be an element of the RKHS of the non-stationary kernel. In this section, we prove convergence results in the special case that the RKHS is norm-equivalent to a Sobolev space, as considered in section \ref{sec:nativespaces_nonstationary}.
A general convergence result, based only on the general requirements given in Assumption \ref{assump:kernel2}, is presented in Proposition \ref{thm:Convergence_in_sob} below. 

\begin{assumption}\label{assump:kernel2}
For some \(\beta\geq d/2\), we have $\H_k(\O) \cong H^\beta(\O)$ with constants \(C_\low(k)\) and \(C_\up(k)\) such that
\[
    C_\low(k) \|\cdot\|_{\H_{k}(\O)}\leq \|\cdot\|_{H^\beta(\O)}\leq C_\up(k)\|\cdot\|_{\H_{k}(\O)},
\]
and $f \in H^\beta(\O)$.
\end{assumption}
\begin{proposition}[Convergence in mean for $f \in H^\beta(\O)$ {\cite[Theorem 3.5]{Teckentrup2019}} ] \label{thm:Convergence_in_sob}
Suppose Assumption \ref{assump:kernel2} holds with $\beta \geq d/2$. Then there exist constants \(h_0\) and \(C_{\ref{thm:Convergence_in_sob}}\), independent of $f$, $k$ and $N$, such that
\begin{equation*}
    \|f-m_{N,0}^f\|_{H^\alpha(\Omega)}\leq C_{\ref{thm:Convergence_in_sob}} C_{\low}(k)^{-1}C_{\up} (k) \, h_{U_N,\Omega}^{\beta-\alpha} \,\|f\|_{H^{\beta}(\Omega)},
\end{equation*}
for all \(\alpha \leq \beta\) and \(h_{U_N,\O}\leq h_0\).
\end{proposition}
The specific statement of Proposition \ref{thm:Convergence_in_sob}, with the dependence on the constants $C_\low(k)$ and $C_\up(k)$ made explicit, is due to \cite{Teckentrup2019}. This will be important in section \ref{subsec:eagr_empirical}, when we estimate hyper-parameters in $k$. However, the result was originally proved in \cite{Narcowich2005}. 
Additionally, note that due to \cite[Theorem 3.5]{Teckentrup2019}, Proposition \ref{thm:Convergence_in_sob} (and all other parts of this work that rely on Assumption \ref{assump:kernel2}) still holds if the regularity of the true function \(f\) is overestimated in the kernel \(k\). { Specifically, if \(f \in H^{\beta_+}(\Omega)\) for any \(\beta_+ \geq \beta\), then the results of Proposition \ref{thm:Convergence_in_sob} are applicable.} Conversely, this is not the case if \(f \in H^{\beta_-}(\Omega)\) where \(\beta_- \leq \beta\); in such scenarios, we would observe slower convergence rates in any analysis depending on Assumption \ref{assump:kernel2}.

An application of Proposition \ref{thm:Convergence_in_sob} then gives the following convergence result. 
\begin{corollary}\label{cor:conv_sob}
Let $k \in \{k_{\warp}^{w,k_s}, k_{\mix}^{\{\s_\ell,k_\ell\}_{\ell=1}^L}\}$. Suppose $(i)$ the assumptions of Theorem \ref{lemma:sob equiv warping} or \ref{lemma: sob equiv mixture} hold with $\beta \geq d/2$, respectively, and $(ii)$ $f \in H^\beta(\O)$.
Then there exist constants \(h_0\) and \(C_{\ref{thm:Convergence_in_sob}}\), independent of $f$, $k$ and $N$, such that
\begin{equation*}
    \|f-m_{N,0}^f\|_{H^\alpha(\Omega)}\leq C_{\ref{thm:Convergence_in_sob}} C_{\low}(k)^{-1}C_{\up} (k) \, h_{U_N,\Omega}^{\beta-\alpha} \,\|f\|_{H^{\beta}(\Omega)},
\end{equation*}
for all \(\alpha \leq \beta\) and \(h_{U_N,\O}\leq h_0\).
\end{corollary}


\subsection{Convergence of $m_{N,\delta^2}^f$ in a Sobolev space with noisy or noise-free data}\label{sec:conv_sobolev_noisy}

In the case of noisy data, we have the following general convergence result. The proof is given at the end of this section. Recall that our observations take the form $y_N = \{u_n, f(u_n)+\varepsilon_n\}_{n=1}^N$.

\begin{theorem} \label{thm:conv_sob_noisy} Suppose Assumption \ref{assump:kernel2} holds with $\beta \geq d/2$. Then there exists a constant \(h_0\), independent of $f$, $k$, $\delta^2$ and $N$, such that for any $\varepsilon = \{\varepsilon_n\}_{n=1}^N \in \R^N$ we have
\begin{align*}
 \|f-m_{N,\delta^2}^f\|_{H^\alpha(\Omega)} &\leq  h_{U_N,\Omega}^{\beta-\alpha} \, \left( 2 C_{\low}(k)^{-1}C_{\up}(k) \|f\|_{H^\beta(\Omega)} +  \delta^{-1 }C_{\up}(k) \|\varepsilon\|_2\right)
\\
& \qquad \qquad + \, h_{U_N,\Omega}^{d/2-\alpha} \left( 2 \|\varepsilon\|_2 + \delta C_{\low}(k)^{-1}C_{\up}(k)) \|f\|_{H^\beta(\Omega)} \right),
\end{align*}
for all \(\alpha \leq \lfloor \beta \rfloor\) and \(h_{U_N,\O}\leq h_0\).

If further $f \in H^{\beta}(\overline \O)$, $0 < k_\mathrm{min} := \min_{u \in \overline \O} k(u,u)$ and $\varepsilon \sim N(0,\delta^2\mathrm{I})$, then there exists a constant \(C_{\ref{thm:conv_sob_noisy}}\), independent of $f$, $k$, $\delta^2$ and $N$, such that
\begin{align*}
\mathbb{E}_\varepsilon \left[ \|f-m_{N,\delta^2}^f\|_{H^\alpha(\Omega)}^2\right]^{1/2} &\leq C_{\ref{thm:conv_sob_noisy}} \left( h_{U_N,\Omega}^{\beta-\alpha} \, \left( C_{\low}(k)^{-1}C_{\up}(k) \|f\|_{H^\beta(\Omega)} +  C_{\up}(k) N^{1/2}\right) \right.
\\
& \qquad \left. + \, h_{U_N,\Omega}^{d/2-\alpha} \left( C_{\low}(k)^{-1+d/2\beta} \|f\|^{1-d/2\beta}_{H^\beta(\overline \O)}  + C_{\up}(k)^{d/2\beta} \right) N^{d/4\beta} \right),
\end{align*}
for all \(\alpha \leq \lfloor \beta \rfloor\), \(h_{U_N,\O}\leq h_0\) and $N \geq  N^*(k_{\mathrm{min}})$. 
\end{theorem}

The first term in the error bounds in Theorem \ref{thm:conv_sob_noisy}, involving  $h_{U_N,\Omega}^{\beta-\alpha} C_{\low}(k)^{-1}C_{\up}(k) \|f\|_{H^\beta(\Omega)}$, converges at the same rate as the error for noise-free data in Proposition \ref{thm:Convergence_in_sob}. The other terms quantify the effect of the noise on convergence. 
Note that the second bound does not derive from taking the second moment of the first bound. Applying the equality $\mathbb{E}_\varepsilon[\|\varepsilon\|_2^2] = \delta^2 N$ to the first bound would give a factor of $h_{U_N,\Omega}^{d/2-\alpha} N^{1/2}$ in the second bound, which is non-decreasing in $N$ for any $\alpha \geq 0$ since $h_{U_N,\O} = o(N^{-1/d})$. However, a refined proof technique allows to conclude on the sharper bound with $N^{d/4\beta}$.

As already discussed in \cite{wynne2021convergence}, we note that the first error bound in Theorem \ref{thm:conv_sob_noisy} applies to any noise $\varepsilon$, and hence also applies in the misspecified setting where the assumption $\varepsilon \sim N(0,\delta^2\mathrm{I})$ used to derive $m_{N,\delta^2}^f$ is not satisfied. In particular, in the case of noise-free data with $\varepsilon=0$ {, where the posterior mean \(m^f_{N,\delta^2}\) (see \eqref{eq:gp_mean}) is still constructed under the assumption $\varepsilon \sim N(0,\delta^2\mathrm{I})$}, we have
\[
\|f-m_{N,\delta^2}^f\|_{H^\alpha(\Omega)} \leq  h_{U_N,\Omega}^{\beta-\alpha} \, \left( 2 C_{\low}(k)^{-1}C_{\up}(k) \|f\|_{H^\beta(\Omega)} \right) + \, h_{U_N,\Omega}^{d/2-\alpha} \left( \delta C_{\low}(k)^{-1}C_{\up}(k)) \|f\|_{H^\beta(\Omega)} \right).
\]
Since $\beta\geq d/2$, the second term converges at a generally slower rate than the first. This can be offset by choosing the noise $\delta$ small, and in particular, we recover the rate $h_{U_N,\Omega}^{\beta-\alpha}$ of $\|f-m_{N,0}^f\|_{H^\alpha(\Omega)}$ if we choose $\delta = O(h_{U_N,\Omega}^{\beta-d/2})$.

An application of Theorem \ref{thm:conv_sob_noisy} gives the following convergence result.

\begin{corollary}\label{cor:conv_sob_noisy}
Let $k \in \{k_{\warp}^{w,k_s}, k_{\mix}^{\{\s_\ell,k_\ell\}_{\ell=1}^L}\}$. Suppose $(i)$ the assumptions of Theorem \ref{lemma:sob equiv warping} or \ref{lemma: sob equiv mixture} hold with $\beta \geq d/2$, respectively, $(ii)$ $0 < \sigma_\mathrm{min} \leq \sigma_{j}$ for $\nu_j = \min_{1 \leq \ell \leq L} \nu_\ell$ if $k = k_{\mix}^{\{\s_\ell,k_\ell\}_{\ell=1}^L}$, and $(iii)$ $f \in H^\beta(\overline \O)$.
Then Assumption \ref{assump:kernel2} holds with $\beta$, and the conclusions of Theorem \ref{thm:conv_sob_noisy} hold.
\end{corollary}

\begin{proof} [Of Theorem \ref{thm:conv_sob_noisy}] We follow the same proof structure as in \cite[Theorem 2 and Lemma 17]{wynne2021convergence}, making simplifications/improvements in some steps and keeping track of how the constants appearing in the error estimates depend on the covariance kernel $k$. For brevity, we outline the proof and only give detail in the steps where modifications take place. Note that $(i)$ Assumptions 1-5 of \cite{wynne2021convergence} are satisfied in the setting of this work, $(ii)$ Assumption 6 of \cite{wynne2021convergence} is proved below, and $(iii)$ Assumption 7 of \cite{wynne2021convergence} can be relaxed to $f \in H^{\beta}(\overline \O)$ since we are assuming that the RKHS of $k$ is $H^\beta(\O)$.

An application of \cite[Theorem 3.2]{arcangeli2012extension}  (i.e. \cite[Theorem 12]{wynne2021convergence}) gives, for a constant $C_1$ independent of $f$, $k$ and $N$, 
\[
\|f-m_{N,\delta^2}^f\|_{H^\alpha(\Omega)} \leq C_1 \left( h_{U_N,\Omega}^{\beta-\alpha} \,\|f - m_{N,\delta^2}^f\|_{H^{\beta}(\Omega)} + h_{U_N,\Omega}^{d/2-\alpha} \,\|f - m_{N,\delta^2}^f\|_{2,U_N} \right), 
\]
for all $\alpha \leq \lfloor \beta \rfloor$ and $h_{U_N,\Omega} \leq h_0$. Here $\|g\|_{2,U_N} = \sqrt{ \sum_{n=1}^N g(u_n)^2}$ denotes the 2-norm of a function $g \in H^\beta(\O)$ evaluated on the discrete set $U_N$. Following the same proof technique as in \cite[Theorem 2 and Lemma 17]{wynne2021convergence}, we obtain for the first term the bound
\begin{align*}
\|f-m_{N,\delta^2}^f\|_{H^\beta(\Omega)} &\leq (1+C_{\low}(k)^{-1}C_{\up}(k)) \|f\|_{H^\beta(\Omega)} + \delta^{-1} C_{\up}(k) \|\varepsilon\|_2.
\end{align*}
For the second term, we have by \cite[Lemma 17]{wynne2021convergence}
\begin{align*}
\|f - m_{N,\delta^2}^f\|_{2,U_N} \leq 2 \|\varepsilon\|_2 + \delta C_{\low}(k)^{-1}C_{\up}(k)) \|f\|_{H^\beta(\Omega)}.
\end{align*}
This finishes the proof of the first claim.

For the second claim, For the second term, we use \cite[Theorem 1]{Vaart2011} (i.e. \cite[Theorem 11]{wynne2021convergence}), which gives, for a constant $C_2$ independent of $f$, $k$ and $N$, 
\[
\mathbb{E}_\varepsilon \left[\|f - m_{N,\delta^2}^f\|_{2,U_N}^2\right]^{1/2} \leq C_2 N^{1/2}(\psi_{f,k}^{-1}(N)).
\]
Here, $\psi_{f,k}^{-1}$ is defined via
\begin{align}\label{eq:def_conc}
    \psi_{f,k}^{-1}(N) &= \sup \{\epsilon >0 : \psi_{f,k}(\epsilon) \geq N \}, \\
    \psi_{f,k}(\epsilon) &= \frac{\phi_{f,k}(\epsilon)}{\epsilon^2}, \nonumber \\
    \phi_{f,k}(\epsilon) &= \inf_{\substack{h \in \mathcal H_k(\overline \O)  \\ \|h-f\|_{L^\infty(\overline \O)} < \epsilon}} \frac{1}{2} \|h\|^2_{\mathcal H_k(\overline \O)} - \log \mathbb{P}_k\left[\{ g: \|g\|_{L^\infty(\overline \O)} \leq \epsilon \} \right], \nonumber
\end{align}
where $\mathbb{P}_k$ denotes the measure of the prior $\mathcal{GP}(0,k)$. We hence proceed to bound the concentration function $\phi_{f,k}(\epsilon)$. For the first term, we use the simple bound $h=f$ to give
\[
\inf_{\substack{h \in \mathcal H_k(\overline \O) \\ \|h-f\|_{L^\infty(\overline \O)} < \epsilon}} \frac{1}{2} \|h\|^2_{\mathcal H_k(\overline \O)} \leq \frac{1}{2} \|f\|^2_{\mathcal H_k(\overline \O)} \leq \frac{1}{2} C_{\low}(k)^{-2} \|f\|^2_{H^\beta(\overline \O)},
\]
where by slight abuse of notation we use $C_{\low}(k)$ to denote the norm equivalence constant on $\O$ as well as $\overline \O$.

For the second term, we track the dependence on $k$ in the proof of \cite[Lemma 24]{wynne2021convergence} to obtain the bound
\[
H(\mathcal H'_k(\overline \Omega), \epsilon) \leq C_3 C_{\up}(k)^{d/\beta} \epsilon^{-d/\beta},
\]
where $\mathcal H'_k(\overline \Omega)$ denotes the unit ball in $\mathcal H_k(\overline \Omega)$ and $H(\mathcal H'_k(\overline \Omega), \epsilon)$ denotes the metric entropy, i.e. the logarithm of the $\epsilon$-covering number of $\mathcal H'_k(\overline \Omega)$. The main observation to make is that the unit ball in $\mathcal H_k(\overline \Omega)$ is contained in the ball of radius $C_{\up}(k)$ in $H^\beta(\overline \O)$. As in \cite[Equation (34)]{wynne2021convergence}, we then have with $\phi^{(2)}_{f,k}(\epsilon) := - \log \mathbb{P}_k\left[\{ g: \|g\|_{L^\infty(\overline \O)} \leq \epsilon \} \right]$
\[
\phi^{(2)}_{f,k}(\epsilon) = \log 2 + C_3 C_{\up}(k)^{d/\beta} \epsilon^{-d/\beta} 8^{d/2\beta} \phi^{(2)}_{f,k}(\epsilon/2)^{d/2\beta}.
\]
Now observe that, using the lower bound $\int_c^\infty \exp(-t^2/2) \mathrm{d} t \geq \frac{c}{c^2+1} \exp(-c^2/2)$ (see e.g. \cite[Equation (2.23)]{gine_nickl_2015}, we have for $u \in \overline \O$ s.t. $k(u,u) = \min_{u \in \overline \O} k(u,u) =k_\mathrm{min}$ the bound
\[
 \mathbb{P}_k\left[\{ g: \|g\|_{L^\infty(\overline \O)} \leq c \} \right] \leq \mathbb{P}_k\left[\{ g: |g(u)| \leq c \} \right] \leq 1 - \frac{2c k_\mathrm{min}}{\sqrt{2\pi}(c^2+k_\mathrm{min}^2)} \exp(-\frac{c^2}{2 k_\mathrm{min}^2}).
\]
This means that \cite[Assumption 6]{wynne2021convergence} is satisfied for any $N \in \N$ with any $c>0$ and $\gamma_N \equiv \gamma(k,c) = - \log(1 - \frac{2c k_\mathrm{min}}{\sqrt{2\pi}(c^2+k_\mathrm{min}^2)} \exp(-\frac{c^2}{2 k_\mathrm{min}^2}))$.

For $\epsilon/2 \leq \min \{ c, (\log 2)^{-1} C_3 8^{d/2\beta} \gamma(k,c)^{d/2\beta})^{\beta/d} \} =: \epsilon^*$ for some $0<c<1$, we then have
\[
\phi^{(2)}_{f,k}(\epsilon) =  C_3 (C_{\up}(k)^{d/\beta}+1) \epsilon^{-d/\beta} 8^{d/2\beta} \phi^{(2)}_{f,k}(\epsilon/2)^{d/2\beta}.
\]
Taking logarithms and following the iterative procedure as in the proof of \cite[Theorem 2]{wynne2021convergence}, we arrive at the estimate
\[
\phi^{(2)}_{f,k}(\epsilon) \leq C_4 C_{\up}(k)^{2d/(2\beta-d)} \epsilon^{-2d/(2\beta-d)},
\]
for all $\epsilon \leq \epsilon^*$ and a constant $C_4$ independent of $f$, $k$ and $N$. This gives 
\begin{align*}
\phi_{f,k}(\epsilon) \epsilon^{-2} &\leq \frac{1}{2} C_{\low}(k)^{-2} \|f\|^2_{H^\beta(\overline \O)} \epsilon^{-2} + C_4 C_{\up}(k)^{2d/(2\beta-d)} \epsilon^{-2-2d/(2\beta-d)}, \\
&\leq \left( \frac{1}{2} C_{\low}(k)^{-2} \|f\|^2_{H^\beta(\overline \O)}  + C_4 C_{\up}(k)^{2d/(2\beta-d)} \right) \epsilon^{-2-2d/(2\beta-d)},
\end{align*}
for all $\epsilon \leq \epsilon^*$. For sufficiently large $N \geq N^*(\epsilon^*)$, we then have
\begin{equation}\label{eq: concentration func upper bound}
\psi_{f,k}^{-1}(N) \leq \left( \frac{1}{2} C_{\low}(k)^{-2} \|f\|^2_{H^\beta(\overline \O)}  + C_4 C_{\up}(k)^{2d/(2\beta-d)} \right)^{1/2-d/(4\beta)} N^{-1/2 + d/4\beta}.
\end{equation}
Combining all steps together and using the equality $\mathbb{E}_\varepsilon[\|\varepsilon\|_2^2] = \delta^2 N$, we have for $\alpha \leq \lfloor \beta \rfloor$, $h_{U_N,\Omega}^{\beta-\alpha} \leq h_0$ and $N \geq N^*(k_{\mathrm{min}}$,
\begin{align*}
&\mathbb{E}_\varepsilon \left[ \|f-m_{N,\delta^2}^f\|_{H^\alpha(\Omega)}^2\right]^{1/2} \leq C_1 \left( h_{U_N,\Omega}^{\beta-\alpha} \, \left( (1+C_{\low}(k)^{-1}C_{\up}(k)) \|f\|_{H^\beta(\Omega)} +  C_{\up}(k) N^{1/2}\right) \right.
\\
& \qquad \qquad \left. + \, h_{U_N,\Omega}^{d/2-\alpha} \,C_2 \left( \frac{1}{2} C_{\low}(k)^{-2} \|f\|^2_{H^\beta(\overline \O)}  + C_4 C_{\up}(k)^{2d/(2\beta-d)} \right)^{-1/2+d/(4\beta)} N^{d/4\beta} \right),
\end{align*}
which completes the proof.
\end{proof}


\subsection{Hyper-parameters estimated by empirical Bayes}\label{subsec:eagr_empirical}

In this section we study the effect of the estimation of hyper-parameters on the error bounds presented in sections \ref{sec:conergence in native space}, \ref{sec:conv_sobolev} and \ref{sec:conv_sobolev_noisy}.
Recall from section \ref{subsec:Estimating hyper-paramters} that within the framework of { empirical Bayes}, we compute a point estimate $\widehat \theta_N$ of the hyper-parameters $\theta$ using the data $y_N$, and then plug this into the prior \eqref{eq:GP empirical prior} to be used in the standard GP regression prior-to-posterior update. For ease of presentation, we will restrict our attention to the case where the stationary kernel is a Mat\'ern kernel $k_{\M(\nu)}$. For the non-stationary kernels considered in this work, we are then interested in estimating the following hyper-parameters:
\begin{itemize}
    \item For $k_{\warp}^{w,k_s}$, we estimate $\theta = \{w, \sigma^2\}$ and fix $\{\lambda,\nu\}$.
    \item For $k_{\mix}^{\{\s_\ell,k_\ell\}_{\ell=1}^L}$, we estimate $\theta = \{L,\{\s_\ell\}_{\ell=1}^L, \{\lambda_\ell\}_{\ell=1}^L \}$ and fix {$\{\sigma^2, \nu\}$}.
    \item For $k_{\conv}^{{\lambda_{a}},k_i}$, we estimate $\theta = \{{\lambda_{a}}, \sigma^2 \}$ and fix $\{\lambda,\nu\}$.
\end{itemize}


For the warping and convolution kernels, the non-stationary functions $w$ and ${\lambda_{a}}$, respectively, are designed to capture length scales. The parameter $\lambda$ in the Mat\'ern kernel is hence superfluous, and can be kept fixed (e.g. at 1). Similarly, the marginal variance parameter $\sigma^2$ in the Mat\'ern kernel can be kept fixed (e.g. at 1) for the mixture kernel. 

{ In all cases, it is possible to extend the results to the smoothness parameter $\nu$ being estimated as well. If the target function $f$ is in the RKHS of the non-stationary kernel for all possible values of $\widehat \nu_N$, the convergence results presented here hold with straightforward modifications. If $f$ on the other hand has limited regularity, and is not necessarily an element of the RKHS for all $\widehat \nu_N$, then the error bounds require a qualitative change as presented in \cite[Theorem 3.5]{Teckentrup2019}. Due to the significant increase in required notation and technical complexity, the estimation of \(\nu\) is presented separately in section \ref{sec: estimating nu}.}


\subsubsection{Convergence of $m_{N,0}^f$ in the RKHS with noise-free data}\label{subsubsec:conv_rkhs_est}
If Assumption \ref{assump:kernel1} holds for \(k(\cdot,\cdot;\widehat\theta_N)\), 
for all $N \in \N$, then Proposition \ref{thm:Wend11.13} immediately gives the error bound
\begin{equation*}
    \|D^{\mu}f-D^{\mu}m_{N,0}^f(\widehat \theta_N)\|_{C^0(\O)}\leq C_{\ref{thm:Wend11.13}}(k(\cdot,\cdot;\widehat\theta_N)) \; h^{p-|\mu|}_{U_N,\Omega}\;  \|f\|_{\H_{k(\cdot,\cdot;\widehat\theta_N)}(\Omega)}.  
\end{equation*}

The convergence rates we obtain hence depend on the behaviour of $C_{\ref{thm:Wend11.13}}(k(\cdot,\cdot;\widehat\theta_N))$ and $\|f\|_{\H_{k(\cdot,\cdot;\widehat\theta_N)}(\Omega)}$. If these quantities can be bounded uniformly in $N$, we obtain the same convergence rates as in the case of fixed hyper-parameters, see Corollary \ref{cor:conv_rkhs_est}. More generally, we obtain a convergence rate of $h^{p-|\mu|}_{U_N,\Omega} \, N^r$ if $ C_{\ref{thm:Wend11.13}}(k(\cdot,\cdot;\widehat\theta_N)) \|f\|_{\H_{k(\cdot,\cdot;\widehat\theta_N)}(\Omega)} \leq C N^r$ for some $r \in \R$ for our sequence of estimated hyper-parameters $\{\widehat \theta_N\}_{N=1}^\infty$.

If the RKHS $\H_{k(\cdot,\cdot;\widehat\theta_N)}(\O)$ is not known explicitly, it may be difficult to quantify the behaviour of $\|f\|_{\H_{k(\cdot,\cdot;\widehat\theta_N)}(\Omega)}$. The constant $C_{\ref{thm:Wend11.13}}(k(\cdot,\cdot;\widehat\theta_N))$ is more tractable, and can be bounded uniformly in $N$ under the assumptions given in Theorem \ref{thm: conditions on hyperparam native space norm} below. 

\begin{theorem}
\label{thm: conditions on hyperparam native space norm}
    Let $p \in \N$ and \(\O\subset\R\). Suppose either
    \begin{itemize}
        \item[(i)]
        $k=k_{\warp}^{w,k_s}$, 
        $k_s = k_{\M(\infty)}$ and \( \{\|\widehat w_N\|_{C^{2p}(\Omega)},\widehat\s^2_{N}\}\) lie in a compact subset of $(0,\infty)^2$,
        \item[(ii)] $k = k_{\mix}^{\{\s_\ell,k_\ell\}_{\ell=1}^L}$, $k_\ell = k_{\M(\infty)}$ for $1 \leq \ell \leq \infty$, and $ \{\widehat L_N, \max_{1 \leq \ell \leq \widehat L_N} \|\widehat\s_{\ell,N}\|_{C^{2p}(\Omega)}\}$ lie in a compact subset of $\N \times (0,\infty)$, or
        \item[(iii)] $k=k_{\conv}^{{\lambda_{a}}, k_i}$, $k_i = k_{\M(\infty)}$ and \( \{\|\widehat {\lambda_{a,N}}\|_{C^{2p}(\Omega)}, \inf_{u \in \O} \widehat {\lambda_{a,N}}(u), \widehat\s^2_{N}\}\) lie in a compact subset of $(0,\infty)^3$.
    \end{itemize}
    Then there exists a constant $C_{\ref{thm: conditions on hyperparam native space norm}}$, independent of $N$, such that 
    \[
\sup_{\widehat \theta_N} C_{\ref{thm:Wend11.13}}(k(\cdot,\cdot;\widehat\theta_N)) \leq C_{\ref{thm: conditions on hyperparam native space norm}}.
    \]
\end{theorem}
\begin{proof} For $(i)$ the condition on \(\widehat\s^2_{N}\) follows directly from \cite[Theorem 3.5]{Teckentrup2019}, and the condition on \(\widehat w_N\) follows from Lemma \ref{lemma:exp_inputwarping constant}. For $(ii)$, the conditions follow from Lemma \ref{lemma:exp_kernel mixture}. For $(iii)$, the conditions on \(\widehat\s^2_{N}\) follow directly from \cite[Theorem 3.5]{Teckentrup2019}, and the condition on \(\widehat \lambda_{a,N}\) and $\inf_{u \in \O} {\lambda_{a,N}}(u)$ follows from Lemma \ref{lemma:exp_convolutionker constant}.
\end{proof}

    Note that there are no conditions on \(\{\widehat\lambda_{\ell,N}\}_{\ell=1}^L\) in Theorem \ref{thm: conditions on hyperparam native space norm} $(ii)$. This is because the only terms that depend on \(\{\widehat\lambda_{\ell,N}\}_{\ell=1}^L\) are inside the exponential and this can always be bounded by \(1\), i.e. \(\exp(-\|u-u'\|_2^2/\widehat\lambda_{\ell,N})\leq 1\) for all \(u,u'\in\O\) and \(0 <\widehat\lambda_{\ell,N} < \infty\), see \eqref{eq:Herrmaneq} in the appendix. 

\begin{corollary}
\label{cor:conv_rkhs_est}
Suppose the assumptions Theorem \ref{thm: conditions on hyperparam native space norm} hold. Then for \(f\in \H_{k(\cdot,\cdot;\widehat\theta_N)}(\Omega)\) there exists a constant $h_0$, independent of $f$ and $N$, such that
\begin{equation*}
    \|D^{\mu}f-D^{\mu}m_{N,0}^f(\widehat \theta_N)\|_{C^0(\O)}\leq C_{\ref{thm: conditions on hyperparam native space norm}} \; h^{p-|\mu|}_{U_N,\Omega}\; \sup_{\widehat\theta_N} \|f\|_{\H_{k(\cdot,\cdot;\widehat\theta_N)}(\Omega)},   
\end{equation*}
 for all $\mu\in\N_0$ with $\mu<p$ and $h_{U_N,\Omega}\leq h_0$.
\end{corollary}

\subsubsection{Convergence of $m_{N,\delta^2}^f$ in a Sobolev space with noisy or noise-free data}\label{subsubsec:conv_sob_est}
{ If Assumption \ref{assump:kernel2} holds for \(k(\cdot,\cdot;\widehat\theta_N)\),  for all $N \in \N$, Proposition \ref{thm:Convergence_in_sob} and Theorem \ref{thm:conv_sob_noisy} similarly give the error bounds
\begin{align*}
    \|f-m_{N,0}^f(\widehat \theta_N)\|_{H^\alpha(\Omega)}&\leq C_{\ref{thm:Convergence_in_sob}} C_{\low}(k(\cdot,\cdot; \widehat \theta_N))^{-1}C_{\up} (k(\cdot,\cdot; \widehat \theta_N))  \, h_{U_N,\Omega}^{\beta-\alpha} \,\|f\|_{H^{\beta}(\Omega)}, \\
     \|f-m_{N,\delta^2}^f(\widehat \theta_N)\|_{H^\alpha(\Omega)} &\leq  h_{U_N,\Omega}^{\beta-\alpha} \, \left( 2 C_{\low}(k(\widehat \theta_N))^{-1}C_{\up}(k(\widehat \theta_N)) \|f\|_{H^\beta(\Omega)} +  \delta^{-1 }C_{\up}(k(\widehat \theta_N)) \|\varepsilon\|_2\right)
\\
& \qquad \qquad + \, h_{U_N,\Omega}^{d/2-\alpha} \left( 2 \|\varepsilon\|_2 + \delta C_{\low}(k(\widehat \theta_N))^{-1}C_{\up}(k(\widehat \theta_N)) \|f\|_{H^\beta(\Omega)} \right).
\end{align*}
}
Note that in contrast to section \ref{subsubsec:conv_rkhs_est}, the norm on $f$ appearing here is independent of any of the estimated hyper-parameters. 
The obtained convergence rates hence depend only on the behaviour of the constants $C_{\low}(k(\cdot,\cdot; \widehat \theta_N))^{-1}$ and $C_{\up} (k(\cdot,\cdot; \widehat \theta_N))$. In particular, Theorem \ref{lem:cond on estimates for Sob bnds} below gives conditions under which moments of these constant can be bounded uniformly in $N$.

\begin{theorem}
\label{lem:cond on estimates for Sob bnds}
    Let $\beta \geq d/2$. Suppose either
    \begin{itemize}
        \item[(i)]
        $k=k_{\warp}^{w,k_s}$, 
        $k_s = k_{\M(\beta - d/2)}$ and \( \{\|\widehat w_N\|_{C^{\beta}(\Omega)}, \|\inv(\widehat w_N)\|_{C^{\beta}(\Omega)}, \min_{u \in \O} \widehat w_N'(u),\min_{u \in \O} \inv(\widehat w_N)'(u),\widehat\s^2_{N}\}\) lie in a compact subset of $(0,\infty)^5$, or
        \item[(ii)] $k = k_{\mix}^{\{\s_\ell,k_\ell\}_{\ell=1}^L}$, $k_\ell = k_{\M(\nu_\ell)}$ for $1 \leq \ell \leq \infty$, and $ \{\widehat L_N, \max_{1 \leq \ell \leq \widehat L_N} \{ \|\widehat\s_{\ell,N}\|_{H^{\beta}(\Omega)}, \|1/\widehat\s_{\ell,N}\|_{H^{\beta}(\Omega)}\} \}$ lie in a compact subset of $\N \times (0,\infty)^2$.
    \end{itemize}
    Then for any $a,b > 0$ there exists a constant $C_{\ref{lem:cond on estimates for Sob bnds}}(a,b)$, independent of $N$, such that
    \[
\sup_{\widehat \theta_N} C_{\low}(k(\cdot,\cdot; \widehat \theta_N))^{-a}C_{\up} (k(\cdot,\cdot; \widehat \theta_N))^b \leq C_{\ref{lem:cond on estimates for Sob bnds}}(a,b).
    \]
\end{theorem}
\begin{proof}
For $(i)$, the condition on \(\widehat\s^2_{N}\) follows directly from \cite[Theorem 3.5]{Teckentrup2019} and the conditions on \(\widehat w_N\) follow from Theorem \ref{lemma:sob equiv warping}.       For $(ii)$, this follows from Theorem \ref{lemma: sob equiv mixture}.
\end{proof}

\begin{corollary}
\label{cor:conv_sob_est}
    Suppose the assumptions of Theorem \ref{lem:cond on estimates for Sob bnds} and Corollary \ref{cor:conv_sob} hold. Then 
\begin{equation*}
    \|f-m_{N,0}^f(\widehat \theta_N)\|_{H^\alpha(\Omega)}\leq C_{\ref{thm:Convergence_in_sob}} C_{\ref{lem:cond on estimates for Sob bnds}}(1,1)  \, h_{U_N,\Omega}^{\beta-\alpha} \,\|f\|_{H^{\beta}(\Omega)},
\end{equation*}
 for all \(\alpha \leq \beta\) and \(h_{U_N,\O}\leq h_0\). 
\end{corollary}

\begin{corollary}
\label{cor:conv_sob_est_noisy}
    Suppose the assumptions of Theorem \ref{lem:cond on estimates for Sob bnds} and Corollary \ref{cor:conv_sob_noisy} hold, for $\s_\mathrm{min}$ independent of $N$. Then the conclusions of Theorem \ref{thm:conv_sob_noisy} hold.
\end{corollary}

\subsubsection{Computing hyper-parameter estimates}
The results presented in sections \ref{subsubsec:conv_rkhs_est} and \ref{subsubsec:conv_sob_est} do not make any assumptions on how the estimated hyper-parameters $\widehat \theta_N$ are computed. The methodology to do this efficiently is an important research area of its own, but is not the focus of this work and so we only briefly discuss this here. 

Estimation of the stationary parameters $\{\lambda, \sigma^2\}$ in the Mat\'ern kernel, including questions of identifiability, has been studied extensively (see, e.g. the discussion and references in \cite{Teckentrup2019}), and is frequently done using maximum marginal likelihood estimation or cross validation.
The non-stationary parameters $w$, $\{\sigma_\ell\}_{\ell=1}^L$ and ${\lambda_{a}}$ are functions, and so typically a suitable parametrisation is chosen before estimation.
In deep kernel GPs \cite{Wilson2016}, which have garnered considerable interest recently, the warping function \(w\) takes the form of a deep neural network. In \cite{Volodina2020}, the authors discuss efficient estimation of the mixture coefficients $\{\sigma_\ell\}_{\ell=1}^L$ and the number of components $L$ using methods akin to cross validation. A piece-wise linear model is chosen for the length scale ${\lambda_{a}}$ in \cite{fisher2020}.


\subsection{Extensions of error analysis}\label{subsec:eagr_ext}

In this section, we provide several extensions of the error analysis provided in sections \ref{sec:conergence in native space}-\ref{subsec:eagr_empirical}. We apply established proof techniques in each context, and, for brevity, we will omit detailed derivations and refer the reader instead to the provided references.

\subsubsection*{Convergence of $k_{N,0}$}
In the case of noise-free data $y_N$, the posterior variance $k_{N,0}$ satisfies the equality in Proposition \ref{prop:predvar_sup} below, allowing us to transfer convergence results on $m_{N,0}^f$ to convergence results on $k_{N,0}$ (see, e.g. \cite{Teckentrup2019}).
{Given  the design points $U_N$, we introduce the mapping $m_{N,0}^{(\cdot)} : \H_k(\O) \rightarrow \H_k(\O)$, which is built on the definition of the posterior mean $m_{N,0}^f$ given in \eqref{eq:gp_mean}. With $g(U_N) := [g(u^1); \dots; g(u^N)] \in \R^N$, we define
\begin{align} \label{eq:pred_mean_g}
g(u) \mapsto m_{N,0}^g(u) = k(u,U_N)^T K(U_N,U_N)^{-1} g(U_N).
\end{align}

\begin{proposition}[e.g. {\cite[Proposition 3.2]{Teckentrup2019}}]
\label{prop:predvar_sup} Suppose $k_{N,0}$ and $m_N^{(\cdot)}$ are given by \eqref{eq:gp_variance} with $\delta^2=0$ and \eqref{eq:pred_mean_g}, respectively. Then for any $u \in \O$ we have
\[
k_{N,0}(u,u)^{\frac{1}{2}} = \sup_{\substack{{g \in \H_k(\O) \; \text{s.t. }} \\ \|g\|_{\H_k(\O)}=1}} | g(u) - m^g_N(u)|.
\]
\end{proposition}}

An application of Proposition \ref{prop:predvar_sup} then gives us the following convergence results on $k_{N,0}$. In slight abuse of notation, let use denote by $k(\cdot;\widehat \theta_N)$ the point-wise posterior variance at $u$, i.e. $k(u;\widehat \theta_N) := k(u,u;\widehat \theta_N)$ for $u \in \O$.

\begin{corollary} \label{cor:var_native} Suppose the assumptions of Proposition \ref{thm:Wend11.13} hold. Then there exists a constant $h_0$, independent of $f$ and $N$, such that
\begin{equation*}
    \|k_{N,0}(\widehat \theta_N)^{1/2}\|_{C^0(\O)} \leq C_{\ref{thm:Wend11.13}}(k(\cdot,\cdot;\widehat\theta_N)) h^{p}_{U_N,\Omega},   
\end{equation*}
for all $h_{U_N,\Omega}\leq h_0$.
\end{corollary}
\begin{proof} This follows directly from Propositions \ref{thm:Wend11.13} and \ref{prop:predvar_sup}.
\end{proof}

\begin{corollary}\label{cor:var_sobolev} Suppose the assumptions of Proposition \ref{thm:Convergence_in_sob} hold.
Then there exist constants \(h_0\) and \(C_{\ref{cor:var_sobolev}}\), 
    independent of $f$, $k$ and $N$, such that for any $\epsilon > 0$
\begin{equation*}
    \|k_{N,0}(\widehat \theta_N)^{1/2}\|_{L^2(\Omega)}\leq C_{\ref{cor:var_sobolev}}C_{\low}(k(\cdot,\cdot; \widehat \theta_N))^{-1}C_{\up} (k(\cdot,\cdot; \widehat \theta_N))^2 \, h_{U_N,\Omega}^{\beta-d/2 - \epsilon},
\end{equation*}
 for \(h_{U_N,\O}\leq h_0\).
\end{corollary}
\begin{proof} This follows from Propositions \ref{thm:Convergence_in_sob} and \ref{prop:predvar_sup}, following the same proof technique as in \cite[Theorem 3.8]{Teckentrup2019}. 
\end{proof}

Under the conditions of Theorems \ref{thm: conditions on hyperparam native space norm} and \ref{lem:cond on estimates for Sob bnds}, respectively, the constants in the above error bounds can again be bounded uniformly in $N$.


\subsubsection*{Separable covariance kernels}
All error bounds derived so far have been expressed in terms of the fill distance $h_{U_N,\Omega}$. This can be translated into convergence rates in $N$ for specific sets of design points $U_N$. For example, the fill distance decays at rate $N^{-1/d}$ for uniform grids. This is the fastest decay rate that can be obtained by any $U_N$, and so there is a strong dimension dependence in the error bounds presented so far.

For $f \in H^\beta(\O)$, the rate $N^{-\beta/d}$ given in Proposition \ref{thm:Convergence_in_sob} for $\|f-m_{N,0}^f\|_{L^2(\O)}$ is optimal (see, e.g. the discussion in \cite{Teckentrup2019}) and so cannot be improved. However, by imposing further restrictions on $f$, we are able to improve the error bounds. This additionally requires a specific structure on $\O$, $U_N$ and $k$. In particular, we require that the kernel $k$ is separable, i.e. that it can be written as the product of one-dimensional kernels $k(u,u') = \prod_{j=1}^d k^j(u_{(j)},u'_{(j)})$, and that the domain $\O$ is of tensor product structure $\Omega = \prod_{j=1}^d \Omega_j$. The training points $U_N$ are chosen as a sparse grid \cite{bungartz2004sparse}, defined as the set of points
\begin{equation*}
H(q,{{d}}) := \bigcup_{\|i \|_1 = q} X_1^{(i_1)} \times \cdots X_{{d}}^{(i_{{d}})},
\end{equation*}
where $i \in \N_0^{{d}}$, $q \geq {{d}}$, and for $1 \leq j \leq {{d}}$ we choose a sequence $X_j^{(p)} := \{x_{j,1}^{(p)}, \dots, x_{j,m_p}^{(p)}\}$, $p \in \N$, of nested sets of points in $\Omega_j$.
We then have the following result.
\begin{corollary}\label{cor:conv_sep}  
Suppose $(i)$ $\O = \prod_{j=1}^{{d}} \O_j$,
$(ii)$ $U_N$ is chosen as the Smolyak sparse grid $H(q,{{d}})$ with
$h_{X_j^{(p)},\Omega_j} \leq C_h m_p^{-r_h}$, 
$(iii)$ $k^j \in \{k_{\warp}^{w,k_s}, k_{\mix}^{\{\s_\ell,k_\ell\}_{\ell=1}^L}\}$, satisfying the assumptions of Theorem \ref{lemma:sob equiv warping} or \ref{lemma: sob equiv mixture}, respectively, and
$(iv)$ $f \in H^{\{\beta_j\}}_{\otimes^{{d}}} (\O) := \otimes_{j=1}^{{d}} H^{\beta_j}(\O_j)$. 
Then there exist constants $C_{\ref{cor:conv_sep}}$ and $C_{\ref{cor:conv_sep}}'$, independent of $f$, $k$ and $N$, such that
\begin{align*}
\| f - m_{N,0}^f(\widehat \theta_N)\|_{H^{\{\alpha_j\}}_{\otimes^{{d}}}(\Omega)} &\leq C_{\ref{cor:conv_sep}} \left( \prod_{j=1}^d C_\low^{-1}(k^j(\widehat \theta_N)) C_\up(k^j(\widehat \theta_N))\right) N^{-\gamma} (\log N)^{(1+\gamma)(d - 1)} \|f \|_{H^{\{\beta_j\}}_{\otimes^{{d}}}(\O)}, \\
\|k_N(\widehat \theta_N)^{1/2}\|_{L^2(\O)} &\leq C_{\ref{cor:conv_sep}}' \left( \prod_{j=1}^d C_\low^{-1}(k^j(\widehat \theta_N))  C_\up(k^j(\widehat \theta_N))^2\right) N^{-\gamma+\frac{1}{2}+\epsilon} (\log N)^{(\frac{3}{2}+\epsilon+\gamma)(d - 1)}.
\end{align*}
for any $\epsilon >0$ and $\{\alpha_j\}\leq  \{\beta_j\}$, where $\gamma = \min_{1 \leq j \leq {{d}}} r_h (\beta_j - \alpha_j)$.
\end{corollary}
\begin{proof} This follows directly from \cite[Theorems 3.11 and 3.12]{Teckentrup2019}, together with Theorems \ref{lemma:sob equiv warping} and \ref{lemma: sob equiv mixture}.
\end{proof}

Notice that the rate $\gamma$ in Corollary \ref{cor:conv_sep} is the same rate in $N$ that we would obtain with Proposition \ref{thm:Convergence_in_sob} in the case $d=1$, provided $h_{U_N,\O} \leq C N^{-1/d}$ in Proposition \ref{thm:Convergence_in_sob} and $h_{X_j^{(p)},\Omega_j} \leq C_h m_p^{-1}$ in Corollary \ref{cor:conv_sep} (i.e. the fill distance decays at the optimal rate in both cases). Up to a log factor, the convergence rates in Corollary \ref{cor:conv_sep} are independent of $d$.
Under the conditions of Theorem \ref{lem:cond on estimates for Sob bnds} on $k^j$, the constants in the above error bounds can again be bounded uniformly in $N$. Extending Corollary \ref{cor:conv_sep} to the case of noisy data is beyond the scope of this work.

\subsubsection*{Random training points $U_N$}
{ When the training points $U_N$ are chosen randomly, error bounds in terms of the fill distance for fixed $U_N$ are not always meaningful. Corollary \ref{cor:conv_random} below gives a result in expectation over $U_N$ in the case of noise-free data.}
\begin{corollary} \label{cor:conv_random} Suppose $U_N$ are sampled i.i.d. from a measure $\mu$ with density ${\rho}$ satisfying ${\rho(u)} \geq {\rho_{\mathrm{min}}} > 0$ for all $u \in \overline{\O}$, and the assumptions of Corollary \ref{cor:conv_sob_est} hold $\mu$-almost surely. Then there exist constants $C_{\ref{cor:conv_random}}$ and $C_{\ref{cor:conv_random}}'$, independent of $f$ and $N$, such that for all $\alpha \leq \beta$ and $\epsilon > 0$ we have
\begin{align*}
\mathbb{E}_{\mu} \left[  \|f - m_{N,0}^f(\widehat \theta_N)\|_{H^\alpha(\O)} \right] &\leq C_{\ref{cor:conv_random}} \; N^{-\frac{\beta-\alpha}{d} + \epsilon} \, \|f\|_{H^\beta(\O)}, \\
\mathbb{E}_{\mu} \left[  \|k_{N,0}^{1/2}(\widehat \theta_N)\|_{L^2(\O)} \right] &\leq C_{\ref{cor:conv_random}}' \; N^{-\frac{\beta}{d} + \frac{1}{2d} + \epsilon} \, \|f\|_{H^\beta(\O)}.
\end{align*}
\end{corollary}
\begin{proof} The claim follows from \cite[Theorem 3]{helin2023introduction}, with $U_N = \emptyset$ and $\pi^y(u)\equiv 1$, as well as Proposition \ref{prop:predvar_sup} and the same proof technique as in \cite[Theorem 3.8]{Teckentrup2019}.
\end{proof}

The proof of this result is mainly based on the observation that under the conditions of Theorem \ref{lem:cond on estimates for Sob bnds} (holding almost surely in $\mu$) the error bound in Proposition \ref{thm:Convergence_in_sob} depends on $U_N$ only through the fill distance $h_{U_N,\Omega}$, together with bounds on $\mathbb{E}_\mu[h_{U_N, \O}^{\beta - \alpha}] $ and $\mathbb{P}_\mu[h_{U_N, \O} > h_0]$ proved in \cite{oates2019convergence,Teckentrup2019}. Note that the rates obtained in Corollary \ref{cor:conv_random} are almost optimal, in the sense that the fill distance $h_{U_N,\O}$ in Corollary \ref{cor:conv_sob_est} in the best case decays at rate $N^{-1/d}$. 


For noisy data, we consider the following result for random training points $U_N$. Recall the definition of $\psi_{f,k}^{-1}$ from \eqref{eq:def_conc}. 
\begin{corollary}\label{cor:conv_random w/ noise}
     Suppose $U_N$ are sampled i.i.d. from a measure $\mu$ with density ${\rho}$ satisfying ${\rho(u)} \geq {\rho_{\mathrm{min}}} > 0$ for all $u \in \overline{\O}$, and the assumptions of Corollary \ref{cor:conv_sob_noisy} hold \(\mu\)-almost surely. For \(\hat\beta=\beta\) for \(\beta\) not an integer and \(\hat\beta<\beta\) otherwise, define \(\a(1)=\beta-1/2\) for \(k_{\warp}^{w,k_s}\), and define
    \(\alpha(1)=\hat\beta -1/2\) and \( \alpha(d)=\lceil\hat\beta-d/2-1\rceil\), for \( d\geq 2\), for \(k=k_{\mix}^{\{\s_\ell,k_\ell\}_{\ell=1}^L}\). Let $\psi_{f,k}^{-1}$ be as in \eqref{eq:def_conc}. 
    
    Then for \(f\in H^{\beta}(\overline\O)\) there exists a constant \(C_{\ref{cor:conv_random w/ noise}}(k)\), independent of \(f\) and \(N\), such that for sufficiently large \(N \geq N^*(\sigma_{\min})\)
    \begin{align*}
    \E_{\varepsilon}\left[\E_{\mu}\left[\|f - m_{N,\delta^2}^f(\widehat \theta_N)\|^2_{L^2(\O)}\right]\right]^{1/2}
    &\leq\left( \frac{1}{2} C_{\low}(k(\widehat \theta_N))^{-2} \|f\|^2_{H^\beta(\overline \O)}  + C_4 C_{\up}(k(\widehat \theta_N))^{2d/(2\beta-d)} \right)^{t(1/2-d/(4\beta))}\\&\hspace{4.5cm} \times C_{\ref{cor:conv_random w/ noise}}(k(\widehat \theta_N)) N^{-r}
    \end{align*}
     where
      \[
    \quad  
    (r,t) =
    \begin{cases}
        \left(1 - \frac{d}{2\beta}, 1\right), 
        &\text{if } \psi_{f,k}^{-1}(N) \leq N^{-d/(4\alpha(d)+2d)}, \\[10pt]
        \left(\frac{2\alpha(d) + 2d}{d} - \frac{\alpha(d) + d}{\beta}-1, \frac{4\alpha(d) + 4d}{d} \right), 
        &\text{if } \psi_{f,k}^{-1}(N) > N^{-d/(4\alpha(d)+2d)}.
    \end{cases}
\]
\end{corollary}
{
\begin{proof}
    We break the proof into two cases depending on the kernel \(k\).

    \textbf{Case 1, Warping:} 
    We have \(w\in C^{\beta}(\O)\) for \(\beta = \nu+1/2\) by assumption (see Theorem \ref{lemma:sob equiv warping}). Since then \(w\in C^{\nu}(\O)\), we use Corollary \ref{cor:sample path old school} with \(p=\nu\) to see that for any \(g\sim\GP(0,k_{\warp}^{w,k_s})\) we have \(g\in C^{\a(1)}(\O)\) almost surely.

    \textbf{Case 2, Mixture:} 
    \(\s_{\ell}\in H^{\beta}(\O)\) for  { \(\beta=\nu+d/2\)} by assumption (see Theorem \ref{lemma: sob equiv mixture}). By the Sobolev embedding theorem (see, e.g. {\cite[Theorem 4.12]{RAAdams_JJFFournier_2003}}) \(\s_{\ell}\in C^{\hat\beta-d/2}(\O)\). For \(d=1\) (respectively \(d\geq 2\))  we use Corollary \ref{cor:sample path old school} (respectively Corollary \ref{cor: sample path sullivan}) with \(p=\hat\beta-d/2\leq{\nu_j}\) to see that for any \(g\sim\GP(0,k_{\mix}^{\{\s_\ell,k_\ell\}_{\ell=1}^L})\) we have \(g\in C^{\a(d)}(\O)\) almost surely.
    
    In both cases, using the upper bound for \(\psi_{f,k}^{-1}(N)\) in \eqref{eq: concentration func upper bound} and \cite[Theorem 2]{Vaart2011}, the claim then follows with
    \(C_{\ref{cor:conv_random w/ noise}}(k(\widehat\theta_N)) = \rho_{\min}^{-1} C\), where $C$ is the constant appearing in \cite[Theorem 2]{Vaart2011}.
\end{proof}
Note that Corollary \ref{cor:conv_random w/ noise} is stated for a fixed estimate for the hyper-parameters \(\widehat\theta_N\). It is possible to extend this to a bound uniform in the estimated hyper-parameters, as in Corollary \ref{cor:conv_random}, by tracking the dependence on  \(k\) of the constant \(C\) appearing in \cite[Theorem 2]{Vaart2011}. For brevity, we have omitted this analysis.
}

{\subsubsection*{Estimating \(\nu\)}\label{sec: estimating nu}
In this section, we examine the impact of estimating \(\nu\) on the error bounds presented in section \ref{sec:conv_sobolev}. Following the methodology outlined in section \ref{subsec:eagr_empirical}, we employ the empirical Bayes framework, where a point estimate \(\widehat{\nu}_N\) of \(\nu\) is derived from the data \(y\). 
The focus of this section is on the estimation of \(\nu\); hence, we exclude the estimation of other stationary hyper-parameters and estimate the following:
\begin{itemize}
    \item For \(k_{\text{warp}}^{w,k_s}\), we estimate \(\theta = \{w, \nu\}\) and fix \(\{\sigma^2, \lambda\}\).
    \item For \(k_{\text{mix}}^{\{\sigma_\ell, k_\ell\}_{\ell=1}^L}\), we estimate \(\theta = \{\nu, \{\sigma_\ell\}_{\ell=1}^L\}\) and fix \(\{L, \sigma^2, \{\lambda_\ell\}_{\ell=1}^L\}\).
\end{itemize}
The estimation of \(\nu\) is addressed separately from the other hyper-parameters in section \ref{subsec:eagr_empirical} because the target function \(f\) may not necessarily belong to the RKHS for all \(\widehat{\nu}_N\). This necessitates a revision of the error bounds provided in other sections of this paper. Specifically, the convergence rate will now depend on both the fill distance \(h_{U_N, \Omega}\) and the mesh ratio \(\rho_{U_N,\Omega}\), defined as
\[
    \rho_{U_N,\Omega} := \frac{2h_{U_N,\Omega}}{\min_{i \neq j} \|u_j - u_i\|},
\]
which represents the ratio between the fill distance and half the minimum distance between any two points in \(U_N\). This adjustment is based on \cite[Theorem 3.5]{Teckentrup2019}.

For a kernel \(k(\cdot, \cdot; \widehat{\theta}_N)\), Theorem \ref{lem:cond on nu estimates for Sob bnds}, which is analogous to Theorem \ref{lem:cond on estimates for Sob bnds}, provides conditions under which the moments of \(C_{\text{low}}(k(\cdot, \cdot; \widehat{\theta}_N))^{-1}\) and \(C_{\text{up}}(k(\cdot, \cdot; \widehat{\theta}_N))\) can be uniformly bounded in \(N\).
{
\begin{theorem}\label{lem:cond on nu estimates for Sob bnds}
    Let \(\beta^{+}=\sup_{N\in\N}\{\widehat\nu_{N}\}+d/2\) and \(\beta^{-}=\inf_{N\in\N}\{\widehat\nu_{N}\}+d/2\) be such that \(\beta^-\geq d/2\). Suppose either
    \begin{itemize}
        \item[(i)]
        $k=k_{\warp}^{w,k_s}$, 
        $k_s = k_{\M(\nu)}$ and, \(\{\|\widehat w_N\|_{C^{\beta^+}(\Omega)}, \|\inv(\widehat w_N)\|_{C^{\beta^+}(\Omega)}, \min_{u \in \O} \widehat w_N'(u),\min_{u \in \O} \inv(\widehat w_N)'(u),\widehat\nu_{N}\}\) lie in a compact subset of $(0,\infty)^5$, or 
        \item[(ii)] $k = k_{\mix}^{\{\s_\ell,k_\ell\}_{\ell=1}^L}$, $k_\ell = k_{\M(\nu)}$ for $1 \leq \ell \leq L$, and $ \{\max_{1 \leq \ell \leq L} \{ \|\widehat\s_{\ell,N}\|_{H^{\beta^+}(\Omega)}\}, \|1/\widehat\s_{j,N}\|_{H^{\beta^+}(\Omega)}, \widehat\nu_{N}\}$ lie in a compact subset of $(0,\infty)^3$ for some \(j\in\{1,\ldots,L\}\).
    \end{itemize}
    Then for any $a,b > 0$ there exists a constant $C_{\ref{lem:cond on estimates for Sob bnds}}(a,b)$, independent of $N$, such that
    \[
\sup_{\widehat \theta_N} C_{\low}(k(\cdot,\cdot; \widehat \theta_N))^{-a}C_{\up} (k(\cdot,\cdot; \widehat \theta_N))^b \leq C_{\ref{lem:cond on nu estimates for Sob bnds}}(a,b).
    \]
\end{theorem}}
\begin{proof}
    For \((i)\) the condition on \(\widehat\nu_{N}\) follows follows directly from {\cite[Theorem 3.5]{Teckentrup2019}} and, since \(\widehat w_N\in C^{\widehat\nu_N -d/2}(\O)\) for \(N\in\N\), the conditions on \(\widehat w_N\) follow from Theorem \ref{lemma:sob equiv warping} . For $(ii)$, this follows from Theorem \ref{lemma: sob equiv mixture}, again using the fact that \(\widehat\s_{\ell,N}\in H^{\widehat\nu_N}(\O)\) for $\ell=1,\dots,L$ and \(1/\widehat\s_{j,N}\in H^{\widehat\nu_N-d/2}(\O)\) for some \(j\in\{1,\ldots,L\}\).
\end{proof}
\begin{corollary}
\label{cor:conv_sob_est for nu}
    Suppose the assumptions of Theorem \ref{lem:cond on nu estimates for Sob bnds} hold and \(f\in H^{\tilde\beta}(\O)\) for \(\tilde\beta>d/2\). Then there exists constants \(h_0\) and \(C_{\ref{cor:conv_sob_est for nu}}\), independent of $f$, $k$ and $N$, such that
\begin{equation*}
    \|f-m_{N,0}^f(\widehat \theta_N)\|_{H^\alpha(\Omega)}\leq C_{\ref{cor:conv_sob_est for nu}} C_{\ref{lem:cond on nu estimates for Sob bnds}}(1,1)  \, h_{U_N,\Omega}^{\min\{\tilde\beta,\beta^{-}\}-\alpha} \,\rho_{U_N, \O}^{\max\{\beta^{+}-\tilde\beta, 0\}}\|f\|_{H^{\tilde\beta}(\Omega)},
\end{equation*}
 for all \(\alpha \leq \tilde\beta\) and \(h_{U_N,\O}\leq h_0\).
\end{corollary}
\begin{proof}
    Follows from Theorem \ref{lem:cond on nu estimates for Sob bnds} and {\cite[Theorem 3.5]{Teckentrup2019}}.
\end{proof}
}


\section{Error analysis of deep (and wide) GP regression}\label{sec:EADGPR}
In this section, we employ the developed theory to prove convergence rates for DGP regression. Recall from section \ref{subsec:Estimating hyper-paramters} that in this set-up, we have a deep Gaussian process $f^{D}$ of depth $D$, constructed using a sequence of stochastic processes that are conditionally Gaussian \cite{damianou2013deep,dunlop2018deep}:
\begin{align*}
\nonumber f^0&\sim\mathcal{GP}(0, k_0(u,u')),\\
f^{n}|f^{n-1}&\sim\mathcal{GP}(0, k_{n}(u,u';f^{n-1})), \qquad n=1,\dots, D.
\end{align*}

The initial layer $f^0$ is defined as a stationary GP, for example with Mat\'ern covariance function $k_0 = k_{\M(\nu_0)}$. In subsequent layers, the covariance kernel is constructed using a sample from the previous layer. In this way, we can design a DGP based on non-stationary kernels as in section \ref{subsubsec:nonstationary}, where the non-stationary function(s) \(w\), \(\{\s_\ell\}_{\ell=1}^L\) or $\lambda_a$ rely on a sample from the previous layer. For instance, the penultimate layer \(f^{D-1}\) influences the non-stationarity in \(f^D\); this is achieved by defining the warping function \(w = f^{D-1}\) or the marginal variance function \(\s\) through a strictly positive function \(F\), where \(\s = F(f^{D-1}) : \Omega \to \mathbb{R}_+\).
In particular, we consider the following constructions:
\begin{itemize}
    \item deep GPs: $k_0 = k_{\M(\nu_0)}$, and $k_n \in \{k_{\warp}^{f^{n-1},k_{\M(\nu_n)}}, k_{\mix}^{F(f^{n-1}),k_{\M(\nu_n)}}\}$ with $F$ strictly positive. 
\end{itemize}
All other hyper-parameters appearing in $\{k_n\}_{n \in \N}$ are kept fixed. Note that this is a more general construction than the one considered in \cite{dunlop2018deep}, since we are allowing the smoothness $\nu_n$ to vary with $n$. In the mixture kernel \( k_{\mix}^{F(f^{n-1}),k_{\M(\nu_n)}} \) we only have one mixture component in this set-up, and \( F(f^{n-1})^2k_{\M(\nu_n)}(0) \) hence defines a non-stationary marginal variance for $f^n$. 

Alternatively, we would like to place prior distributions on the multiple mixture components \(\{\s_\ell\}_{\ell=1}^L\). To achieve this, we construct a DGP with width \(L\) and depth $D=1$ so that 
    \begin{align*}
        f^{0,\ell}&\sim\GP(0,k_{0}(u,u'))\quad\text{ for }\ell=1,\ldots, L, \quad i.i.d., \nonumber\\
        f^{1,L}|\{f^{0,\ell}\}_{\ell=1}^{L}&\sim\GP(0,k_{1}(u,u'; \{f^{0,\ell}\}_{\ell=1}^L)).
    \end{align*}
For clarity and to distinguish this construction from the DGP construction discussed above, we refer to this construction as a wide GP (WGP).
We then consider the following construction:
\begin{itemize}
    \item wide GPs: $k_0 = k_{\M(\nu_0)}$, and $k_{1} = k_{\mix}^{\{F(f^{0,\ell}),k_{\M(\nu_{1})}\}_{\ell=1}^{L}}$ with $F$ strictly positive.
\end{itemize}
Again, all other hyper-parameters appearing in $k_1$ are kept fixed.


The remainder of this section is devoted to proving convergence rates for DGP and WGP regression, i.e. regression where the prior is chosen as a DGP or WGP. { Conditioning the DGP or WGP priors on noise-free data becomes very technical and in general requires working with disintegrations (see e.g. \cite{cosg19} and the references therein). We hence only consider posterior distributions conditioned on noisy data $y_N = \{u_n, f(u_n)+\varepsilon_n\}_{n=1}^N$, where $\varepsilon_n \sim N(0,\delta^2)$ i.i.d. with $\delta>0$. We denote the resulting posteriors on the final layers $f^{D}$ and $f^{1,L}$ by $\mu^{y_N,\delta^2,D}_{D}$ and $\mu^{y_N,\delta^2,L}_{1}$, respectively. By Bayes' Theorem (see e.g. \cite{stuart2010inverse}), the posteriors on the final layer satisfy
\[
\frac{\mathrm{d} \mu^{y_N,\delta^2,D}_{D}}{\mathrm{d} \mu_\mathrm{prior}^D}(g) \propto \exp\left(-\frac{1}{2\delta^{2}} \|y_N - g(U_N)\|_2^2 \right), \qquad 
\frac{\mathrm{d} \mu^{y_N,\delta^2,L}_{1}}{\mathrm{d} \mu_{\mathrm{prior}}^{1,L}}(g) \propto \exp\left(-\frac{1}{2\delta^{2}} \|y_N - g(U_N)\|_2^2 \right),
\]
with $\mu_{\mathrm{prior}}^D$ and $\mu_{\mathrm{prior}}^{1,L}$ denoting the corresponding priors. We denote the resulting posterior distributions on the other layers by $\mu^{y_N,\delta^2,D}_{n}$, for $n=0,\dots, D-1$, and $\mu^{y_N,\delta^2,L}_{0}$, respectively. In the case of noise-free data, the posteriors derived in this way can be seen as a computable approximation. }

In our analysis, we will be interested in the convergence of the means $\mathbb{E}_{\mu^{y_N,\delta^2,D}_{D}}[f^{D}]$ and $\mathbb{E}_{\mu^{y_N,\delta^2,L}_{1}}[f^{1,L}]$ to the true function $f$.
In contrast to standard GP regression, the posterior distributions can no longer be found explicitly, and there is no analytical expression for these means. In practice, one typically uses Markov chain Monte Carlo (MCMC) methods for sampling from the posterior \cite{dunlop2018deep,monterrubio2020posterior}, and the means are then computed as sample averages.  

\subsection{Convergence of (constrained) DGPs in a Sobolev space}\label{ssec:conv_dgp_sob}


We have the following general convergence result for DGP regression.
\begin{assumption}\label{assump:dgp_sobolev} Let $D \geq 1$. Suppose \begin{itemize}
    \item[(i)] $f$ and $k_{D}(\cdot,\cdot;f^{D-1})$ satisfy Assumption \ref{assump:kernel2} with $\beta > d/2$
    for almost all $f^{D-1}$,
    \item[(ii)] for some \(\tilde\beta \in \N_0\) and all \( q \in \mathbb{N}_0^d \) with  \(|q|\leq\tilde\beta\), there exists a random variable \(Z\) such that \(\sup_{u\in\O}|D^{q}(f-m_{N,\delta^2}^{f}(u;f^{D-1}))|\leq Z\) almost surely with \(\E_{\mu^{y_N,\delta^2,D}_{D-1}}(Z)<\infty\),
    \item[(iii)]
     $\max \left\{ \E_{\mu^{y_N,\delta^2,D}_{D-1}}\left[C_{\low}(k_{D})^{-1}C_{\up}(k_{D})\right],  \E_{\mu^{y_N,\delta^2,D}_{D-1}}\left[C_{\up}(k_{D}) \right] \right\} \leq C_{\ref{assump:dgp_sobolev}}(y_N)$, with $C_{\ref{assump:dgp_sobolev}}(y_N) \leq C_{\ref{assump:dgp_sobolev}}^a < \infty$ if $\varepsilon=0$, $\E_\varepsilon[C_{\ref{assump:dgp_sobolev}}(y_N)^2]^{1/2} \leq C_{\ref{assump:dgp_sobolev}}^b < \infty$ if $\varepsilon \sim N(0,\delta^2\mathrm{I})$, and $C_{\ref{assump:dgp_sobolev}}^a$ and $C_{\ref{assump:dgp_sobolev}}^b$ independent of $N$ and $U_N$. 
\end{itemize}  
\end{assumption}
 
\begin{theorem}\label{thm:exp convergence rate}
Suppose Assumption \ref{assump:dgp_sobolev} holds. 
    Then, for any $\varepsilon \in \R^N$, we have 
    \begin{align*}
         \|f-\E_{\mu^{y_N,\delta^2,D}_{D}}[f^{D} ]\|_{H^{\alpha}(\O)} &\leq 
  h_{U_N,\Omega}^{\beta-\alpha} \left( 2 \E_{\mu^{y_N,\delta^2,D}_{D-1}}\left[C_{\low}(k_{D})^{-1}C_{\up}(k_{D})\right]  \|f\|_{H^\beta(\Omega)} +  \delta^{-1}  \E_{\mu^{y_N,\delta^2,D}_{D-1}}\left[C_{\up}(k_{D}) \right] \|\varepsilon\|_2 \right) \\
 & \qquad + \, h_{U_N,\Omega}^{d/2-\alpha} \left( 2 \|\varepsilon\|_2 + \delta  \E_{\mu^{y_N,\delta^2,D}_{D-1}}\left[ C_{\low}(k_{D})^{-1}C_{\up}(k_{D}) \right] \|f\|_{H^\beta(\Omega)} \right),
    \end{align*}
    for all \(\alpha\leq \min\{\tilde\beta,\lfloor\beta\rfloor\}\) and \(h_{U_N,\O}\leq h_0\), where $h_0$ is as in Theorem  \ref{thm:conv_sob_noisy}. 
    Hence, 
    \begin{enumerate}
        \item[(i)] for $\varepsilon=0$ and $\delta= \delta_N \leq C_\delta h_{U_N,\Omega}^{\beta-d/2}$, we have for all \(\alpha\leq \min\{\tilde\beta,\lfloor\beta\rfloor\}\) and \(h_{U_N,\O}\leq h_0\)
    \[
    \|f-\E_{\mu^{y_N,\delta_N^2,D}_{D}}[f^{D} ]\|_{H^{\alpha}(\O)} \leq 
  (2+C_\delta) C_{\ref{assump:dgp_sobolev}}^a h_{U_N,\Omega}^{\beta-\alpha}  \|f\|_{H^\beta(\Omega)},
    \]
    \item[(ii)] for $\varepsilon \sim N(0,\delta^2\mathrm{I})$ and $\delta= \delta_N = C_\delta N^{-1/2}$, we have for all \(\alpha\leq \min\{\tilde\beta,\lfloor\beta\rfloor\}\) and \(h_{U_N,\O}\leq h_0\)
    \begin{align*}
         \E_\varepsilon\left[ \|f-\E_{\mu^{y_N,\delta_N^2,D}_{D}}[f^{D} ]\|_{H^{\alpha}(\O)}\right] \leq 
  \left( 2 \|f\|_{H^\beta(\Omega)} +  1 \right) C_{\ref{assump:dgp_sobolev}}^b \, h_{U_N,\Omega}^{\beta-\alpha} \, N^{1/2}   + \left( 2  +  C_{\ref{assump:dgp_sobolev}}^b  \|f\|_{H^\beta(\Omega)} \right) C_\delta \, h_{U_N,\Omega}^{d/2-\alpha} .
    \end{align*}
    \end{enumerate}
    
\end{theorem}

\begin{proof}
We first consider the case $\alpha \in \N_0$. Using the definition of conditional distributions we see that
\[
\E_{\mu^{y_N,\delta^2,D}_{D}}[f^{D}] = \E_{\mu^{y_N,\delta^2,D}_{D-1}}[\E_{\mu^{y_N,\delta^2,D}_{D|D-1}}[f^{D}]]=\E_{\mu^{y_N,\delta^2,D}_{D-1}}[m_{N,\delta^2}^f(f^{D-1})], 
\]
where the final equality uses that the conditional distribution \(\mu^{y_N,\delta^2,D}_{D|D-1}\) on $f^{D}|f^{D-1},y_N$ is Gaussian with mean $m_{N,\delta^2}^f({f^{D-1}})$ as defined in \eqref{eq:gp_mean}. The linearity of expectation then gives
\begin{align*}
    \|f-\E_{\mu^{y_N,\delta^2,D}_{D}}[f^{D}]\|_{H^{\alpha}(\O)}&=\|\E_{\mu^{y_N,\delta^2,D}_{D-1}}[f - m_{N,\delta^2}^f(\cdot;{f^{D-1}})]\|_{H^{\alpha}(\O)}.
\end{align*}
To use Theorem \ref{thm:conv_sob_noisy}, we now wish to exchange the order of $\E_{\mu^{y_N,\delta^2,D}_{D-1}}$ and $\|\cdot\|_{H^{\alpha}(\O)}$. Using Assumption \ref{assump:dgp_sobolev}\((ii)\), we can conclude that we can do so by the dominated convergence theorem (see, e.g. {\cite[Theorem 1.50]{RAAdams_JJFFournier_2003}})) . Together with the convexity of $\|\cdot\|_{L^{2}(\O)}$ and Jensen's inequality, we then have 
\begin{align*}\nonumber
    \left\|\E_{\mu^{y_N,\delta^2,D}_{D-1}}[f-m_{N,\delta^2}^f(f^{D-1})]\right\|_{H^{\alpha}(\O)}
    &=\sum_{|p|\leq\alpha} \left\|\E_{\mu^{y_N,\delta^2,D}_{D-1}}[D^p(f-m_{N,\delta^2}^f(f^{D-1}))]\right\|_{L^{2}(\O)}\nonumber \\
    &\leq\E_{\mu^{y_N,\delta^2,D}_{D-1}}\left[\left\|f-m_{N,\delta^2}^f(f^{D-1})\right\|_{H^{\alpha}(\O)}\right]. 
\end{align*}
Combining the previous bound with Assumption \ref{assump:dgp_sobolev}$(i)$ and Theorem \ref{thm:conv_sob_noisy}, we have 
\begin{align*}
 \|f-\E_{\mu^{y_N,\delta^2,D}_{D}}[f^{D}] \|_{H^{\alpha}(\O)} &\leq 
  h_{U_N,\Omega}^{\beta-\alpha} \left( 2 \E_{\mu^{y_N,\delta^2,D}_{D-1}}\left[C_{\low}(k_D)^{-1}C_{\up}(k_D)\right]  \|f\|_{H^\beta(\Omega)} +  \delta^{-1}  \E_{\mu^{y_N,\delta^2,D}_{D-1}}\left[C_{\up}(k_D) \right] \|\varepsilon\|_2 \right) \\
 & \qquad + \, h_{U_N,\Omega}^{d/2-\alpha} \left( 2 \|\varepsilon\|_2 + \delta  \E_{\mu^{y_N,\delta^2,D}_{D-1}}\left[ C_{\low}(k_D)^{-1}C_{\up}(k_D) \right] \|f\|_{H^\beta(\Omega)} \right),
\end{align*}
which proves the first claim. 
The other claims then follow using Assumption \ref{assump:dgp_sobolev}$(iii)$, H\"older's inequality with exponents $p=q=1/2$ on the term $\E_\varepsilon\left[\E_{\mu^{y_N,\delta^2,D}_{D-1}}\left[C_{\up}(k_D) \right] \|\varepsilon\|_2\right]$, and the inequality $\E_\varepsilon \left[ \|\varepsilon\|_2\right] \leq \E_\varepsilon \left[ \|\varepsilon\|^2_2\right]^{1/2} = \delta$.

The case of non-integer $\alpha$ follows by interpolation theory in Sobolev spaces (see e.g. \cite{bergh2012interpolation}), by considering the linear operator $T$ defined by $T f = f-m_N^f$. The first bound applies to $\lfloor \alpha \rfloor$ and $\lceil \alpha \rceil$, and hence also to $\alpha = \lfloor \alpha \rfloor + \theta(\lceil \alpha \rceil -\lfloor \alpha \rfloor)$, for $0 < \theta < 1$.
\end{proof}

Similar to Theorem \ref{thm:conv_sob_noisy}, the first claim in Theorem \ref{thm:exp convergence rate} does not make any assumptions on the noise $\varepsilon$. The second claim, given in part $(i)$, shows that with noise-free data and the noise parameter $\delta$ in the approximate posterior $\mu^{y_N,\delta^2,D}_{D}$ chosen appropriately, we obtain optimal convergence rates $N^{-(\beta+\alpha)/d}$ of the posterior mean also for DGP regression. The third claim, given in part $(ii)$, shows that convergence of the posterior mean is guaranteed also in the case of noisy data, provided that the noise level $\delta$ goes to zero at rate $N^{-1/2}$. Note however, that the rates are slower than in the case of noise-free data. For a constant noise level $\delta$, independent of $N$, the error bound in Theorem \ref{thm:exp convergence rate} unfortunately does not allow us to conclude on the convergence of the posterior mean as $N \rightarrow \infty$, cf the discussion after Theorem \ref{thm:conv_sob_noisy}.

We now show that Assumption \ref{assump:dgp_sobolev} holds for the mixture and warping kernel DGP constructions, given the conditions specified in Lemmas \ref{lem:DGPkernelWarp} and \ref{lem:DGPkernelMix}. In order to ensure that Assumption \ref{assump:dgp_sobolev} \((ii)\) and \((iii)\) are satisfied, we impose the condition that sample paths from the \(D-1\)-th layer are restricted to appropriate H\"older or Sobolev balls denoted by \(C^{\beta(R)}(\O)\) and \(H^{\beta(R)}(\O)\) respectively for \(0<R<\infty\):
\begin{align*}
C^{\beta(R)}(\O) &:= \{ g \in C^{\beta}(\O) : \|g\|_{C^{\beta}(\O)} \leq R\}, \\
H^{\beta(R)}(\O) &:= \{ g \in H^{\beta}(\O) : \|g\|_{H^{\beta}(\O)} \leq R\}.
\end{align*}
That is, \(f^{D-1}\sim \TGP(0,k_{D-1}, R)\), where this distribution is defined as \(\GP(0,k_{D-1})\) conditioned on the event that \(\{f^{D-1}\in C^{\beta(R)}(\O)\}\) (for the warping construction) or \(\{f^{D-1}\in H^{\beta(R)}(\O)\}\) (for the mixture construction). When \(D\geq 2\) the truncated deep Gaussian process (TDGP) construction takes the following form:
\begin{align}
\nonumber f^0&\sim\mathcal{GP}(0, k_0(u,u')),\nonumber\\
f^{n}|f^{n-1}&\sim\mathcal{GP}(0, k_{n}(u,u';f^{n-1})), \qquad n=1,\dots, D-2\nonumber\\
f^{D-1}|f^{D-2}&\sim\mathcal{TGP}(0, k_{n}(u,u';f^{n-1}), R)\nonumber\\
f^{D}|f^{D-1}&\sim\mathcal{GP}(0, k_{D}(u,u';f^{D-1})).\label{eq:truncatedGP}
\end{align}
For \(D=1\) the TDGP takes a similar form, but with the initial later truncated, that is \( f^0\sim\mathcal{GP}(0, k_0(u,u'), R)\).
The statements of the results are given below, while the proofs are given at the end of this section. 
\begin{lemma}\label{lem:DGPkernelWarp}
Let $D \geq 1$, \(0<R<\infty\) and \(\O\subset\R\). Let \(\hat{\nu}_0=\nu_0\) for \(\nu_0\) not an integer and \(\hat{\nu}_0<\nu_0\) otherwise. Then Assumption \ref{assump:dgp_sobolev} is satisfied, if \(f\in H^{\beta}(\O)\) and \\
$(i)$ $k_0 = k_{\M(\hat\nu_0)}$ , where \(1/2\leq\lfloor\hat\nu_0-D/2\rfloor=:\beta\), \\
$(ii)$ $k_n = k_{\warp}^{f^{n-1},k_{\M(\nu_n)}}$ with $\nu_n=\hat\nu_0$ for \(n = 1,\ldots, D-1\) and \(\nu_{D}= \beta -1/2:=\tilde\beta\)\\
$(iii)$ \(f^{D-1}: \O \rightarrow \O\), $\inv(f^{D-1})$ exists, and \(|(f^{D-1})'| \geq c>0\) 
a.e. in $\O$, all almost surely, and \\
$(iv)$ \(f^{D-1}\in{C^{\beta(R)}(\O)}\) almost surely.
\end{lemma}

\begin{lemma}\label{lem:DGPkernelMix}
Let $D\geq 1$, \(0<R<\infty\) and \(\O\subset\R^d\). Let \(\hat{\nu}_0=\nu_0\) for \(\nu_0\) not an integer and \(\hat{\nu}_0<\nu_0\) otherwise. Then Assumption \ref{assump:dgp_sobolev} is satisfied, if \(f\in H^{\beta}(\O)\) and \\
$(i)$ $k_0 = k_{\M(\hat\nu_0)}$, where \(d/2\leq\beta(d,D)\) for \(\beta(1,D)=\lfloor \hat\nu_0-(D-1)/2\rfloor\), \(\beta(d,1)=\lfloor\hat\nu_0\rfloor\) for \(d\geq 2\) and \(\beta(d,D)=\lceil\hat\nu_0-1\rceil-(D-2)\) for \(d,D\geq 2\), \\
$(ii)$ $k_n = k_{\mix}^{F(f^{n-1}),k_{\M(\nu_n)}}$ with $\nu_n=\hat\nu_0$ for \(n = 1,\ldots, D-1\) and \(\nu_{D}=\beta-d/2:=\tilde\beta\), \\
$(iii)$ $F(x) = x^2 + \eta$, for all $x \in \R$ and some $\eta > 0$, and \\
$(iv)$ \(f^{D-1}\in H^{\beta(R)}(\O)\) almost surely.
\end{lemma}
\begin{corollary}\label{cor:DGP_conv_sobolev}
Suppose the assumptions of Lemma \ref{lem:DGPkernelWarp} or \ref{lem:DGPkernelMix} hold. Then the conclusions of Theorem \ref{thm:exp convergence rate} hold.

Suppose further that $U_N$ are sampled i.i.d. from a measure $\mu$ with density ${\rho}$ satisfying ${\rho(u)} \geq {\rho_{\mathrm{min}}} > 0$ for all $u \in \overline{\O}$. Then for $\varepsilon=0$ and $\delta= \delta_N \leq C_\delta h_{U_N,\Omega}^{\beta-d/2}$, there exists a constant $C_{\ref{cor:DGP_conv_sobolev}}$ such that for all \(\alpha\leq \min\{\tilde\beta,\lfloor\beta\rfloor\}\) and $\epsilon > 0$ we have
    \[
    \E_{\mu} \left[\|f-\E_{\mu^{y_N,\delta_N^2,D}_{D}}[f^{D} ]\|_{H^{\alpha}(\O)}\right] \leq 
  C_{\ref{cor:DGP_conv_sobolev}} N^{-\frac{\beta-\alpha}{d} + \epsilon} \|f\|_{H^\beta(\Omega)}.
    \]
    \end{corollary}
    \begin{proof} The first claim follows immediately. The second claim follows from the bound in Theorem \ref{thm:exp convergence rate}(i), together with the proof technique employed in \cite[Theorem 3]{helin2023introduction}.
    \end{proof}

In both of the previous Lemmas, we impose the condition of truncating the penultimate layer in the DGP as in \eqref{eq:truncatedGP}. In Lemma \ref{lem:DGPkernelWarp}, additional conditions are imposed: \(f^{D-1}: \O \rightarrow \O\), $\inv(f^{D-1})$ exists, and \(|(f^{D-1})'| \geq c>0\) and \(|\inv(f^{D-1})'|\geq c'>0\) a.e. in $\O$, all almost surely. This ensures that the assumptions of Theorem \ref{lemma:sob equiv warping} are satisfied, almost surely. { The requirement that \(f^{D-1}\) be a bijection is inherent in the construction of deep Gaussian processes on bounded domains and is necessary to ensure that all components are well-defined.} Without this condition, for example, \(f^{D}\) would not be well-defined on \(\O\). Furthermore, the existence of \(\inv(f^{D-1})\) is crucial for the theoretical framework of this work. However, in practical numerical simulations, this condition seems to be unnecessary to achieve the convergence rates predicted by our theory. This discrepancy is likely because the inverse exists when \(f^{D-1}\) is restricted to a finite discrete set, which is typical in numerical settings.

Note that the assumption that \(f^{D-1}: \O \rightarrow \O\) can be removed by using the scaling \[
    f^{D-1}\mapsto\frac{f^{D-1}-f^{D-1}_{\min}}{f^{D-1}_{\max}-f^{D-1}_{\min}}(b-a)+a =: \tilde f^{D-1}
\]
on the penultimate layer, where \(f^{D-1}_{\max}=\max\{f^{D-1}(u):u\in\O\}\), \(f^{D-1}_{\min}=\min\{f^{D-1}(u):u\in\O\}\) and \(\O=(a,b)\). This is well-defined, since by assumption $(iv)$ in
Lemma \ref{lem:DGPkernelWarp} we have \(f^{D-1}_{\max}, |f^{D-1}_{\min}| \leq R <\infty\), and by the mean value theorem and assumption $(iii)$ in
Lemma \ref{lem:DGPkernelWarp} (assuming \(|(f^{D-1})'| \geq c>0\) a.e. in $\overline \O$) we have
\[
|f^{D-1}_{\max}-f^{D-1}_{\min}| \geq c(b-a) >0.
\]
Then, using $k_{D} = k_{\warp}^{\tilde f^{D-1},k_{\M(\nu_n)}}$ in the TDGP construction gives the same conclusion as Lemma \ref{lem:DGPkernelWarp} under the same assumptions on $f^{D-1}$, since
\begin{align*}
\|\tilde f^{D-1}\|_{C^\beta(\O)} &\leq \|f^{D-1}\|_{C^\beta(\O)}/c + |a| + R/c =: \tilde R < \infty, \\
|(\tilde f^{D-1})'| &\geq 2(b-a)/R c =: \tilde c > 0.
\end{align*}

Various methods facilitate the conditioning required in Lemmas \ref{lem:DGPkernelWarp} and \ref{lem:DGPkernelMix}. For instance, rejection sampling, relying on constraints on derivatives, is one approach (see, e.g., \cite{Riikimake2010}). In \cite[section 7.1]{finocchio2023posterior}, where a similar truncation is essential in alternative method to construct DGPs, the authors propose several alternative approaches following a number of approaches given in \cite{Swiler_2020}. These include using a well-selected link function on the output of the penultimate layer, a form of constrained maximum likelihood estimation, and finite-dimensional approximations of Gaussian processes. For a comprehensive overview, we direct the reader to their summary. 

In Lemma \ref{lem:DGPkernelMix}, we have chosen the function \(F\) in a particular way that seems simple and practical. However, by Lemma \ref{lemma:reciprocal has derivatives}, the same conclusions hold for any function $F$ that is \(C^{\infty}(\O)\) and bounded below away from \(0\). 

Lemmas \ref{lem:DGPkernelWarp} and \ref{lem:DGPkernelMix}  unfortunately do not apply to the setting considered in \cite{dunlop2018deep}, where the same covariance kernel $k_n \equiv k$ is used in all layers of the construction. This is due to the fact that in Theorems \ref{lemma:sob equiv warping} and \ref{lemma: sob equiv mixture}, we require the same regularity of the stationary kernel $k_s$ and the hyper-parameters $w$ or $\{\sigma_\ell\}_{\ell=1}^L$. Hence, the difference in regularity of sample paths and the RKHS of a Gaussian process poses an issue. We can take the approach as presented in Lemmas \ref{lem:DGPkernelWarp} and \ref{lem:DGPkernelMix}, where we correct for this on the last layer only and keep the same kernel $k_n \equiv k$ on layers $n=0, \dots, D-1$, or we can correct for this gradually and adapt the kernel $k_n$ on each layer. We have chosen the former approach, since it mimics more closely the set-up in \cite{dunlop2018deep} and the corresponding convergence results as $D \rightarrow \infty$ may hence still apply. In particular, \cite{dunlop2018deep} gives convergence to a limiting distribution of the penultimate layer $f^{D-1}$ in Lemmas \ref{lem:DGPkernelWarp} and \ref{lem:DGPkernelMix} as $D \rightarrow \infty$.

\begin{remark}\label{rmk:sub-opt DGP} In Lemmas \ref{lem:DGPkernelWarp} and \ref{lem:DGPkernelMix}, there are three places where our proof technique may lead to sub-optimal regularity, and hence sub-optimal convergence rates in Corollary \ref{cor:DGP_conv_sobolev}. Firstly, as already discussed in Remark \ref{rem:sob_vs_class}, the H\"older sample path regularity proved in Corollaries \ref{cor:sample path old school} and \ref{cor: sample path sullivan} may not be optimal. Similarly, since we know that $k_{\M(\nu)}(\cdot,u)\in H^{\nu+d/2}(\O)$, the regularity $k_{\M(\nu)}(\cdot,u)\in C^{\nu}(\O)$ of the Mat\'ern kernel provided by Proposition \ref{prop:regularity of matern kernel lemma 11} may not be sharp. Lastly, there is the restriction that Lemma \ref{lemma:reciprocal has derivatives} has only been proven for integer-valued $\beta$.  
It may be possible to sharpen the sub-optimal results and improve the obtained convergence rates, however, this is beyond the scope of this work.
\end{remark}

The remainder of this section is devoted to the proofs of Lemmas \ref{lem:DGPkernelWarp} and \ref{lem:DGPkernelMix}. 

\begin{corollary}\label{eq:F(f0) is Hbeta}
    Suppose we are in the set up of Proposition \ref{prop:banachalg} and \(F:[\eta,\infty)\to\R\) is defined by \(F(x)=x^2+\eta\) for \(\eta>0\). Then for \(\s\in H^{\beta}(\O)\) there exists a constant \(C_{\ref{eq:F(f0) is Hbeta}}\) such that
    \[
        \|F(\s)\|_{H^{\alpha}(\O)}\leq C_{\ref{eq:F(f0) is Hbeta}}(1+\|\s\|^2_{H^{\alpha}(\O)}).
    \]
\end{corollary}
\begin{proof}
By Proposition \ref{prop:banachalg} we see that
            \[
            \|F(\s)\|_{H^{\alpha}(\O)}=\|\s^2+\eta\|_{H^{\alpha}(\O)}
            \leq\|\s^2\|_{H^{\alpha}(\O)}+\eta|\O|\leq C_{\ref{prop:banachalg}} \|\s\|^2_{H^{\alpha}(\O)}+\eta|\O|.
        \]
\end{proof}

{
\begin{proof}[Of Lemma \ref{lem:DGPkernelWarp}] We work through parts $(i)$-$(iii)$ in Assumption \ref{assump:dgp_sobolev}. Note that the sample path regularity of the TDGP posterior is the same as the TDGP prior, since the posterior is absolutely continuous with respect to the prior by Bayes' theorem. 

    \textbf{Assumption \ref{assump:dgp_sobolev} (i):} Setting \(w\equiv 1\) in Corollary \ref{cor:sample path old school} shows that \(f^{0}\in C^{\hat\nu_0-1/2}(\O)\). Since
    \[
        k_1=k_{\warp}^{f^0,k_{\M(\hat\nu_0)}}
    \]
    another application of Corollary \ref{cor:sample path old school} implies that \(f^{1}\in C^{\hat\nu_0-1}(\O)\).
        Using this argument inductively we observe that \(f^{n}\in C^{\hat\nu_0-(n+1)/2}(\O)\) for \(n=1,\ldots,D-2\).  Particularly, \(f^{D-1}\in C^{\hat\nu_0-D/2}(\O)\) since the truncation of the \(D-1\)-th layer has no impact on the sample path regularity. By assumption $(iii)$ $\inv(f^{D-1})$ exists almost surely, and by Lemma \ref{lemma:inverse has derivatives} we then have \(\inv(f^{D-1})\in C^{\hat\nu_0-D/2}(\O)\). Finally, since \(C^{\hat\nu_0-D/2}(\O)\subset  C^{\lfloor\hat\nu_0-D/2\rfloor}(\O)\) we apply Theorem \ref{lemma:sob equiv warping} with \(\nu_{D}= \lfloor\hat\nu_0-D/2\rfloor -1/2\) and \(\beta = \lfloor\hat\nu_0-D/2\rfloor\) to observe that for almost all \(f^{D-1}\) we have
        \[
            \H_{k_{\warp}^{f^{D-1},k_{\M(\nu_{D})}}}(\O) \cong H^\beta{(\O)}.
        \]
        In the last step, we have used assumptions $(iii)$ and $(iv)$ to verify the assumptions of Theorem \ref{lemma:sob equiv warping}. Note in particular that the inverse function theorem gives \(|\inv(f^{D-1})'|\geq 1/R >0\). 
        
        \textbf{Assumption \ref{assump:dgp_sobolev} (ii):} Let \(q\in\N_0\) with \(q\leq\tilde\beta= \beta-1/2\). Note that \(k_{\warp}^{f^{D-1}, k_{\M(\nu_{D})}}\in C^{2\tilde\beta}(\O\times\O)\) by Proposition \ref{prop:regularity of matern kernel lemma 11} and the chain rule. Hence, \(k_{\warp}^{f^{D-1}, k_{\M(\nu_{D})}}\in C^{\tilde \beta,\tilde \beta}(\O\times\O)\) as well. Observe that by \eqref{eq:gp_mean}, we can write \(m_{N,\delta^2}^{f}(u;f^{D-1}))\) as a linear combination of kernel functions. Let \(a_1,\ldots,a_N\in\R\) denote the coefficients in this linear combination, and let \(a_{\max}\) represent their maximum absolute value. Using the chain rule, we observe that there exits a constant \(C(k_{\M(\nu_{D})})>0\) such that for any \(u\in\O\) 
        \begin{align*}
            \bigg|D^{q}(f-&m_{N,\delta^2}^{f}(u;f^{D-1}))\bigg|=\left|D^q(f(u)-\sum_{i=1}^{N}a_ik_{\warp}^{f^{D-1}, k_{\M(\nu_{D})}}(u_i,u))\right|\\
            &\leq \|f\|_{C^{q}(\O)}+\hat{a}N\left\|k_{\warp}^{f^{D-1}, k_{\M(\nu_{D})}}\right\|_{C^{q}(\O\times\O)}\leq \|f\|_{C^{q}(\O)}+a_{\max}NC(k_{\M(\nu_{D})})\left\|f^{D-1}\right\|_{C^{q}(\O)}.
        \end{align*}
         Since \(f\in H^{\beta}(\O)\hookrightarrow C^{q}(\O)\) by the Sobolev embedding theorem (see, e.g. {\cite[Theorem 4.12]{RAAdams_JJFFournier_2003}}) and \(\|f^{D-1}\|_{C^{q}(\O)}\leq\|f^{D-1}\|_{C^{\beta}(\O)}\leq R\) almost surely, we observe that
        \[
            \E_{\mu^{y_N,\delta^2,D}_{D-1}}\left[\|f\|_{
            C^{q}(\O)}+a_{\max}NC(k_{\M(\nu_{D})})\left\|f^{D-1}\right\|_{C^{q}(\O)}\right]\leq \|f\|_{H^{\beta}(\O)}+a_{\max}NC(k_{\M(\nu_{D})})R<\infty.
        \]
     \textbf{Assumption \ref{assump:dgp_sobolev} (iii):} Using Theorem \ref{lemma:sob equiv warping} and H\"older's inequality we see that
        \begin{align*}
        &\E_{\mu^{y_N,\delta^2,D}_{D-1}}\left[C_{\low}(k_{D})^{-1}C_{\up}(k_{D})\right]\nonumber\\
        &= \E_{\mu^{y_N,\delta^2,D}_{D-1}}\Big[(C_\low(k_{\M(\nu_{D})}))^{-1}C(\beta)\frac{1}{\sqrt{c'}}\max\left\{1,\|\inv(f^{D-1})\|_{C^{\beta}(\O)}^{\beta} \right\}\nonumber\\
        &\hspace{3cm}
    C_\up(k_{\M(\nu_{D})}) C(\beta)\frac{1}{\sqrt{c}}\max\left\{1,\|f^{D-1}\|_{C^{\beta}(\O)}^{\beta} \right\}\Big]\nonumber\\
    &\leq \frac{1}{\sqrt{cc'}}C(\beta)^{^2}(C_\low(k_{\M(\nu_{D})}))^{-1}
    C_\up(k_{\M(\nu_{D})})\E_{\mu^{y_N,\delta^2,D}_{D-1}}\left[(1+\|\inv(f^{D-1})\|_{C^{\beta}(\O)}^{\beta} )(1+\|f^{D-1}\|_{C^{\beta}(\O)}^{\beta} )\right]\nonumber\\
    &\leq \frac{1}{\sqrt{cc'}}C(\beta)^{^2}(C_\low(k_{\M(\nu_{D})}))^{-1}
    C_\up(k_{\M(\nu_{D})})(1+R^\beta)^2.\label{eq: warping dgp 1st bound}
        \end{align*}
    Similarly,
    \begin{equation*}\label{eq: warping dgp 2nd bound}
        \E_{\mu^{y_N,\delta^2,D}_{D-1}}\left[C_{\up}(k_{D})\right]\leq \frac{1}{\sqrt{c}}C(\beta)
    C_\up(k_{\M(\nu_{D})})(1+R^\beta).
    \end{equation*}
    The claim then follows, since the bounds derived above have no dependence on \(\varepsilon\), \(N\) or \(U_N\).

\end{proof}

\begin{proof}[Of Lemma \ref{lem:DGPkernelMix}] We work through parts $(i)$-$(iii)$ in Assumption \ref{assump:dgp_sobolev}. Note that the sample path regularity of the TDGP posterior is the same as the TDGP prior, since the posterior is absolutely continuous with respect to the prior by Bayes' theorem. 

        \textbf{Assumption \ref{assump:dgp_sobolev} (i):}
        We show that for all \(d, D\) and \(\beta(d,D)\), we have that for almost all \(f^{D-1}\)
         \begin{equation*}\label{eq: DGP mixture equiv spaces}
            \H_{k_{\mix}^{F(f^{D-1}),k_{\M(\nu_{D})}}}(\O) \cong H^\beta{(\O)}.
         \end{equation*}
         
        \textbf{Case 1: \(d=1\).}
        By setting with \(\s_1\equiv 1\) we use Corollary \ref{cor:sample path old school} to see that \(f^{0}\in C^{\hat \nu_0-1/2}(\O)\). By the chain rule, we also have  \(F(f^{0})\in C^{\hat\nu_0-1/2}(\O)\). Then since
    \[
        k_1=k_{\mix}^{F(f^0),k_{\M(\hat\nu_0)}}
    \]
    another application of Corollary \ref{cor:sample path old school} and the chain rule shows that \(F(f^{1})\in C^{\hat\nu_0-1}(\O)\).
        Using this argument inductively we see that \(F(f^{n})\in C^{\hat\nu_0-(n+1)/2}(\O)\) for \(n=1,\ldots,D-2\) and in particular \(F(f^{D-1})\in C^{\hat\nu_0-D/2}(\O)\).
        Note that \(f^{D-1}\in H^{\hat\nu_0-(D-1)/2}(\O)\) by Proposition \ref{prop:regularity of GP samples}, and so \(F(f^{D-1})\in H^{\lfloor\hat\nu_0-(D-1)/2\rfloor}(\O)\) by Corollary \ref{eq:F(f0) is Hbeta}. Additionally, by Lemma \ref{lemma:reciprocal has derivatives} we also have that \(1/F(f^{D-1})\in H^{\lfloor\hat\nu_0-(D-1)/2\rfloor}(\O)\). Using these facts and Theorem \ref{lemma: sob equiv mixture} with \(L=1\), we see that for \(\beta = \lfloor\hat\nu_0-(D-1)/2\rfloor\) and \(\nu_{D}=\lfloor\hat\nu_0-(D-1)/2\rfloor-1/2\)
        the required norm equivalence is achieved.

        \textbf{Case 2a: \(d\geq 2\), \(D=1\).}
         By Proposition \ref{prop:regularity of GP samples} we have \(f^{0}\in H^{\hat\nu_0}(\O)\subset H^{\lfloor\hat\nu_0\rfloor}(\O)\). Since \(F(f^{0})\in H^{\lfloor\hat\nu_0\rfloor}(\O)\) by Corollary \ref{eq:F(f0) is Hbeta}, an application of Theorem \ref{lemma: sob equiv mixture} with \(\beta = \lfloor\hat\nu_0\rfloor\) and \(\nu_1 = \beta-d/2\) gives the desired result.
         
         \textbf{Case 2b: \(d\geq 2\), \(D\geq 2\).}
         Setting \(\s_1\equiv 1\) in Corollary \ref{cor: sample path sullivan}, we observe that \(f^{0}\in C^{\lceil\hat\nu_0-1\rceil}(\O)\) almost surely. Noting that \(F(f^{0})\in C^{\lfloor\hat\nu_0\rfloor}(\O)\) almost surely and that \(\lceil\lceil\hat\nu_0-1\rceil-1\rceil=\lceil\hat\nu_0-1\rceil-1\), another application of Corollary \ref{cor: sample path sullivan} shows \(f^{1}\in C^{\lceil\hat\nu_0-1\rceil-1}(\O)\) almost surely.
         Using this argument inductively, we deduce that \(F(f^{n})\in C^{\lceil\hat\nu_0-1\rceil-n}(\O)\) almost surely for \(n=1,\ldots, D-2\). 
         In particular, since \(k_{D-1}=k_{\mix}^{F(f^{D-2}), k_{\M_{\hat\nu_{D-1}}}}\in C^{\lceil\hat\nu_0-1\rceil-{(D-2)},\lceil\hat\nu_0-1\rceil-{(D-2)}}(\O\times\O)\), Proposition \ref{prop:regularity of GP samples} shows that \(f^{D-1}\in H^{\lceil\hat\nu_0-1\rceil-{(D-2)}}(\O)\).
        Hence, by applying Theorem \ref{lemma: sob equiv mixture} with  \(\nu_D=\lceil\hat\nu_0-1\rceil-{(D-2)}-d/2\) and \(\beta = \lceil\hat\nu_0-1\rceil-{(D-2)}\), the claim follows.

        \textbf{Assumption \ref{assump:dgp_sobolev} (ii):}
        
        Let \(q\in\N_0^d\) with \(|q|\leq\tilde\beta=\beta-d/2\). As in the proof of Lemma \ref{lem:DGPkernelWarp} let \(a_{\max}\) be the maximum in absolute value of the coefficients \(a_1,\ldots,a_N\) in the linear combination representation of \(m_{N,\delta^2}^{f}(u;f^{D-1})\). Again, notice that \(k_{\mix}^{F(f^{D-1}),k_{\M(\nu_{D})}}(u_i,u)\in C^{2\tilde\beta}(\O\times\O)\subset C^{2|q|}(\O\times\O) \).
        Using the product rule and the chain rule we see that there exists a constant \(C(k_{\M(\nu_{D})},F)\) such that for any \(u\in\O\) 
        \begin{align*}
            \big|D^{q}&(f(u)-m_{N,\delta^2}^{f}(u;f^{D-1}))\big|=\left|D^q(f(u)-\sum_{i=1}^{N}a_ik_{\mix}^{F(f^{D-1}),k_{\M(\nu_{D})}}(u_i,u))\right|\\
            &\leq \|f\|_{C^{|q|}(\O)}+a_{\max}N\left\|k_{\mix}^{F(f^{D-1}),k_{\M(\nu_{D})}}\right\|_{C^{|q|}(\O\times\O)}\leq \|f\|_{C^{|q|}(\O)}+a_{\max}NC(k_{\M(\nu_{D})}, F)\left\|f^{D-1}\right\|^2_{C^{|q|}(\O)}.
        \end{align*}
         Given that \(f\in H^{\beta}(\O)\hookrightarrow C^{|q|}(\O)\) by the Sobolev embedding theorem (see, e.g. {\cite[Theorem 4.12]{RAAdams_JJFFournier_2003}}) and \(\|f^{D-1}\|_{H^{\beta}(\O)}\leq R\) almost surely, the claim follows since
        \[
            \E_{\mu^{y_N,\delta^2,D}_{D-1}}\left[ \|f\|_{C^{|q|}(\O)}+a_{\max}NC(k_{\M(\nu_{D})}, F)\left\|f^{D-1}\right\|^2_{C^{|q|}(\O)}\right]\leq \|f\|_{H^{\beta}(\O)}+a_{\max}NC(k_{\M(\nu_{D})}, F)R^2<\infty.
        \]

        \textbf{Assumption \ref{assump:dgp_sobolev} (iii):}
        Using Theorem \ref{lemma: sob equiv mixture} and Lemma \ref{lemma:reciprocal has derivatives} we see that
        \begin{align*}
        \mathbb{E}_{\mu^{y_N,\delta^2,D}_{D-1}}&\left[C_\low(k_{D}(\cdot,\cdot;f^{D-1}))^{-1} C_\up(k_{D}(\cdot,\cdot;f^{D-1}))\right]\nonumber\\ 
        &=\sqrt{2}  (C_\low(k_{\M(\nu_{D})}))^{-1}  C_\up(k_{\M(\nu_{D})})\E_{\mu^{y_N,\delta^2,D}_{D-1}}\left[\|1/F(f^{D-1}))\|_{H^{\beta}(\O)} \|F(f^{D-1})\|_{H^{\beta}(\O)}\right]\nonumber\\
        &\leq\sqrt{2}  C_{\ref{lemma:reciprocal has derivatives}}'C_{\ref{lemma:reciprocal has derivatives}}''(C_\low(k_{\M(\nu_{D})}))^{-1}  C_\up(k_{\M(\nu_{D})})\E_{\mu^{y_N,\delta^2,D}_{D-1}}\left[ \left(1+\|f^{D-1}\|^{\beta}_{H^{\beta}(\O)}\right)^2\right]\nonumber\\
        &\leq \sqrt{2}  C_{\ref{lemma:reciprocal has derivatives}}'C_{\ref{lemma:reciprocal has derivatives}}''(C_\low(k_{\M(\nu_{D})}))^{-1}  C_\up(k_{\M(\nu_{D})}) (1+R^\beta)^2.\label{eq: mix dgp 1st bound}
        \end{align*}
        Similarly,
        \begin{equation*}\label{eq: mix dgp 2nd bound}
            \mathbb{E}_{\mu^{y_N,\delta^2,D}_{D-1}}\left[C_\up(k_{D}(\cdot,\cdot;f^{D-1}))\right]\leq\sqrt{2}C_{\ref{lemma:reciprocal has derivatives}}'C_\up(k_{\M(\nu_{D})}) (1+R^\beta).
        \end{equation*}
        The claim then follows, since the bounds derived above have no dependence on \(\varepsilon\), \(N\) or \(U_N\).
        \end{proof}

        For the second result, using the definition of the posterior mean \eqref{eq:gp_mean}, Lemma \ref{lemma:exp_kernel mixture}, the Sobolev embedding theorem (see, e.g. {\cite[Theorem 4.12]{RAAdams_JJFFournier_2003}}), Corollary \ref{eq:F(f0) is Hbeta} to see that and, since \(\beta\geq |p|+d/2\), the inequality \(\|\cdot\|_{H^{|p|+d/2}(\O)}\leq\|\cdot\|_{H^{\beta}(\O)}\), we have, for \(u\in\O\)
        \begin{align*}
            \big|D^{p}(f(u)-m_{N,\delta^2}^{f}&(u;f^{D-2}))\big|=\left|D^p(f(u)-\sum_{i=1}^{N}k_{\mix}^{F(f^{D-2}),k_{\M(\nu_{D-1})}}(u_i,u))\right|\\
            &\leq \|f\|_{C^{|p|}(\O)}+NC_{\ref{thm:Wend11.13}}\left\|F(f^{D-2})\right\|^2_{C^{|p|}(\O)}
            \leq \|f\|_{C^{|p|}(\O)}+NCC_{\ref{thm:Wend11.13}}\left\|F(f^{D-2})\right\|^2_{H^{|p|+d/2}(\O)}\\
            &\leq \|f\|_{C^{|p|}(\O)}+NCC_{\ref{thm:Wend11.13}}\left\|f^{D-2}\right\|^2_{H^{|p|+d/2}(\O)}
            \leq\|f\|_{C^{|p|}(\O)}+ NCC_{\ref{thm:Wend11.13}}\left\|f^{D-2}\right\|^2_{H^{\beta}(\O)}.
        \end{align*}
        Since \(f\in C^{|p|}(\O)\) is the true function we can then use Lemma \ref{lem:Expected val of f^d is bounded} to see that
        \[
            \E_{\mu_0}\left[\|f\|_{C^{|p|}(\O)}+NCC_{\ref{thm:Wend11.13}}\left\|f^{D-2}\right\|^2_{H^{\beta}(\O)}\right]\leq\|f\|_{C^{|p|}(\O)}+\E_{\mu_0}\left[NCC_{\ref{thm:Wend11.13}}\left\|f^{D-2}\right\|^2_{H^{\beta}(\O)}\right]<\infty.
        \]
        For the final result, we use Theorem \ref{lemma: sob equiv mixture}, Lemma \ref{lemma:reciprocal has derivatives} and the linearity of expectation to see that
        \begin{align*}
        \mathbb{E}_{\mu_0}&\left[C_\low(k_{D-1}(\cdot,\cdot;f^{D-2}))^{-1} C_\up(k_{D-1}(\cdot,\cdot;f^{D-2}))\right]\\ 
        &=\sqrt{2}(C_{\ref{prop:banachalg}})^2 (C_\low(k_{\M(\nu_{D-1})}))^{-1}  C_\up(k_{\M(\nu_{D-1})})\E_{\mu_0}\left[\|1/F(f^{D-2}))\|_{H^{\beta}} \|F(f^{D-2})\|_{H^{\beta}(\O)}\right]\\
        &\leq\sqrt{2}(C_{\ref{prop:banachalg}})^2 C_{\ref{lemma:reciprocal has derivatives}}'C_{\ref{lemma:reciprocal has derivatives}}''(C_\low(k_{\M(\nu_{D-1})}))^{-1}  C_\up(k_{\M(\nu_{D-1})})\E_{\mu_0}\left[ \left(1+\|f^{D-2}\|^{\beta}_{H^{\beta}(\O)}\right)^2\right]
        <\infty,
        \end{align*}
        
    where the expectation in the final bound is finite by Lemma \ref{lem:Expected val of f^d is bounded}.
}
\subsection{Convergence of (constrained) WGPs in a Sobolev space}

In this section, we consider WGP priors. The results are very similar to those in section \ref{ssec:conv_dgp_sob}, and in particular, we have the following equivalents of Assumption \ref{assump:dgp_sobolev} and Theorem \ref{thm:exp convergence rate}. 

\begin{assumption}\label{assump:wgp_sobolev} Let $L \geq 1$. Suppose \begin{itemize}
    \item[(i)] $f$ and $k_{1}(u,u'; \{f^{0,\ell}\}_{\ell=1}^L)$ satisfy Assumption \ref{assump:kernel2} with $\beta > d/2$
    for almost all $\{f^{0,\ell}\}_{\ell=1}^L$,
    \item[(ii)] for some $\tilde \beta \in \N_0$ and all \( q \in \mathbb{N}_0^d \) with  \(|q|\leq\tilde \beta\) there exists a random variable \(Z\) such that \(\sup_{u\in\O}|D^{q}(f-m_{N,\delta^2}^{f}(u;\{f^{0,\ell}\}_{\ell=1}^L))|\leq Z\) almost surely with \(\E_{\mu^{y_N,\delta^2,L}_{0}}(Z)<\infty\),
    \item[(iii)]
     $\max \left\{ \E_{\mu^{y_N,\delta^2,L}_{0}}\left[C_{\low}(k_1)^{-1}C_{\up}(k_{1})\right],  \E_{\mu^{y_N,\delta^2,L}_{0}}\left[C_{\up}(k_{1}) \right] \right\} \leq C_{\ref{assump:wgp_sobolev}}(y_N)$, with $C_{\ref{assump:wgp_sobolev}}(y_N) \leq C_{\ref{assump:wgp_sobolev}}^a < \infty$ if $\varepsilon=0$,  $\E_\varepsilon[C_{\ref{assump:wgp_sobolev}}(y_N)^2]^{1/2} \leq C_{\ref{assump:wgp_sobolev}}^b < \infty$ if $\varepsilon \sim N(0,\delta^2\mathrm{I})$, and $C_{\ref{assump:wgp_sobolev}}^a$ and $C_{\ref{assump:wgp_sobolev}}^b$ independent of $N$ and $U_N$.
\end{itemize}  
\end{assumption}
 
\begin{theorem}\label{thm:exp convergence rate_wgp}
Suppose Assumption \ref{assump:wgp_sobolev} holds. 
    Then, for any $\varepsilon \in \R^N$, we have 
    \begin{align*}
         \|f-\E_{\mu^{y_N,\delta^2,L}_{1}}[f^{1,L} ]\|_{H^{\alpha}(\O)} &\leq 
  h_{U_N,\Omega}^{\beta-\alpha} \left( 2 \E_{\mu^{y_N,\delta^2,L}_{0}}\left[C_{\low}(k_{1})^{-1}C_{\up}(k_{1})\right]  \|f\|_{H^\beta(\Omega)} +  \delta^{-1}  \E_{\mu^{y_N,\delta^2,L}_{0}}\left[C_{\up}(k_{1}) \right] \|\varepsilon\|_2 \right) \\
 & \qquad + \, h_{U_N,\Omega}^{d/2-\alpha} \left( 2 \|\varepsilon\|_2 + \delta  \E_{\mu^{y_N,\delta^2,L}_{0}}\left[ C_{\low}(k_{1})^{-1}C_{\up}(k_{1}) \right] \|f\|_{H^\beta(\Omega)} \right),
    \end{align*}
    for all \(\alpha\leq \lfloor \beta \rfloor\) and \(h_{U_N,\O}\leq h_0\), where $h_0$ is as in Theorem  \ref{thm:conv_sob_noisy}. 
    Hence, 
    \begin{enumerate}
        \item[(i)] for $\varepsilon=0$ and $\delta= \delta_N \leq C_\delta h_{U_N,\Omega}^{\beta-d/2}$, we have for all \(\alpha\leq \lfloor \beta \rfloor\) and \(h_{U_N,\O}\leq h_0\)
    \[
    \|f-\E_{\mu^{y_N,\delta_N^2,L}_{1}}[f^{1,L} ]\|_{H^{\alpha}(\O)} \leq 
  (2+C_\delta) C_{\ref{assump:wgp_sobolev}}^a h_{U_N,\Omega}^{\beta-\alpha}  \|f\|_{H^\beta(\Omega)},
    \]
    \item[(ii)] for $\varepsilon \sim N(0,\delta^2\mathrm{I})$ and $\delta= \delta_N = C_\delta N^{-1/2}$, we have for all \(\alpha\leq \lfloor \beta \rfloor\) and \(h_{U_N,\O}\leq h_0\)
    \begin{align*}
         \E_\varepsilon\left[ \|f-\E_{\mu^{y_N,\delta_N^2,L}_{1}}[f^{1,L} ]\|_{H^{\alpha}(\O)}\right] \leq 
  \left( 2 \|f\|_{H^\beta(\Omega)} +  1 \right) C_{\ref{assump:wgp_sobolev}}^b \, h_{U_N,\Omega}^{\beta-\alpha} \, N^{1/2}   + \left( 2  +  C_{\ref{assump:wgp_sobolev}}^b  \|f\|_{H^\beta(\Omega)} \right) C_\delta \, h_{U_N,\Omega}^{d/2-\alpha} .
    \end{align*}
    \end{enumerate}
    
\end{theorem}

The proof of Theorem \ref{thm:exp convergence rate_wgp} is identical to that of Theorem \ref{thm:exp convergence rate} and is omitted for brevity. As for DGP regression, we observe that we obtain optimal rates of convergence for noise-free data, and that convergence is guaranteed for noisy data only in the setting where the level of the noise goes to 0 as $N \rightarrow \infty$.
The following Lemma gives conditions under which Assumption \ref{assump:wgp_sobolev} holds for the mixture kernel. 
As with the DGP construction, in order to ensure Assumption \ref{assump:wgp_sobolev} \((ii)\) and \((iii)\) we explicitly impose a truncation on the initial (and penultimate) layer of the WGP. That is, we impose that \({f^{0,\ell}}\in H^{\beta(R)}(\O)\) almost surely for \(\ell=1,\ldots,L\) and define the truncated wide Gaussian process (TWGP)
    \begin{align*}
        f^{0,\ell}&\sim\GP(0,k_{0}(u,u'), R)\quad\text{ for }\ell=1,\ldots, L, \quad i.i.d., \nonumber\\
        f^{1,L}|\{f^{0,\ell}\}_{\ell=1}^{L}&\sim\GP(0,k_{1}(u,u'; \{f^{0,\ell}\}_{\ell=1}^L)).
    \end{align*}
\begin{lemma}\label{lem:DGPWideGP}
    Let \(L\geq 1, 0<R<\infty,\nu_0\geq d/2\), \(\O\in\R^d\) and \(\beta=\lceil\nu_0-1\rceil\).     
    Then Assumption \ref{assump:wgp_sobolev} is satisfied, if \(f\in H^{\beta}(\O)\) and \\
    $(i)$ $k_{0} = k_{\M(\nu_{0})}$, \\
    $(ii)$ \(k_1=k_{\mix}^{\{F(f^{0,\ell}), \, k_{\M(\nu_1)}\}_{\ell=1}^{L}}\) subject to the constraint that $\nu_1=\beta-d/2=:\tilde\beta$, \\
    $(iii)$ $F(x) = x^2 + \eta$, for all $x \in \R$ and some $\eta > 0$, and\\
    $(iv)$ \({f^{0,\ell}}\in H^{\beta(R)}(\O)\) almost surely for \(\ell=1,\ldots,L\).
\end{lemma}
{
In Lemma \ref{lem:DGPWideGP}, we have chosen the function \(F\) in a particular way that seems simple and practical. However, by Lemma \ref{lemma:reciprocal has derivatives}, the same conclusions hold for any function $F$ that is \(C^{\infty}(\O)\) and bounded below away from \(0\). In fact, we could further choose different function $F^\ell$, with only one $F^j$ being bounded away from zero. 
}

\begin{corollary}\label{cor:wideGP}
Suppose the assumptions of Lemma \ref{lem:DGPWideGP} hold. Then the conclusions of Theorem \ref{thm:exp convergence rate_wgp} hold.

{Suppose further that $U_N$ are sampled i.i.d. from a measure $\mu$ with density ${\rho}$ satisfying ${\rho(u)} \geq {\rho_{\mathrm{min}}} > 0$ for all $u \in \overline{\O}$. Then for $\varepsilon=0$ and $\delta= \delta_N \leq C_\delta h_{U_N,\Omega}^{\beta-d/2}$, there exists a constant $C_{\ref{cor:wideGP}}$ such that for all \(\alpha\leq \min\{\tilde\beta,\lfloor\beta\rfloor\}\) and $\epsilon > 0$ we have
    \[
    \E_{\mu} \left[\|f-\E_{\mu^{y_N,\delta_N^2,L}_{1}}[f^{1,L} ]\|_{H^{\alpha}(\O)}\right] \leq 
  C_{\ref{cor:wideGP}} N^{-\frac{\beta-\alpha}{d} + \epsilon} \|f\|_{H^\beta(\Omega)}.
    \]}
    \end{corollary}
    {
    \begin{proof} The first claim follows immediately. The second claim follows from the bound in Theorem \ref{thm:exp convergence rate_wgp}(i), together with the proof technique employed in \cite[Theorem 3]{helin2023introduction}.
    \end{proof}
    }

\begin{proof}[Proof of Lemma \ref{lem:DGPWideGP}] 
We work through parts $(i)$-$(iii)$ in Assumption \ref{assump:wgp_sobolev}. Note that the sample path regularity of the TWGP posterior is the same as the TWGP prior, since the posterior is absolutely continuous with respect to the prior by Bayes' theorem. 

     \textbf{Assumption \ref{assump:wgp_sobolev} (i):}
     Propositions \ref{prop:regularity of matern kernel lemma 11}, the product rule and \ref{prop:regularity of GP samples} show that \(f^{0,\ell}\in H^{\beta}(\O)\) almost surely, for \(\ell=1,\ldots,L\). Moreover, Lemma \ref{lemma:reciprocal has derivatives} ensures that \(F(f^{0,\ell})\in H^{\beta}(\O)\) and \(1/F(f^{0,\ell})\in H^{\beta}(\O)\).
    With $\nu_1=\beta-d/2$, Theorem \ref{lemma: sob equiv mixture} implies that, for almost all \(\{f^{0,\ell}\}_{\ell=1}^{L}\) we have
    \[
        \H_{k_{\mix}^{\{F(f^{0,\ell}),k_{\M(\nu_1)}\}_{\ell=1}^{L}}}(\O) \cong H^\beta{(\O)}.
    \]

    \textbf{Assumption \ref{assump:wgp_sobolev} (ii):}
     Let \(q\in\N_0^d\) with \(|q|\leq\tilde{\beta}=\beta-d/2\). Let \(a_{\max}\) be the maximum in absolute value of the coefficients \(a_1,\ldots,a_N\) in the linear combination representation of \(m_{N,\delta^2}^{f}(u;f^{D-1})\). By the Sobolev embedding Theorem, we have \(F(f^{0,\ell})\in H^{\beta}(\O)\hookrightarrow C^{|q|}(\O)\), and so \(k_{\mix}^{\{F(f^{0,\ell}),k_{\M(\nu_1)}\}_{\ell=1}^{L}}\in C^{\tilde{\beta},\tilde{\beta}}(\O\times\O)\).
     By the product rule and the chain rule we see that there exists a  constant \(C(\{k_{\M(\nu_1)}\}_{\ell=1}^{L},F)>0\) such that for any \(u\in\O\) 
        \begin{align*}
            &\left|D^{q}(f-m_{N,\delta^2}^{f}(u;f^{D-2}))\right|=\left|D^q(f-\sum_{i=1}^{N}a_ik_{\mix}^{F(f^{0,\ell}),k_{\M(\nu_1)}}(u_i,u))\right|\\
            &\leq \|f\|_{C^{|q|}(\O)}+a_{\max}N\max_{1\leq\ell\leq L}
            \left\|k_{\mix}^{F(f^{0,\ell}),k_{\M(\nu_1)}}\right\|^2_{C^{|q|}(\O\times\O)}\leq \|f\|_{C^{|q|}(\O)}+a_{\max}N\sum_{\ell=1}^{L}
            \left\|k_{\mix}^{F(f^{0,\ell}),k_{\M(\nu_1)}}\right\|^2_{C^{|q|}(\O\times\O)}\\
            &\leq\|f\|_{C^{|q|}(\O)}+a_{\max}NC(\{k_{\M(\nu_1)}\}_{\ell=1}^{L},F)\sum_{\ell=1}^{L}
            \left\|f^{0,\ell}\right\|^2_{C^{|q|}(\O)}.
        \end{align*}
         Since \(f\in H^{\beta}(\O)\)
        and \(f^{D-1}\in {H^{\beta(R)}(\O)}\) almost surely we see that
        \begin{align*}
            \E_{\mu^{y_N,\delta^2,L}_{0}}&\left[\|f\|_{C^{|q|}(\O)}+a_{\max}NC(\{k_{\M(\nu_1)}\}_{\ell=1}^{L},F)\sum_{\ell=1}^{L}
            \left\|f^{0,\ell}\right\|^2_{C^{|q|}(\O)}\right]\\ 
            &\hspace{0.5cm}\leq\|f\|_{H^{\beta}(\O)}+a_{\max}NC(\{k_{\M(\nu_1)}\}_{\ell=1}^{L},F)LR^2<\infty.
        \end{align*}

        \textbf{Assumption \ref{assump:wgp_sobolev} (iii):}
    Using Theorem \ref{lemma: sob equiv mixture}, Lemma \ref{lemma:reciprocal has derivatives}, the linearity of expectation, and H\"older's inequality, we see that
    \begin{align*}
        \mathbb{E}_{\mu^{y_N,\delta^2,L}_{0}}&\left[C_\low(k_{1}(\cdot,\cdot;\{f^{0,\ell}\}_{\ell=1}^L))^{-1} C_\up(k_{1}(\cdot,\cdot;\{f^{0,\ell}\}_{\ell=1}^L))\right]\nonumber\\
        &=\sqrt{2}C_\low(k_{\M(\nu_{1,j})})^{-1}\E_{\mu^{y_N,\delta^2,L}_{0}}\left[\|1/F(f^{0,j})\|_{H^{\beta}(\O)}\sum_{\ell=1}^L C_\up(k_{\M(\nu_1)}) \|F(f^{0,\ell})\|_{H^{\beta}(\O)}\right]\nonumber\\
        &\leq\sqrt{2}C'_{\ref{lemma:reciprocal has derivatives}}C''_{\ref{lemma:reciprocal has derivatives}}C_\low(k_{\M(\nu_{1,j})})^{-1}\E_{\mu^{y_N,\delta^2,L}_{0}}\left[(1+\|f^{0,j}\|_{H^{\beta}(\O)}^{\beta})\sum_{\ell=1}^L C_\up(k_{\M(\nu_1)}) (1+\|f^{0,\ell}\|_{H^{\beta}(\O)}^{\beta})\right]\nonumber\\
        &\leq\sqrt{2}C'_{\ref{lemma:reciprocal has derivatives}}C''_{\ref{lemma:reciprocal has derivatives}}C_\low(k_{\M(\nu_{1,j})})^{-1}L(1+R^{\beta})^2\sum_{\ell=1}^LC_\up(k_{\M(\nu_1)})\label{eq: mix wgp 1st bound}.
    \end{align*}
    Similarly,
    \begin{equation*}\label{eq: mix wgp 2nd bound}
        \mathbb{E}_{\mu^{y_N,\delta^2,L}_{0}}\left[C_\up(k_{D}(\cdot,\cdot;f^{D-1}))\right]\leq \sqrt{2}C'_{\ref{lemma:reciprocal has derivatives}}L(1+R^{\beta})\sum_{\ell=1}^LC_\up(k_{\M(\nu_1)}).
    \end{equation*}
    The claim then follows, since the bounds derived above has no dependence on \(\varepsilon\), \(N\) or \(U_N\).
\end{proof}

\section{Numerical simulations}\label{sec:numsim}
In this section, we present some illustrative experiments to highlight the convergence theory outlined in section \ref{sec:EAGPR}.
We recover the function \(f(u)=\sin(2u)\), chosen due to the fact that \(f\in H^{\beta}(\O)\) for any \(\beta>0\). 
The goal of these experiments is to verify the rates predicted by our theory, and to investigate the necessity of the assumptions made. For more extensive discussion and a thorough numerical analysis, we refer the reader to \cite{MoriartyOsborne2024}.

The training data is given by noise-free observations $y_N := \{u_n, f(u_n)\}_{n=1}^N$, on design points \(U_N\) that are uniformly spaced with \(N = 2^\ell\) for \(\ell=1,\ldots,10\). With \(N\) training points, the fill distance is \(h_{U_N,\O}={|\O|}/{N}\) (see e.g. \cite{Stuart2016}). 
Throughout, \(\O=(0,5)\), and the $L^2(\O)$-norm is approximated on a mesh of \(N=4096=2^{12}\) points.  
To improve numerical stability, we approximate \(K(U_N,U_N)\) by \(K(U_N,U_N)+ 10^{-15} \mathrm{I}\). Although this technically puts us in the mis-specified setting in Theorem \ref{thm:conv_sob_noisy}, we expect that $\delta^2 = 10^{-15}$ is small enough to not see the effect of this approximation in the error.

 When using a kernel \(k\in\{k_{\warp}^{w,k_s},k_{\mix}^{\{\s_\ell,k_\ell\}^{L}_{\ell=1}}\}\), the converge rates we observe align with those predicted by Corollary \ref{cor:conv_sob}.
However, when \(k=k_{\conv}^{{\lambda_{a}},k_i}\), much faster rates are observed than predicted by Corollary \ref{cor:convergence in the native space}. 
This is due to two factors: the additional tightness in the result provided by using the information that the RKHS equal to a Sobolev space as a vector space, and the loss in bounding the \(L^{2}\) norm by the stronger \(C^0\) norm. 
Additionally, it is noteworthy that all of the figures exhibit a pre-asymptotic phase where \(h_{U_N,\Omega}<h_0\), as expected.  We distinguish between cases where the construction fulfils the conditions of Corollaries \ref{cor:convergence in the native space} or \ref{cor:conv_sob} and where it does not.  Even when our constructions do not entirely fulfil the conditions, we still observe rates consistent with the theoretical predictions.

\subsection*{Examples covered by our theory}
\begin{itemize}
    \item Our first kernel is \(k_{\mix}^{\{\s_\ell,k_\ell\}_{\ell=1}^3}\) with \(k_1=k_{\M(5/2)}\), \(k_2=k_{\M(3/2)}\), and \(k_3=k_{\M(7/2)}\), and \(\s_1(u)=(u/2+1)^2+0.5\) , \(\s_2(u) = (u/2 -0.5)^2+0.5\) and \(\s_3(u)=(u/2+0.5) ^2+0.5\). This satisfies the assumptions of Corollary \ref{cor:conv_sob} since \(\s^{\ell},1/\s_{\ell}\in H^{\beta}(\O)\) for \(\ell=1,2,3\) and so Theorem \ref{lemma: sob equiv mixture} is satisfied with \(\nu_j=3/2\) and \(\beta=2\) . See that in Figure \ref{fig:mixL=3convergence} the observed rate agrees with the predicted rate of \(2\) from Corollary \ref{cor:conv_sob}.
    
    \item The kernel \(k_{\warp}^{w,k_s}\) with \(k_s=k_{\M(5/2)}\) and \(w(u)=(u/5+0.1)^2\)
     satisfies the conditions of Corollary \ref{cor:conv_sob}, since \(\inv(w)\) exists, both \(w'\) and \(\inv(w)'\) are bounded away from zero, and hence Theorem \ref{lemma:sob equiv warping} holds with \(\nu=5/2\) and $\beta=3$. See in Figure \ref{fig:warp_convergence} that the observed rate agrees with the predicted rate of 3.  

    \item The kernel \(k_{\conv}^{{\lambda_{a}},k_i}\)  with \(k_i=k_{\M(1/2)}\) and \(\lambda_{a}(u)=(u-3)^2+\sin(u)+4\) satisfies the assumptions of Corollary \ref{cor:convergence in the native space}, since \(k\in C^{1}(\O\times\O)\). As seen in Figure \ref{fig:convrate1}, the observed rate is significantly faster than the rate $1/2$ predicted by Corollary \ref{cor:convergence in the native space}.
      Note that by establishing an equivalence of the RKHS with a Sobolev space, i.e., \(\H_{k_{\conv}^{{\lambda_{a}},k_i}}(\O) \cong \H^{\beta}(\O)\) with \(\beta\) defined similarly to the warping and mixture kernels, Corollary \ref{cor:conv_sob} would predict a rate of 2, aligning with our observed rate.
\end{itemize}
\subsection*{Examples not covered by our theory}
\begin{itemize}   
    \item This example uses the kernel \(k_{\mix}^{\{\s_\ell,k_\ell\}^{3}_{\ell=1}}\) with \(k_1=k_{\M(3)}\),  \(k_2=k_{\M(5/2)}\) and \(k_3=k_{\M(7/2)}\), along with \(\s_1(u)=\mathbbm{1}_{[0,2]}/2\), \(\s_2 = \mathbbm{1}_{(1,4)}/2\) and \(\s_3=\mathbbm{1}_{[3,5]}/2\). This example does not entirely satisfy the assumptions of Corollary \ref{cor:conv_sob}, since \(\s_\ell\notin H^{\beta}(\O)\) for \(\ell=1,2,3\). However with \(\nu=5/2\) in Theorem \ref{lemma: sob equiv mixture} we see that Corollary \ref{cor:conv_sob} would predict a convergence rate of \(3\), which is close to our observed rate, as seen in Figure \ref{fig:mixconvergence1}.

    \item In Figure \ref{fig:warp_convergence2}, we use \(k=k_{\warp}^{w, k_s}\) with \(k_s=k_{\M(3/2)}\) and \(w(u)=(u-3\pi/4)^2\). Although this kernel does not entirely satisfy the assumptions of Corollary \ref{cor:conv_sob} due to \(\inv(w)\) not existing on \(\O\), the observed results agree with the predicted convergence rate of \(2\) if we take \(\nu=3/2\) in Theorem \ref{lemma:sob equiv warping}.

    \item  In Figure \ref{fig:convratewarp2}, the results using \(k=k_{\warp}^{w, k_s}\) with \(k_s=k_{\M(3/2)}\) and \(w(u)=(u/5+0.1)^2\) for \(u<2.5\) and  \(w(u)=(u/3+0.1)^2\) for \(u\geq 2.5\) are presented. This kernel does not entirely satisfy the assumptions of Corollary \ref{cor:conv_sob} due to the discontinuity of \(w\) at \(u=2.5\). Nevertheless, since the discontinuity occurs on a set of measure zero, \(w\in W^{\beta,\infty}(\O)\) for all \(\beta>0\). As mentioned in Remark \ref{rmk:on ess bnd w}, it is possible to modify our theory to accommodate this case. See that the attained rate agrees with the expected rate of \(2\) with \(\nu=3/2\) in Theorem \ref{lemma:sob equiv warping}.
\end{itemize}

\begin{figure}[htp]%
\centering
\begin{minipage}{0.42\textwidth}
\includegraphics[width=\textwidth]{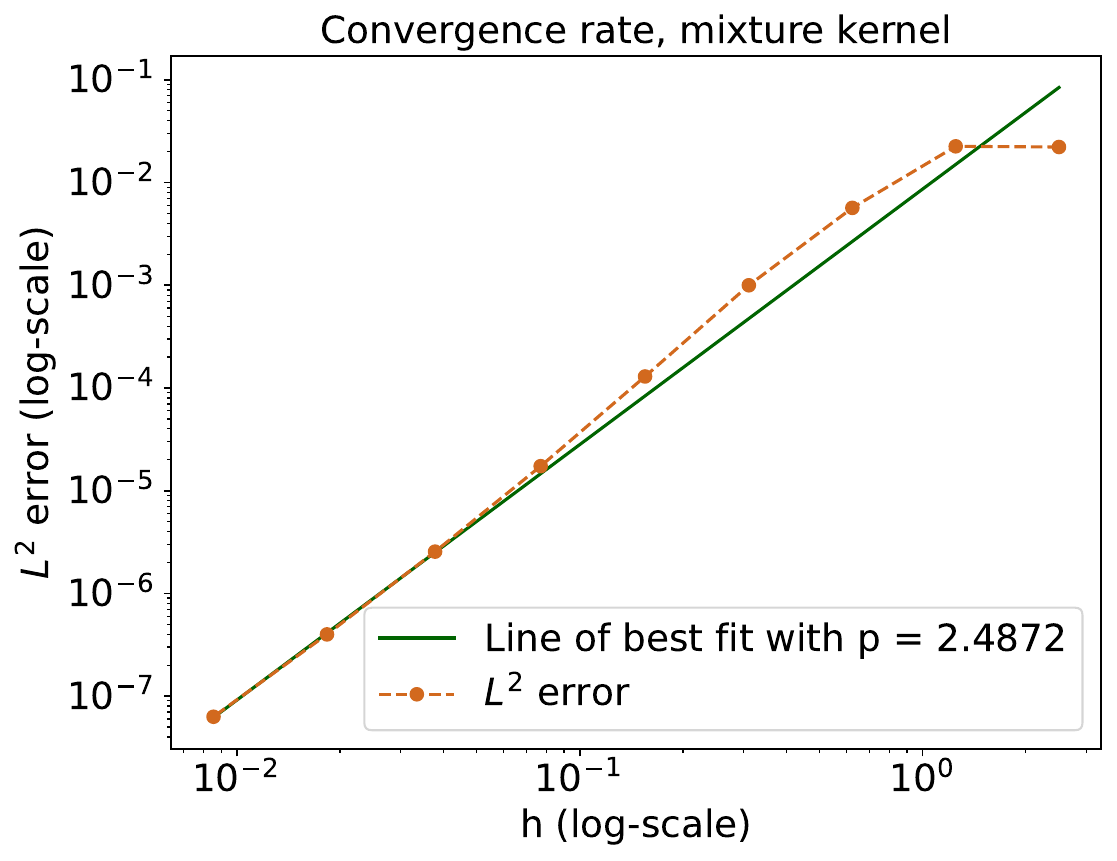}
\caption{\(\GP(0,k_{\mix}^{\{\s_\ell,k_\ell\}^{3}_{\ell=1}})\) with \(k_1=k_{\M(5/2)}\), \(\s_1(u)=(u/2+1)^2+0.5\) \(k_2=k_{\M(3/2)}\), \(\s_2(u) = (u/2 -0.5)^2+0.5\) and \(k_3=k_{\M(7/2)}\), \(\s_3(u)=(u/2+0.5) ^2+0.5\). Expected rate is \(2\).
\vspace{0.0cm}}
\label{fig:mixL=3convergence}
\end{minipage}\hfill
\begin{minipage}{0.42\textwidth}
\vspace{-1.8cm}
\includegraphics[width=\textwidth]{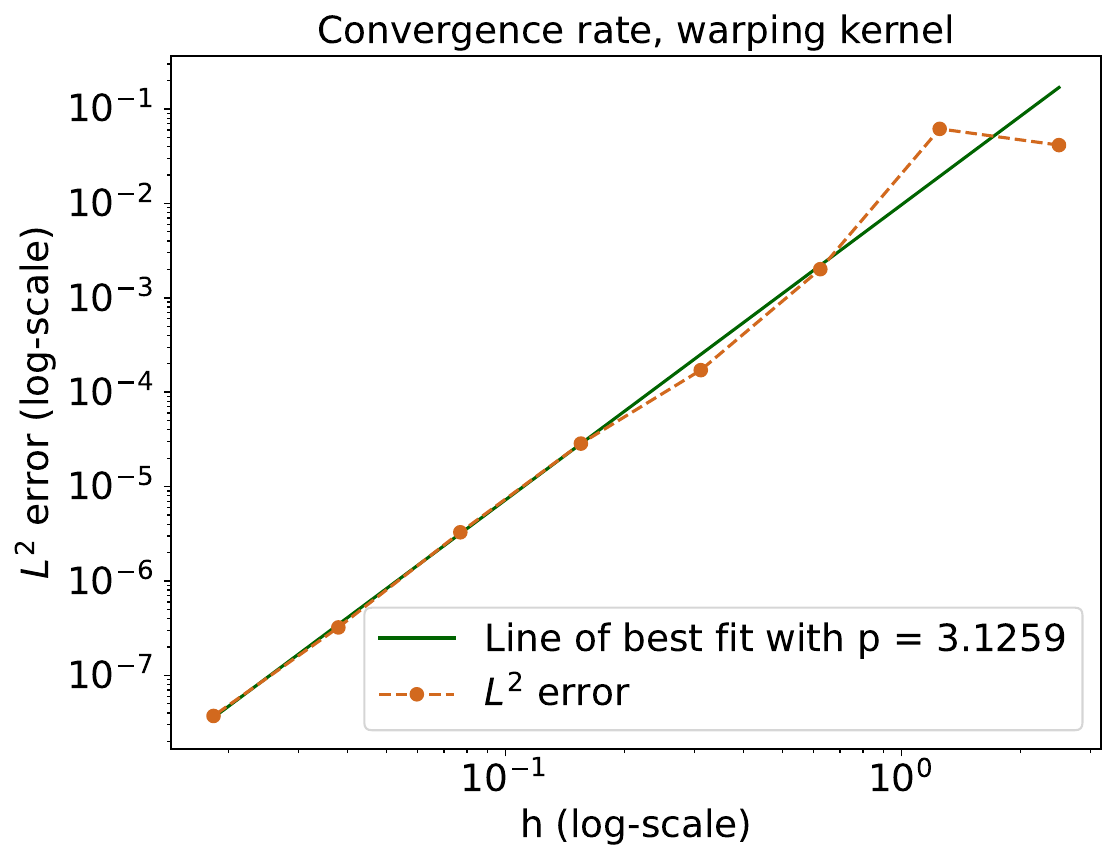}
\caption{\(\GP(0,k_{\warp}^{w, k_s})\) with \(k_s=k_{\M(5/2)}\) and \(w(u)=(u/5+0.1)^2\). Expected rate is \(3\).
\vspace{0.0cm}}
\label{fig:warp_convergence}
\end{minipage}\par
\vskip\floatsep
\begin{minipage}{0.42\textwidth}
\includegraphics[width=\textwidth]{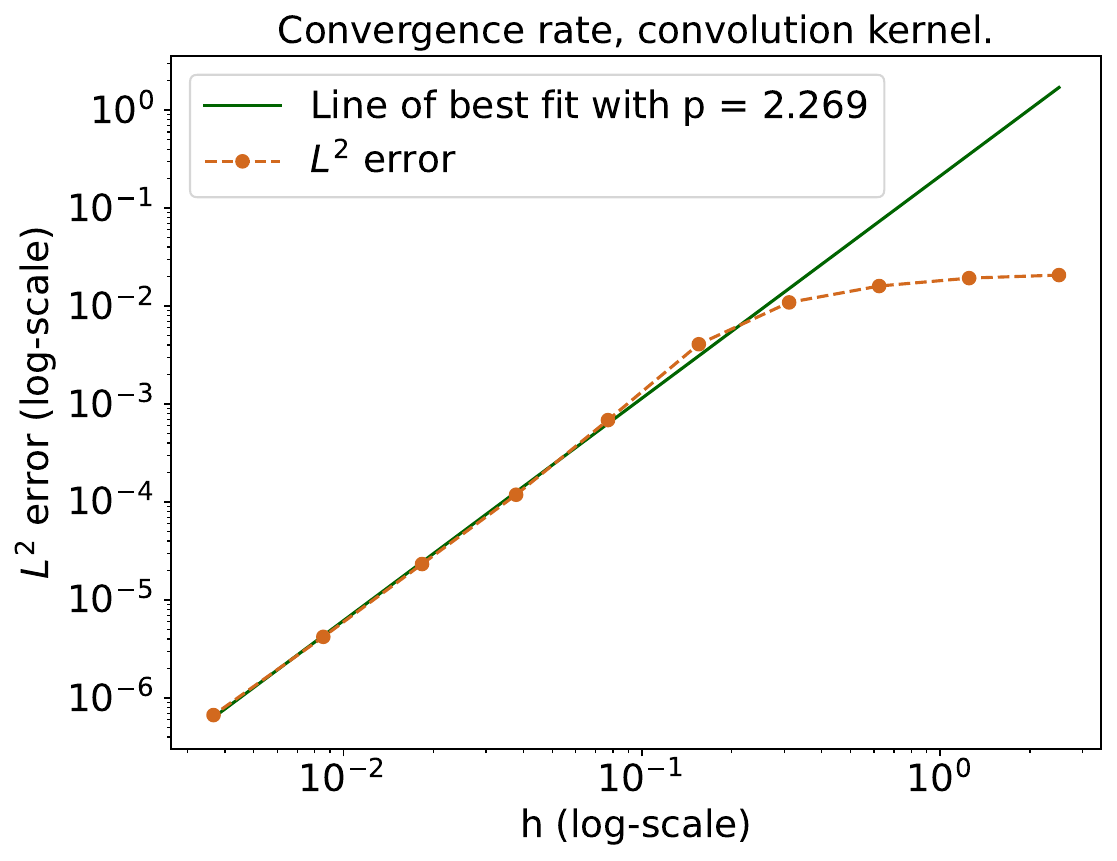}
\caption{\(\GP(0, k_{\conv}^{{\lambda_{a}},k_i})\)  with \(k_i=k_{\M(1/2)}\) and \(\lambda_{a}(u)=(u-3)^2+\sin(u)+4\). Expected rate is \(1/2\).
\vspace{0.3cm}}
\label{fig:convrate1}
\end{minipage}\hfill
\begin{minipage}{0.42\textwidth}
\includegraphics[width=\textwidth]{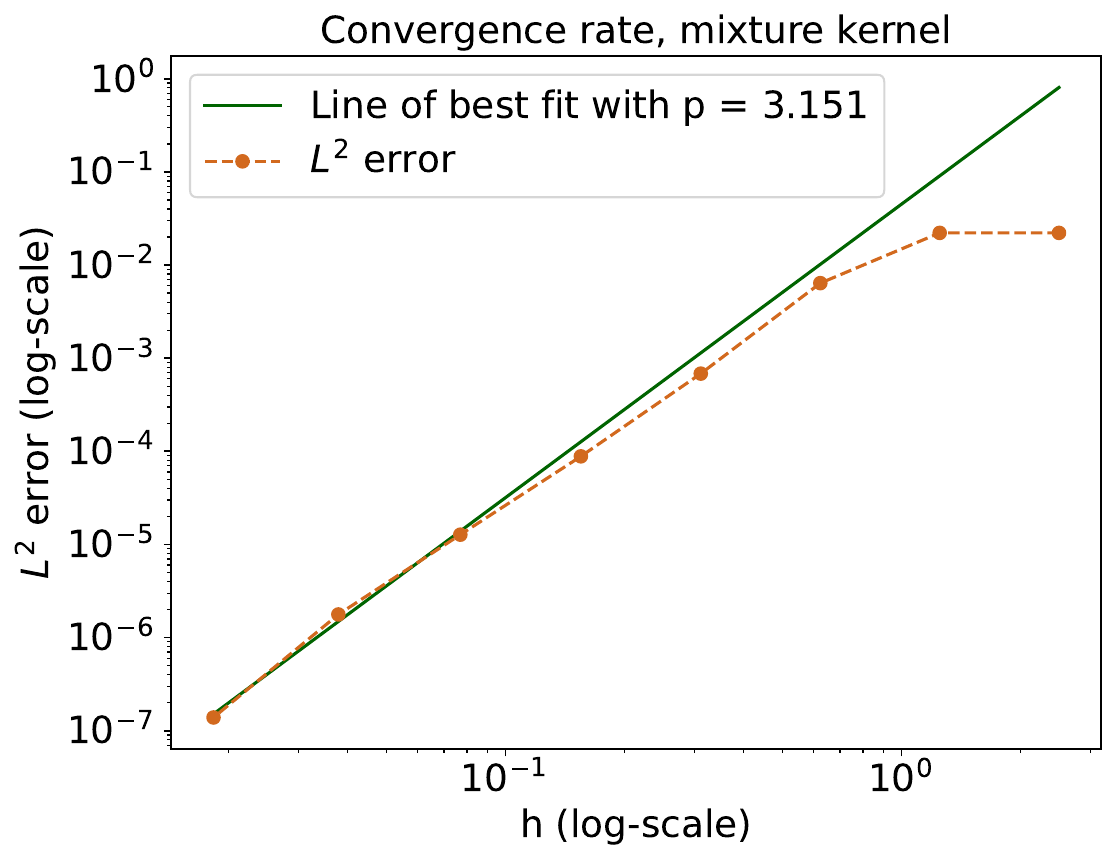}
\caption{\(\GP(0,k_{\mix}^{\{\s_\ell,k_\ell\}^{3}_{\ell=1}})\) with \(k_1=k_{\M(3)}\), \(\s_1(u)=\mathbbm{1}_{[0,2]}/2\) \(k_2=k_{\M(5/2)}\), \(\s_2 = \mathbbm{1}_{(1,4)}/2\) and \(k_3=k_{\M(7/2)}\), \(\s_3=\mathbbm{1}_{[3,5]}/2\). Expected rate is \(3\).
\vspace{0.0cm}}
\label{fig:mixconvergence1}
\end{minipage}
\begin{minipage}{0.42\textwidth}
\includegraphics[width=\textwidth]{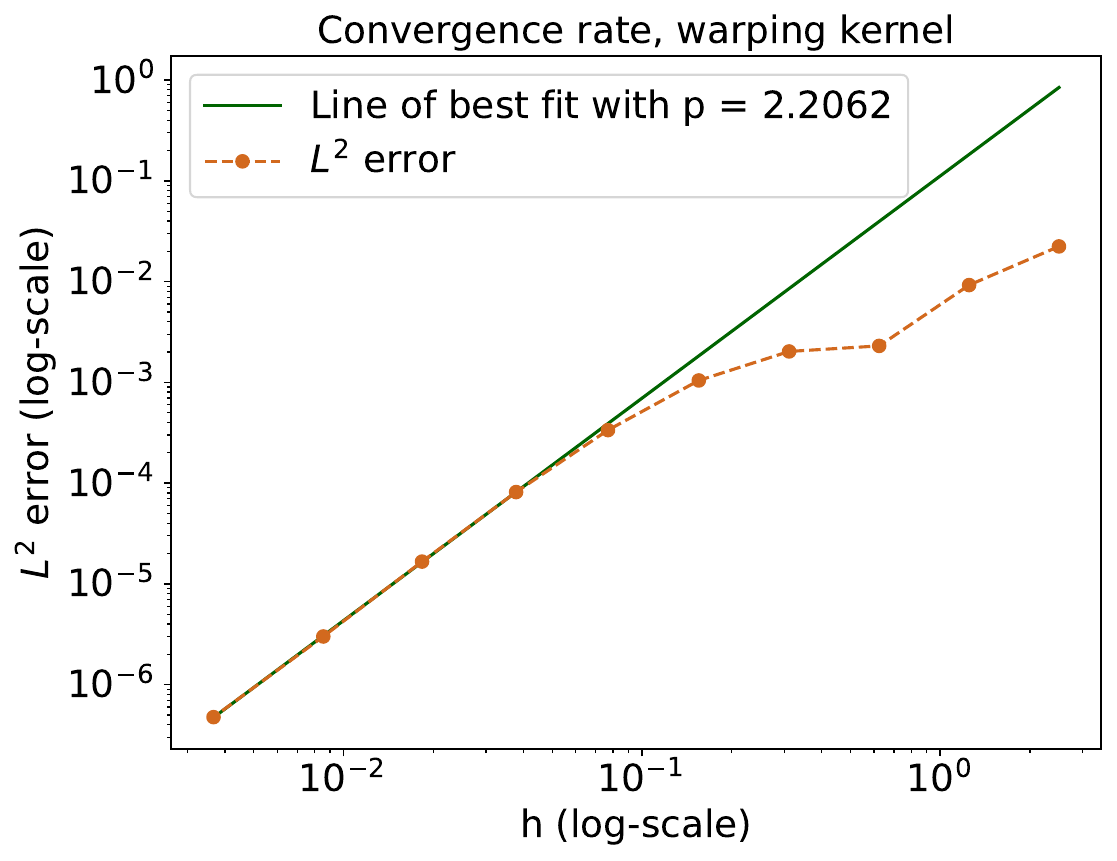}
\caption{\(\GP(0,k_{\warp}^{w, k_s})\) with \(k_s=k_{\M(3/2)}\) and \(w(u)=(u-3\pi/4)^2\). Expected rate is \(2\).
\vspace{0.5cm}}
\label{fig:warp_convergence2}
\end{minipage}\hfill
\begin{minipage}{0.42\textwidth}
\includegraphics[width=\textwidth]{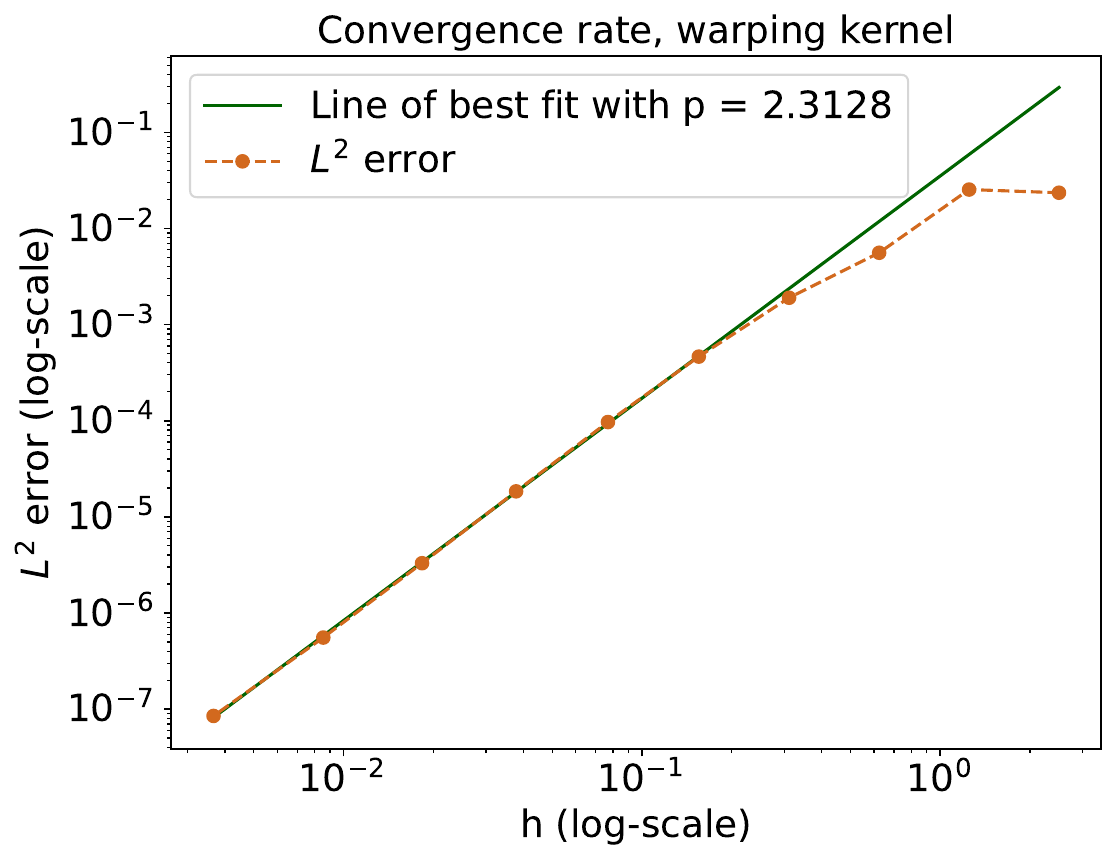}
\caption{\(\GP(0,k_{\warp}^{w, k_s})\) with \(k_s=k_{\M(3/2)}\) and \(w(u)=(u/5+0.1)^2,u<2.5, w(u)=(u/3+0.1)^2,u\geq 2.5\). Expected rate is \(2\).
\vspace{0.0cm}}
\label{fig:convratewarp2}
\end{minipage}
\end{figure}



\section{Conclusion and discussion}\label{sec:conclusion}
{
Gaussian processes are widely employed for approximating complex models. When dealing with non-stationary models, it is crucial that the Gaussian processes used can accurately capture this non-stationarity. This work explores how the validity of this approximation depends on the number of evaluations of the model used in the construction of the emulator, i.e. the number of training points. The error estimates we derive follow the general structure
\[
    \|f-m_{N}\|\leq C N^{-r}\|f\|,
\]
where $f$ is the function of interest and \(m_{N}\) is the predictive mean of the non-stationary or deep GP approximation. The constant \(C\) is dependent on all the hyper-parameters in prior, while the rate \(r\) depends on the regularity of the prior covariance kernel, along with the number of (weak) derivatives  considered in the error norm and the regularity of the function \(f\).

In the following, we summarise our main findings and take-away messages.

\subsection{Brief summary of results}
The novel analytical results in this work can be found in sections \ref{sec:nativespaces} to \ref{sec:EADGPR}, with illustrative numerical examples given in section \ref{sec:numsim}.

In Theorems \ref{lemma:sob equiv warping} and \ref{lemma: sob equiv mixture}, we show that under certain conditions, the reproducing kernel Hilbert space (RKHS) associated to the mixture and warping kernels introduced in section \ref{subsubsec:nonstationary} can be characterised by its equivalence with a Sobolev space. For this to hold, we require the regularity of the stationary kernel \(k_s\) and the non-stationary hyper-parameters $w$ or $\{\sigma_\ell\}_{\ell=1}^L$ to match.

 In section \ref{sec:sample path regularity}, we explore sample path properties of non-stationary Gaussian processes, focusing again on the warping and mixture kernels, see in particular Corollaries \ref{cor:sample path old school} and \ref{cor: sample path sullivan}.  We show that for covariance kernel \(k\in C^{p,p}(\O\times\O)\), we have that \(g\sim \GP(0,k)\) satisfies \(g\in C^{\tilde p}(\O)\) almost surely, for some \(\tilde p\leq p\) depending on the dimension $d$ of $\O$. Note that this may not necessarily be optimal, as it is known that \(g\in H^{p}(\O)\) (see Proposition \ref{prop:regularity of GP samples}), and that under certain conditions the Sobolev and H\"older sample path regularities are equal (see Remark \ref{rem:sob_vs_class}). 

 In sections \ref{sec:EAGPR} and \ref{sec:EADGPR}, we then present a convergence analysis for non-stationary and deep Gaussian process regression. Our analysis is divided into two cases: one where the characterisation of the RKHS as a Sobolev space is possible, and the other when it is not. The majority of this work, including all analysis on DGPs and WGPs, falls into the first case, since this allows for error bounds more explicit in the covariance kernel $k$.

In the case where it is not known that the RKHS is a Sobolev space, convergence rates for non-stationary GP regression with noise-free data are obtained for the warping, mixture and convolution kernels in Corollaries \ref{cor:convergence in the native space} and \ref{cor:conv_rkhs_est}. The latter considers hyper-parameter estimation by an empirical Bayesian approach. These results require that the function $f$ of interest is in the RKHS, which can be difficult to verify.

In the case where the RKHS is norm-equivalent to a Sobolev space, convergence rates for non-stationary GP regression with noise-free data are obtained for the warping and mixture kernels in Corollaries \ref{cor:conv_sob} and \ref{cor:conv_sob_est}, and with noisy data in \ref{cor:conv_sob_est_noisy} and \ref{cor:conv_sob_est_noisy}. Corollaries \ref{cor:conv_sob_est} and \ref{cor:conv_sob_est_noisy} consider hyper-parameter estimation by an empirical Bayesian approach. These results require that the function $f$ of interest is in a Sobolev space, which is typically easier to verify. Various extensions to the error analysis are given in section \ref{subsec:eagr_ext}.

In the case where the RKHS is norm-equivalent to a Sobolev space, convergence rates for deep and wide GP regression are obtained for constructions based on the warping and mixture kernels in Corollaries \ref{cor:DGP_conv_sobolev} and \ref{cor:wideGP}. These results apply to noise-free and noisy data.

\subsection{Discussion and main insights}
\begin{itemize}
    \item  The core idea of this work is to present constructions of non-stationary kernels that achieve the same convergence rates as the stationary kernels used to construct them. A future direction might be to construct non-stationary kernels which overcome any shortcomings of their stationary counterparts. 
    \item 
For most of our results, we only assume that \(f\) belongs to the Sobolev space \(H^\beta(\Omega)\) (or a corresponding tensor product version), for a bounded domain \(\Omega \subseteq \mathbb{R}^d\) and some \(\beta \geq d/2\). This approach offers two main advantages. First, in practice, GP and DGP regression are often used in a black-box setting, where little is known about the function \(f\). Thus, assuming \(f\) is continuous and has a certain number of square-integrable weak derivatives is often reasonable. Second, our results show that the methodology remains robust even if the non-stationarity in the prior is poorly chosen, ensuring that convergence results are still achieved despite potential mismatches in the assumed non-stationary structure compared to the true function \(f\).

This complements recent studies \cite{finocchio2023posterior, Bachoc2024, castillo2024deep, abraham2023deep} which demonstrate that for certain compositional functions \(f\), it is possible to construct a DGP prior distribution that exploits this structure, facilitating faster convergence rates than those found in our study. We believe that our approach can be extended to incorporate more structural assumptions on \(f\) and achieve similar, better convergence results. However, due to the  technicality required to achieve this, it has been reserved for future works.

    \item This work has focused mostly on modelling non-stationary length scales. An interesting future direction would be to model non-stationary regularity, in which case standard Sobolev spaces may no longer be appropriate for the analysis.  
    \item In the case of noise-free data, our analysis recovers the optimal convergence rates $N^{-(\beta-\alpha)/d}$ for approximating $f \in H^\beta(\O)$ with $N$ function values and measuring the error in $H^\alpha(\O)$. This holds for the mixture and warping kernels, and requires that the fill distance decays optimally as \(h_{U_N,\O}\leq CN^{-1/d}\). For non-stationary GP regression, this result can be seen in Corollaries \ref{cor:conv_sob} for fixed hyper-parameters and Corollary \ref{cor:conv_sob_est} for estimated hyper-parameters.  For deep and wide GP regression, this result can be seen in Corollaries \ref{cor:DGP_conv_sobolev} and \ref{cor:wideGP}. For noise-free data on randomly sampled locations $U_N$, for which we have \(h_{U_N,\O}\leq CN^{-1/d+\epsilon}\) for any $\epsilon>0$, we obtain almost optimal rates $N^{-(\beta-\alpha)/d+\epsilon} $ in Corollaries \ref{cor:conv_sob_noisy} and \ref{cor:conv_sob_est_noisy}.
     \item In the case of noisy data, our analysis is more limited and does not recover optimal rates. As can be seen in Theorems \ref{thm:exp convergence rate} and \ref{thm:exp convergence rate_wgp}, convergence of DGP and WGP regression in the noisy case is only guaranteed when the level of noise goes to \(0\) sufficiently fast as \(N\to\infty\). 
     \item Due to the technicality of conditioning deep and wide GPs on noise-free data, we consider approximations of the conditional distribution derived with assumed noisy data. We are then free to choose the size of the noise used in this approximation, and in particular, recovering optimal convergence rates for deep and wide GP regression requires letting this noise go to 0 as $N\rightarrow \infty$ (see Theorems \ref{thm:exp convergence rate} and \ref{thm:exp convergence rate_wgp} for details).
     \item Our analysis for the convolution kernel is quite limited. While we have successfully leveraged the theory on kernel operations to establish that the RKHS for the mixture and warping kernels can be constructed to be norm-equivalent to a Sobolev space, extending this result to the convolution kernel was beyond our reach. However, numerical tests suggest that a similar result should hold true. We do obtain convergence results for the convolution kernel in Corollaries \ref{cor:convergence in the native space} and \ref{cor:conv_rkhs_est}, but they are likely to not be optimal since this is the case for the warping and mixture kernels.
    \item To reduce the number of parameters that need to be tuned, it is common to use the same covariance kernel $k_n \equiv k$ in all layers of the deep (and wide) GP constructions (see e.g. \cite{damianou2013deep,dunlop2018deep}). As discussed below Corollary \ref{cor:DGP_conv_sobolev}, our theory unfortunately does not apply to this setting. The issue arises from the need to ensure the same regularity of the stationary kernel \(k_s\) and the non-stationary the hyper-parameters $w$ or $\{\sigma_\ell\}_{\ell=1}^L$ in Theorems \ref{lemma:sob equiv warping} and \ref{lemma: sob equiv mixture}. This requires more regular kernels in the hidden layers \(n=0,\ldots,D-1\), and the more hidden layers there are, the higher regularity we need in the hidden layers (see Lemmas \ref{lem:DGPkernelWarp}, \ref{lem:DGPkernelMix} and \ref{lem:DGPWideGP} for details). In our approach, we have opted to set \(k_n\equiv k\) for \(n=0,\ldots,D-1\) and \(k_D\neq k\). An alternative approach could involve gradually reducing the regularity over each layer, and it remains to see what works better in practice.
    \item Theorems \ref{thm:exp convergence rate} and \ref{thm:exp convergence rate_wgp} provide general assumptions that guarantee the convergence of deep (and wide) GP regression. Unfortunately, we have been unable to prove these assumptions for constructions based on warping or mixture kernels without introducing the additional assumption that suitable H\"older or Sobolev norms of  the \(D-1\)-th layer are bounded by a constant $R<\infty$ almost surely (see Lemmas \ref{lem:DGPkernelWarp}, \ref{lem:DGPkernelMix} and \ref{lem:DGPWideGP} for details). We note that the same truncation assumption was introduced in \cite{finocchio2023posterior}, albeit it there was introduced on every layer rather than only the penultimate one. Although we were unable to show that the required posterior expected values in Assumption \ref{assump:dgp_sobolev}$(ii),(iii)$ and \ref{assump:wgp_sobolev}$(ii),(iii)$ are bounded independently of $N$ without this truncation assumption, one can show that they are finite using nested applications of Fernique's theorem (see \cite{MoriartyOsborne2024} for more details).
\end{itemize}

}

\section*{Acknowledgements} The authors would like to thank Neil Deo for pointing our attention to {\cite[Lemma 29]{Nickl2020}, which was used in the proof of Lemma \ref{lemma:reciprocal has derivatives}}. CMO was supported by the EPSRC Centre for Doctoral Training in Mathematical
Modelling, Analysis and Computation (MAC-MIGS) funded by the UK Engineering and
Physical Sciences Research Council (grant EP/S023291/1), Heriot-Watt University and the
University of Edinburgh. ALT was supported by EPSRC grant no EP/X01259X/1. CMO and ALT would like to thank the Isaac Newton Institute for Mathematical Sciences, 
Cambridge, for support and hospitality during the programme {\em Mathematical and
Statistical Foundation of Future Data-driven Engineering} where work on this paper
was undertaken. This work was supported by EPSRC grant no EP/R014604/1.



\bibliographystyle{siam}
\bibliography{references}

\begin{thebibliography}{10}

\bibitem{abraham2023deep}
{\sc K.~Abraham and N.~Deo}, {\em Deep gaussian process priors for bayesian
  inference in nonlinear inverse problems}, arXiv preprint arXiv:2312.14294,
  (2023).

\bibitem{abramowitz1965handbook}
{\sc M.~Abramowitz and I.~A. Stegun}, {\em Handbook of Mathematical Functions},
  Dover, 1965.

\bibitem{RAAdams_JJFFournier_2003}
{\sc R.~A. Adams and J.~J.~F. Fournier}, {\em Sobolev Spaces}, Elsevier, 2003.

\bibitem{arcangeli2012extension}
{\sc R.~Arcang{\'e}li, M.~C. L{\'o}pez~de Silanes, and J.~J. Torrens}, {\em
  {Extension of sampling inequalities to Sobolev semi-norms of fractional order
  and derivative data}}, Numer. Math., 121 (2012), pp.~587--608.

\bibitem{Aronszajn1950}
{\sc N.~Aronszajn}, {\em {Theory of Reproducing Kernels}}, Transactions of the
  AMS, 68 (1950), pp.~337--404.

\bibitem{Aronszajn1955}
{\sc N.~Aronszajn}, {\em {Boundary values of functions with finite Dirichlet
  integral}}, Techn. Report Univ. of Kansas, 14 (1955), pp.~77--94.

\bibitem{Bachoc2024}
{\sc F.~Bachoc and A.~Lagnoux}, {\em {Posterior contraction rates for
  constrained deep Gaussian processes in density estimation and
  classification}}, Comm. Statist. Theory Methods,  (2024), pp.~1--0.

\bibitem{bergh2012interpolation}
{\sc J.~Bergh and J.~L{\"o}fstr{\"o}m}, {\em {Interpolation Spaces: An
  Introduction}}, Springer, 1976.

\bibitem{bungartz2004sparse}
{\sc H.-J. Bungartz and M.~Griebel}, {\em Sparse grids}, Acta Numer., 13
  (2004), pp.~147--269.

\bibitem{castillo2024deep}
{\sc I.~Castillo and T.~Randrianarisoa}, {\em {Deep Horseshoe Gaussian
  Processes}}, arXiv preprint arXiv:2403.01737,  (2024).

\bibitem{Choi2007OnPC}
{\sc T.~Choi and M.~J. Schervish}, {\em On posterior consistency in
  nonparametric regression problems}, J. Multivariate Anal., 98 (2007),
  pp.~1969--1987.

\bibitem{cosg19}
{\sc J.~Cockayne, C.~J. Oates, T.~J. Sullivan, and M.~Girolami}, {\em {Bayesian
  Probabilistic Numerical Methods}}, SIAM Rev., 61 (2019), pp.~756--789.

\bibitem{da2014stochastic}
{\sc G.~Da~Prato and J.~Zabczyk}, {\em {Stochastic Equations in Infinite
  Dimensions}}, Cambridge University Press, 1992.

\bibitem{damianou2013deep}
{\sc A.~C. Damianou and N.~D. Lawrence}, {\em {Deep Gaussian Processes}}, Proc.
  Mach. Learn. Res. (PMLR), 31 (2013), pp.~207--215.

\bibitem{Desai_2023}
{\sc A.~Desai, E.~Gujarathi, S.~Parikh, S.~Yadav, Z.~Patel, and N.~Batra}, {\em
  Deep gaussian processes for air quality inference}, ACM Trans. Math.,
  (2023), pp.~278--279.

\bibitem{dunlop2018deep}
{\sc M.~M. Dunlop, M.~A. Girolami, A.~M. Stuart, and A.~L. Teckentrup}, {\em
  {How Deep Are Deep Gaussian Processes?}}, J. Mach. Learn. Res., 19 (2018),
  pp.~1--46.

\bibitem{Dunlop2016}
{\sc M.~M. Dunlop, M.~A. Iglesias, and A.~M. Stuart}, {\em {Hierarchical
  Bayesian Level Set Inversion}}, Stat. Comput., 27 (2016), pp.~1573--1375.

\bibitem{finocchio2023posterior}
{\sc G.~Finocchio and J.~Schmidt-Hieber}, {\em {Posterior contraction for deep
  Gaussian process priors}}, J. Mach. Learn. Res., 24 (2023), pp.~1--49.

\bibitem{fisher2020}
{\sc M.~Fisher, C.~J. Oates, C.~E. Powell, and A.~L. Teckentrup}, {\em {A
  Locally Adaptive Bayesian Cubature Method}}, Proc. Mach. Learn. Res. (PMLR),
  108 (2020), pp.~1265--1275.

\bibitem{Fuentes2001}
{\sc M.~Fuentes}, {\em A high frequency kriging approach for non-stationary
  environmental processes}, Environmetrics, 12 (2001), pp.~469--483.

\bibitem{gine_nickl_2015}
{\sc E.~Giné and R.~Nickl}, {\em Mathematical Foundations of
  Infinite-Dimensional Statistical Models}, Cambridge University Press, 2015.

\bibitem{Minh2010}
{\sc M.~Ha~Quang}, {\em {Some Properties of Gaussian Reproducing Kernel Hilbert
  Spaces and Their Implications for Function Approximation and Learning
  Theory}}, Constr. Approx., 32 (2010), pp.~307--338.

\bibitem{helin2023introduction}
{\sc T.~Helin, A.~M. Stuart, A.~L. Teckentrup, and K.~C. Zygalakis}, {\em
  {Introduction To Gaussian Process Regression In Bayesian Inverse Problems,
  With New Results On Experimental Design For Weighted Error Measures}}, arXiv
  preprint arXiv:2302.04518,  (2023).

\bibitem{Herrman2020}
{\sc L.~Herrmann, M.~Keller, and C.~Schwab}, {\em {Quasi-Monte Carlo Bayesian
  estimation under Besov priors in elliptic inverse problems}}, Math. Comp, 90
  (2020), pp.~111--135.

\bibitem{hoffmann2015adaptive}
{\sc M.~Hoffmann, J.~Rousseau, and J.~Schmidt-Hieber}, {\em {On adaptive
  posterior concentration rates}}, Ann. Statist., 43 (2015), pp.~2259--2295.

\bibitem{Ping2020}
{\sc P.~Jiang, R.~Li, H.~Lu, and X.~Zhang}, {\em Modeling of electricity demand
  forecast for power system}, Neural Computing and Applications, 32 (2020),
  pp.~6857--6875.

\bibitem{Warren2002}
{\sc W.~P. Johnson}, {\em {Combinatorics of Higher Derivatives of Inverses}},
  Amer. Math. Monthly, 109 (2002), pp.~273--277.

\bibitem{jones_1951}
{\sc C.~W. Jones}, {\em {Calculus of Finite Differences}}, Chelsea Publishing
  Co., New York, 1951.

\bibitem{karvonen2020maximum}
{\sc T.~Karvonen, G.~Wynne, F.~Tronarp, C.~Oates, and S.~Sarkka}, {\em Maximum
  likelihood estimation and uncertainty quantification for gaussian process
  approximation of deterministic functions}, SIAM/ASA J. Uncertain. Quantif., 8
  (2020), pp.~926--958.

\bibitem{kennedy2001bayesian}
{\sc M.~C. Kennedy and A.~O'Hagan}, {\em Bayesian calibration of computer
  models}, J. R. Stat. Soc. Ser. B. Stat. Methodol., 63 (2001), pp.~425--464.

\bibitem{matern}
{\sc B.~Mat{\'e}rn}, {\em Spatial Variation}, Springer, 2013.

\bibitem{monterrubio2020posterior}
{\sc K.~Monterrubio-G{\'o}mez, L.~Roininen, S.~Wade, T.~Damoulas, and
  M.~Girolami}, {\em Posterior inference for sparse hierarchical non-stationary
  models}, Comput. Statist. Data Anal., 148 (2020), p.~106954.

\bibitem{MoriartyOsborne2024}
{\sc C.~Moriarty-Osborne}, {\em {Non-stationary and deep Gaussian processes}},
  PhD thesis, University of Edinburgh and Heriot-Watt University, 2024 \,
  (expected).

\bibitem{Muir2023}
{\sc J.~B. Muir and Z.~E. Ross}, {\em {A deep Gaussian process model for
  seismicity background rates}}, Geophysical Journal International, 234 (2023),
  pp.~427--438.

\bibitem{Nabney1996}
{\sc I.~T. Nabney, A.~McLachlan, and D.~Lowe}, {\em {Practical methods of
  tracking of nonstationary time series applied to real-world data}},
  Proceedings of the SPIE, 2760 (1996), pp.~152--163.

\bibitem{Narcowich2005}
{\sc F.~Narcowich, J.~Ward, and H.~Wendland}, {\em Sobolev bounds on functions
  with scattered zeros, with application to radial basis function surface
  fitting}, Math. Comput., 74 (2005), pp.~743--763.

\bibitem{Nason2006}
{\sc G.~P. Nason}, {\em {Stationary and non-stationary time series}},
  {Statistics in Volcanology}, 01 (2006), pp.~129--142.

\bibitem{Nickl2020}
{\sc R.~Nickl, S.~van~de Geer, and S.~Wang}, {\em {Convergence rates for
  Penalised Least Squares estimators in PDE-constrained regression problems}},
  SIAM/ASA J. Uncertain. Quantif., 8 (2020), pp.~374--413.

\bibitem{nieman2022contraction}
{\sc D.~Nieman, B.~Szabo, and H.~Van~Zanten}, {\em {Contraction rates for
  sparse variational approximations in Gaussian process regression}}, Proc.
  Mach. Learn. Res. (PMLR), 23 (2022), pp.~9289--9314.

\bibitem{Sethian2021}
{\sc M.~M. Noack and J.~A. Sethian}, {\em {Advanced Stationary and
  Non-Stationary Kernel Designs for Domain-Aware Gaussian Processes}}, Comm. in
  Appl. Math. and Comp. Sci., 17 (2022), p.~131–156.

\bibitem{Nobile2015}
{\sc F.~Nobile and F.~Tesei}, {\em {A Multi Level Monte Carlo Method with
  Control Variate for elliptic PDEs with log-normal coefficients}}, Stoch.
  Partial Differ. Equ. Anal. Comput., 3 (2015), p.~398–444.

\bibitem{oates2019convergence}
{\sc C.~Oates, J.~Cockayne, F.~Briol, and M.~Girolami}, {\em {Convergence Rates
  for a Class of Estimators Based on Stein’s Method}}, Bernoulli, 25 (2019),
  pp.~1141--1159.

\bibitem{Paciorek2003phd}
{\sc C.~J. Paciorek}, {\em Nonstationary Gaussian processes for regression and
  spatial modelling}, PhD thesis, Carnegie Mellon University, 2003.

\bibitem{Paciorek2003}
{\sc C.~J. Paciorek and M.~J. Schervish}, {\em {Nonstationary covariance
  functions for Gaussian process regression}}, NeurIPS,  (2003), p.~273–280.

\bibitem{paulsen_raghupathi_2016}
{\sc V.~I. Paulsen and M.~Raghupathi}, {\em {An Introduction to the Theory of
  Reproducing Kernel Hilbert Spaces}}, Cambridge University Press, 2016.

\bibitem{Rasmussen06gaussianprocesses}
{\sc C.~E. Rasmussen and C.~K. Williams}, {\em {Gaussian Processes for Machine
  Learning}}, MIT Press, 2006.

\bibitem{Remes2017}
{\sc S.~Remes, M.~Heinonen, and S.~Kaski}, {\em {Non-Stationary Spectral
  Kernels}}, NeurIPS,  (2017), p.~4645–4654.

\bibitem{Riikimake2010}
{\sc J.~Riihimäki and A.~Vehtari}, {\em Gaussian processes with monotonicity
  information}, Proc. Mach. Learn. Res. (PMLR), 9 (2010), pp.~645--652.

\bibitem{sacks1989design}
{\sc J.~Sacks, W.~J. Welch, T.~J. Mitchell, and H.~P. Wynn}, {\em {Design and
  Analysis of Computer Experiments}}, Statist. Sci., 4 (1989), pp.~409--423.

\bibitem{Sampson1992}
{\sc P.~D. Sampson and P.~Guttorp}, {\em {Nonparametric Estimation of
  Nonstationary Spatial Covariance Structure}}, J. Amer. Statist. Assoc. J, 87
  (1992), pp.~108--119.

\bibitem{sauer2023nonstationary}
{\sc A.~Sauer, A.~Cooper, and R.~B. Gramacy}, {\em {Non-stationary Gaussian
  Process Surrogates}}, arXiv preprint arXiv:2305.19242,  (2023).

\bibitem{Sauer2022}
{\sc A.~Sauer, R.~B. Gramacy, and D.~Higdon}, {\em {Active Learning for Deep
  Gaussian Process Surrogates}}, Technometrics, 65 (2023), pp.~4--18.

\bibitem{SCHEUERER20101879}
{\sc M.~Scheuerer}, {\em Regularity of the sample paths of a general second
  order random field}, Stochastic Process. Appl., 120 (2010), pp.~1879--1897.

\bibitem{sss13}
{\sc M.~Scheuerer, R.~Schaback, and M.~Schlather}, {\em {Interpolation of
  spatial data--A stochastic or a deterministic problem?}}, European J. Appl.
  Math., 24 (2013), pp.~601--629.

\bibitem{Soltanpour2023}
{\sc M.~Soltanpour, M.~Yousefnezhad, R.~Greiner, P.~Boulanger, and B.~Buck},
  {\em {Ischemic Stroke Lesion Prediction using imbalanced Temporal Deep
  Gaussian Process (iTDGP)}}, arXiv preprint arXiv:2211.09068,  (2022).

\bibitem{Stein1999}
{\sc M.~L. Stein}, {\em {Interpolation of Spatial Data: Some Theory for
  Kriging}}, Springer, 1999.

\bibitem{stuart2010inverse}
{\sc A.~M. Stuart}, {\em {Inverse problems: A Bayesian perspective}}, Acta
  numerica, 19 (2010), pp.~451--559.

\bibitem{Stuart2016}
{\sc A.~M. Stuart and A.~L. Teckentrup}, {\em {Posterior Consistency for
  Gaussian Process Approximations of Bayesian Posterior Distributions}}, Math.
  Comp., 87 (2016), pp.~721--753.

\bibitem{Swiler_2020}
{\sc L.~P. Swiler, M.~Gulian, A.~L. Frankel, C.~Safta, and J.~D. Jakeman}, {\em
  {A Survey of Constrained Gaussian Process Regression: Approaches and
  Implementation Challenges}}, Journal of Machine Learning for Modeling and
  Computing, 1 (2020), pp.~119--156.

\bibitem{Teckentrup2019}
{\sc A.~L. Teckentrup}, {\em {Convergence of Gaussian Process Regression with
  Estimated Hyper-Parameters and Applications in Bayesian Inverse Problems}},
  SIAM/ASA J. Uncertain. Quantif., 8 (2020), pp.~1310--1337.

\bibitem{Vaart2011}
{\sc A.~W. van~der Vaart and H.~van Zanten}, {\em {Information Rates of
  Nonparametric Gaussian Process Methods}}, J. Mach. Learn. Res., 12 (2011),
  pp.~2095--2119.

\bibitem{van_der_Vaart_2008}
{\sc A.~W. van~der Vaart and J.~H. van Zanten}, {\em Rates of contraction of
  posterior distributions based on {G}aussian process priors}, Ann. Statist.,
  36 (2008), pp.~1435--1463.

\bibitem{Volodina2020}
{\sc V.~Volodina and D.~Williamson}, {\em {Diagnostics-Driven Nonstationary
  Emulators Using Kernel Mixtures}}, SIAM/ASA J. Uncertain. Quantif., 8 (2020),
  pp.~1--26.

\bibitem{Wang2021BayesianNM}
{\sc J.~Wang, J.~Cockayne, O.~A. Chkrebtii, T.~J. Sullivan, and C.~J. Oates},
  {\em Bayesian numerical methods for nonlinear partial differential
  equations}, Stat. Comp., 31 (2021), pp.~1--20.

\bibitem{wendland_2004}
{\sc H.~Wendland}, {\em Scattered Data Approximation}, Cambridge University
  Press, 2004.

\bibitem{Wilson2016}
{\sc A.~G. Wilson, Z.~Hu, R.~Salakhutdinov, and E.~P. Xing}, {\em {Deep Kernel
  Learning}}, Proc. Mach. Learn. Res. (PMLR), 51 (2016), pp.~370--378.

\bibitem{wynne2021convergence}
{\sc G.~Wynne, F.-X. Briol, and M.~Girolami}, {\em {Convergence Guarantees for
  Gaussian Process Means With Misspecified Likelihoods and Smoothness}}, J.
  Mach. Learn. Res., 22 (2021), pp.~5468--5507.

\end{thebibliography}

\appendix
\begin{appendices}
\section{Additional proofs}\label{appendixA}


\begin{proof}[of Lemma \ref{lemma:exp_inputwarping constant}] 
Let \(g:\R\to\R\) and \(h:\R^2\to\R\) be defined as \(g(\zeta)=\s^2\exp(-\frac{\zeta^2}{2\lambda^2})\) and \(h(u_1, u_2)=w(u_1)-w(u_2)\). Then $k_{\warp}^{w,k_s}(u_1,u_2) = g(h(u_1,u_2))$.
We introduce the variable \(v=(v_1,\ldots,v_{|\alpha|})\), for $\alpha \in \N_0^2$ a multi-index as in $C_{\ref{thm:Wend11.13}}$ and $|\alpha|=\|\alpha\|_1$, such that
\begin{equation}\label{eq:app_v}
    v_i = \begin{cases}
            u_1 \quad&\text{ if }i\leq\a_1,\\
             u_2 &\text{ if }i>\a_1.
          \end{cases}
\end{equation}
This notation allows us to write Fa\`{a} di Bruno's formula as \cite{jones_1951}
\begin{align*}
    D^{\a}g(h(u_1, u_2))=\frac{\partial^{|\a|}g(h(u_1, u_2))}{\partial v_1\ldots\partial v_{|\a|}}=\sum_{\pi\in\Pi}d^{|\pi|}_{\zeta}g(\zeta)\bigg|_{\zeta=h(u_1,u_2)}
    \prod_{B\in\pi}\frac{\partial^{|B|}h(u_1,u_2)}{\prod_{i\in B}\partial v_i},
\end{align*}
where \(\Pi\) is the set of all possible partitions of the set \(\{1,\ldots,|\alpha|\}\), \(B\in\pi\) means that \(B\) runs over all the blocks of the partition \(\pi\), and $|\pi|$ is the number of blocks in $\pi$.
It can be shown (see, e.g. \cite[Proposition 7.1]{Herrman2020}) that
\begin{equation}\label{eq:Herrmaneq}
    \left|d_\zeta^ng(\zeta)\right|=\left|d_\zeta^{n}\s^2\exp\left(-\frac{\zeta^2}{2\lambda^2}\right)\right|\leq C_0\s^2\exp\left(-\frac{\zeta^2}{4\lambda^2}\right)\sqrt{n!}\leq C_0\s^2\sqrt{n!},\quad\forall  n\in\N,\;\forall\zeta\in\R,
\end{equation}
where \(C_0\leq1.0866\).
Notice also that for any partition \(\pi\) and any block \(B\in\pi\) we have that
\begin{equation}\label{eq:DLh}
    \frac{\partial^{|B|}h(u_1,u_2)}{\prod_{i\in B}\partial v_i}=
                        \begin{cases}
                        0 \quad&\text{ if }\exists i,j\in B \text{ such that }i\leq\a_1\text{ and } j>\a_1,\\
                         d_{u_1}^{|B|}w(u_1) &\text{ if }\forall i\in B, i\leq\a_1,\\
                         -d_{u_2}^{|B|}w(u_2) &\text{ if }\forall i\in B,\a_1<i.
                      \end{cases}
\end{equation}
Using \eqref{eq:Herrmaneq}, \eqref{eq:DLh} and $|\pi|\leq |\alpha|$, we see that
\begin{align*}
    \left|D^{\a}g(h(u_1, u_2))\right|
    &\leq 
    \sum_{\pi\in\Pi}C_0\s^2\sqrt{|\pi|!}
    \prod_{B\in\pi}\left|
    d_u^{|\beta|}w(u)\right|
    \\
    &\leq 
     {C_0\s^2\sqrt{|\alpha|!}}\|w\|_{C^{|\alpha|}(\Omega)}\sum_{\pi\in\Pi}1.
\end{align*}
With \(B_{|\alpha|}\) denoting the number of ways to partition the set \(\{1,\ldots,|\alpha|\}\), this finishes the proof.
\end{proof}

\begin{proof}[of Lemma \ref{lemma:exp_kernel mixture}]
Let \(g:\R\to\R\) and \(h:\R^2\to\R\) be defined as \(g(\zeta)=\exp(-\frac{\zeta^2}{2\lambda^2})\) and \(h(u, u')=u-u'\). Then we can write the kernel \(k_{\mix}^{\{\sigma_\ell, k_\ell\}_{\ell=1}^L}\) as a sum of terms 
 of the form of $k_\ell(u_1,u_2) = \sigma_\ell(u_1) \sigma_\ell(u_2) g(h(u_1,u_2))$.
Let $\alpha$ and \(v=(v_1,\ldots,v_{|\alpha|})\) as in \eqref{eq:app_v}. 
The Leibniz rule then gives
\begin{equation}\label{eq:prodrulesig}
    D^{\a}k_\ell(u,u')=\sum_{S}\frac{\partial^{|S|}\sigma_\ell(u')}{\Pi_{i\in S}\partial v_i}
    \left(\sum_{T_S}\frac{\partial^{|T_S|}\sigma_\ell(u)}{\Pi_{j\in T_S}\partial v_j}\frac{\partial^{|\a|-|S|-|T_S|}g(h(u,u'))}{\Pi_{m \notin S\cup T_S}\partial v_m}\right),
\end{equation}
where $S$ runs over the set $\mathcal{P}(\{1,\ldots, |\alpha|\})$, and $T_S$ runs over the set $\mathcal{P}(\{1,\ldots, |\alpha|\}\setminus S)$.
Using Fa\`{a} di Bruno's formula \cite{jones_1951}, we see that for any \(\widehat\a=(\widehat\a_1,\widehat\a_2)\in\N^2_0\) with \(\widehat \a_1\leq \a_1\) and \(\widehat \a_2\leq \a_2\)
\begin{align}\label{eq:dalf}
    D^{\widehat\a}g(h(u, u'))=\sum_{\pi\in\Pi}d^{|\pi|}_{\zeta}g(\zeta)\bigg|_{\zeta=h(u,u')}
    \prod_{B\in\pi}\frac{\partial^{|B|}h(u,u')}{\prod_{i\in B}\partial v_i}.
\end{align}
where \(\Pi\) is the set of all possible partitions of the set \(\{1,\ldots,|\widehat\alpha|\}\). 
For any \(\pi\) and \(B\in\pi\), we have
\begin{equation}\label{eq:SSh}
    \frac{\partial^{|B|}h(u,u')}{\prod_{i\in B}\partial v_i}=
    \frac{\partial^{|B|}(u-u')}{\prod_{i\in B}\partial v_i}=
                        \begin{cases}
                        \:\:1 \quad&\text{ if } |B|=1\text{ and }v_1=u,\\
                        -1 \quad&\text{ if } |B|=1\text{ and }v_1=u',\\
                         \:\:0 &\text{ otherwise}.
                      \end{cases}
\end{equation}
This means that the only non-zero term in \eqref{eq:dalf} corresponds to the partition that contains only singletons. Let \(\tilde\Pi\) denote the set containing only this partition. Then \eqref{eq:Herrmaneq}, \eqref{eq:SSh}, $|\pi|\leq |\widehat\alpha|$ and $|\tilde\Pi|=|\widehat\a|$ give
\begin{equation}
    \left| D^{\widehat\a}g(h(u, u'))\right|
    \leq \sum_{\pi\in\tilde\Pi}C_0\sqrt{|\pi|!}
    \prod_{B\in\pi}1 \\ 
    \leq {C_0|\widehat\a|\sqrt{|\widehat\alpha|!}}\label{eq:unforseen}.
\end{equation}
For any subset \(S\) (or \(T_S\)) in \eqref{eq:prodrulesig}, 
we further have
\begin{equation}\label{eq:ssbound}
    \left|\frac{\partial^{|S|}\sigma_\ell(u')}{\Pi_{i\in S}\partial v_i}\right|\leq\|\sigma_\ell\|_{C^{|\a|}(\Omega)}.
\end{equation}
Using \eqref{eq:ssbound} and \eqref{eq:unforseen}, we then have
\begin{align*}
    \left|D^{\a}k(u,u')\right|
    &\leq \|\s_\ell\|^2_{C^{|\a|}(\Omega)}\sum_{S}
    \sum_{T_S}\left|\frac{\partial^{|\a|-|S|-|T_S|}g(h(u,u'))}{\Pi_{\ell\notin S\cup T_S}\partial v_\ell}\right|\\
    &\leq C_0|\a|\sqrt{|\a|!}\|\s_\ell\|_{C^{|\a|}(\Omega)}^2\sum_{S}
    \sum_{T_S}1.
\end{align*}

By taking the maximum over each \(\ell\in \{1,\ldots,L\}\) we achieve the desired result since the power set of \(\{1,\ldots,|\a|\}\) has \(2^{|\a|}\) elements.
\end{proof}



\begin{proof}[of Lemma \ref{lemma:exp_convolutionker constant}]
 Let $\alpha$ and \(v=(v_1,\ldots,v_{|\alpha|})\) be as in \eqref{eq:app_v}.  
With $\lambda=1$, the Leibniz rule gives
\begin{align*}
D^{\alpha}k_{\conv}^{{\lambda_{a}},k_i}(u,u')&=2^{1/2}\s^2\sum_S\frac{\partial^{|S|}{\lambda_{a}}(u)^{1/4}}{\Pi_{i\in S}\partial v_i}\Bigg(\sum_{T_S}\frac{\partial^{|T_S|}{\lambda_{a}}(u')^{1/4}}{\Pi_{j\in T_S}\partial v_j}\\
&\Bigg(\sum_{V_{S,T_S}}\frac{\partial^{|V_{S,T_S}|}({\lambda_{a}}(u)+{\lambda_{a}}(u'))^{-1/2}}{\Pi_{k\in V_{S, T_S}}\partial v_k}\frac{\partial^{|\alpha|-|S|-|T_S|-|V_{S,T_S}|}\exp\left(-\frac{|u-u'|^2}{(({\lambda_{a}}(u)+{\lambda_{a}}(u'))/2)}\right)}{\Pi_{m \not\in S\cup T_S\cup V_{S,T_S}}\partial{v_m}}\Bigg)\Bigg)\nonumber
\end{align*}
where $S$ runs over the set $\mathcal{P}(\{1,\ldots, |\alpha|\})$, $T_S$ runs over the set $\mathcal{P}(\{1,\ldots, |\alpha|\}\setminus S)$, and
    $V_{S,T_S}$ runs over the set $\mathcal{P}(\{1,\ldots, |\alpha|\}\setminus S\cup T_S)$.
We now look at each term separately.

\textbf{First (and second) term.}
Let \(h_1(\zeta)=\zeta^{1/4}\) and \(g_1(u)={\lambda_{a}}(u)\).
For any \(n\in\N\), we have 
\begin{align*}
\left|d^{n}_\zeta h_1(\zeta)\big|_{\zeta = g_1(u)}\right|=\left|d^{n}_\zeta \zeta^{1/4}\big|_{\zeta = g_1(u)}\right|=\left|\prod_{i=0}^{n-1}|1/4 - i|({\lambda_{a}}(u))^{1/4-n}\right| \leq|5/4 - n|^{n}\max\{\|{\lambda_{a}}\|_{C^0(\Omega)}, 1\}^{1/4 - n},
\end{align*}
and for any \(B\), a block of any partition of any \(S\), we have
\[
    \frac{\partial^{|B|}g_1(u)}{\prod_{i\in B}\partial v_i}=                        \begin{cases}
                        0 \quad&\text{ if }\exists i\in B \text{ such that }i>\a_1,\\
                         d_{u}^{|B|}{\lambda_{a}}(u) &\text{ if }\forall i\in B, i\leq\a_1.
                      \end{cases}
\]
We can then use Fa\`a di Bruno's formula to see that for any \(S\) we have
\begin{align}
    \left|\frac{\partial^{|S|}h_1(g_1(u))}{\Pi_{i\in S}\partial v_i}\right|&=\left|\sum_{\pi\in\Pi}d^{|\pi|}_{\zeta}h_1(\zeta)\bigg|_{\zeta=h(u,u')}
    \prod_{B\in\pi}\frac{\partial^{|B|}g_1(u)}{\prod_{i\in B}\partial v_i}\right|\notag\\
    &\leq\sum_{\pi\in\Pi}|5/4 - |\pi||^{|\pi|}\max\{\|{\lambda_{a}}\|_{C^{|\pi|}(\Omega)}, 1\}^{1/4 - |pi|}
    \prod_{B\in\pi}d_{u}^{|B|}{\lambda_{a}}(u)\notag\\
    &\leq|5/4 - |S||^{|S|}\max\{\|{\lambda_{a}}\|_{C^{0}(\Omega)}, 1\}^{1/4 - |S|}
    \|{\lambda_{a}}\|_{C^{|S|}(\Omega)}B_{|S|}\label{eq:term1pac}.
\end{align}
The same argument shows that for any \(T_S\) we have
\begin{equation}\label{eq:term2pac}
    \left|\frac{\partial^{|T_S|}{\lambda_{a}}(u')^{1/4}}{\Pi_{j\in T_S}\partial v_j}\right|\leq|5/4 - |T_S||^{|T_S|}\max\{\|{\lambda_{a}}\|_{C^{0}(\Omega)}, 1\}^{1/4 - |T_S|}
    \|{\lambda_{a}}\|_{C^{|T_S|}(\Omega)}B_{|T_S|}.
\end{equation}


\textbf{Third term.}
Let \(h_2(\zeta)=\zeta^{-1/2}\) and \(g_2(u,u')={\lambda_{a}}(u)+{\lambda_{a}}(u')\). For any \(n\in\N\) we have
\begin{align*}\nonumber
    \left|d^{n}_\zeta h_2(\zeta)\big|_{\zeta = g_2(u,u')}\right|=\left|d^{n}_\zeta \zeta^{-1/2}\big|_{\zeta = g_2(u,u')}\right|
    =\frac{(2n)!}{2^{2n}n!}({\lambda_{a}}(u)+{\lambda_{a}}(u')),
\end{align*}
and for any \(B\), a block of any partition of any \(V_{T_S}\)
\begin{equation*}
    \frac{\partial^{|B|}g_2(u,u')}{\prod_{i\in B}\partial v_i}=                        \begin{cases}
                         d_{u}^{|B|}{\lambda_{a}}(u) &\text{ if }\forall i\in B, i\leq\a_1,\\
                         d_{u'}^{|B|}{\lambda_{a}}(u') &\text{ if }\forall i\in B, i>\a_1,\\
                         0 \quad&\text{ otherwise.}\\
                      \end{cases}
\end{equation*}
We again use Fa\`a di Bruno's formula to see 
that for any $V=V_{S,T_S}$ we have 
\begin{align}
    \left|\frac{\partial^{|V|}h_2(g_2(u,u'))}{\Pi_{i\in V}v_i}\right|
    &=\left|\sum_{\pi\in\Pi}d^{|\pi|}_{\zeta}h_2(\zeta)\bigg|_{\zeta=h(u,u')}
    \prod_{B\in\pi}\frac{\partial^{|B|}g_2(u,u')}{\prod_{i\in B}\partial v_i}\right|\notag\\
    &\leq \sum_{\pi\in\Pi}\frac{(2|\pi|)!}{2^{2|\pi|}|\pi|!}2 \|{\lambda_{a}}\|_{C^{|V|}(\Omega)} \prod_{B\in\pi}\left|\frac{\partial^{|B|}g_2(u,u')}{\prod_{i\in B}\partial v_i}\right|\notag\\
    &\leq(2|V|)2  \|{\lambda_{a}}\|^2_{C^{|V|}(\Omega)}\sum_{\pi\in\Pi}1\notag\\
    &\leq(2|V|)2 \|{\lambda_{a}}\|^2_{C^{|V|}(\Omega)}B_{|V|}.\label{eg:term3pac}
\end{align}

\textbf{Fourth term.}
Let \(h_3(\zeta)=\exp(-\zeta^2/2)\) and $g_3(u,u')=\frac{|u-u'|^2}{({\lambda_{a}}(u)+{\lambda_{a}}(u'))/2}$.
For any \(\bar{S}=(S\cup T_S\cup V_{T_S})^c\),
we use  Fa\`{a} di Bruno's formula to see that
\begin{align*}
    D^{\bar{T}}h_3(g_3(u_1, u_2))=\sum_{\pi\in\Pi}d^{|\pi|}_{\zeta}h_3(\zeta)\bigg|_{\zeta=g_3(u_1,u_2)}
    \prod_{B\in\pi}\frac{\partial^{|B|}g_3(u_1,u_2)}{\prod_{i\in B}\partial v_i}.
\end{align*}
For any block \(B\), we can use the Leibniz rule to see that
\begin{equation*}
    \frac{\partial^{|B|}g_3(u_1,u_2)}{\prod_{i\in B}\partial v_i}=\sum_{\bar{S}}\frac{\partial^{|\bar{S}|}(|u-u'|^2/2)}{\prod_{i\in\bar S}v_i}\frac{\partial^{|B|-|\bar{S}|}({\lambda_{a}}(u)+{\lambda_{a}}(u'))^{-1}}{\prod_{j\notin\bar S}v_j}.
\end{equation*}
We have 
\begin{equation*}
    \left|\frac{\partial^{|\bar{S}|}(|u-u'|^2/2)}{\prod_{i\in \bar{S}}\partial v_i}\right|=                        \begin{cases}
                         |u-u'| &\text{ if }|\bar{S}|=1,\\
                         1 &\text{ if }|\bar{S}|=2\text{ and } \exists i,j \in \bar{S} \text{ with }i \leq\a_1\text{ and }j >\a_1,\\
                         0 \quad&\text{ otherwise,}\\
                      \end{cases}
\end{equation*}
and with \(h_4(\zeta)=\zeta^{-1}\) and \(g_4(u,u')={\lambda_{a}}(u)+{\lambda_{a}}(u')\), we again use Fa\`{a} di Bruno's formula to see that
\begin{equation*}
    \left|\frac{\partial^{|B|-|\bar{S}|}({\lambda_{a}}(u)+{\lambda_{a}}(u'))^{-1}}{\prod_{j\notin\bar S}v_j}\right|\leq\sum_{\pi\in\Pi}\left|d^{|\pi|}_{\zeta}h_4(\zeta)\bigg|_{\zeta=g_4(u,u')}\right|
    \prod_{B\in\pi}\left|\frac{\partial^{|B|}g_4(u_1,u_2)}{\prod_{i\in B}\partial v_i}\right|.
\end{equation*}
where \(\Pi\) is the set of all possible partitions of the set \(\{1,\ldots,|B|-|\bar S|\}\), \(B\in\pi\) means that \(B\) runs over all the blocks of the partition \(\pi\), and $|\pi|$ is the number of blocks in $\pi$.
Furthermore
\begin{align*}
    \left|d^{|\pi|}_{\zeta}h_4(\zeta)\bigg|_{\zeta=g_4(u,u')}\right|= \left|(-1)^{|\pi|}|\pi|!\zeta^{-1-|\pi|}\bigg|_{\zeta=g_4(u,u')}\right| 
    \leq\left||\pi|!(2c)^{-1-|\pi|}\right|,
\end{align*}
and
\begin{equation*}
    \frac{\partial^{|B|}g_4(u,u')}{\prod_{i\in B}\partial v_i}=                        \begin{cases}
                         d_{u}^{|B|}{\lambda_{a}}(u) &\text{ if }\forall i\in B, i\leq\a_1,\\
                         d_{u'}^{|B|}{\lambda_{a}}(u') &\text{ if }\forall i\in B, i>\a_1,\\
                         0 \quad&\text{ otherwise,}\\
                      \end{cases}
\end{equation*}
and hence
\begin{align}
    \left|\frac{\partial^{|B|-|\bar{S}|}({\lambda_{a}}(u)+{\lambda_{a}}(u'))^{-1}}{\prod_{j\notin\bar S}v_j}\right|&\leq
    \sum_{\pi\in\Pi}\left||\pi|!(2c)^{-1-|\pi|}\right|
    \|{\lambda_{a}}\|_{C^{|B|-|\bar{S}|}(\Omega)}\sum_{\pi\in\Pi} \prod_{B\in\pi}1
    \notag\\
   &=(|B|-|\bar{S}|)!(2c)^{-1-(|B|-|\bar{S}|)}\|{\lambda_{a}}\|_{C^{|B|-|\bar{S}|}(\Omega)}B_{|B|-|\bar{S}|}.
    \label{eq:4thgood}
\end{align}

Finally we can bring together \eqref{eq:term1pac}, \eqref{eq:term2pac}, \eqref{eg:term3pac} and \eqref{eq:4thgood} to obtain the bound
\begin{align*}
     &\left|D^{\alpha}k_{\warp,k_i}^{\lambda_{a}}(u,u')\right| \leq \\ 
    &\qquad 2^{1/2}\s^2\sum_S|5/4 - |S||^{|S|}\max\{\|{\lambda_{a}}\|_{C^{0}(\Omega)}, 1\}^{1/4 - |S|}
    \|{\lambda_{a}}\|_{C^{|S|}(\Omega)}B_{|S|} \\
    &\qquad \qquad \Bigg(\sum_{T_S}|5/4 - |T_S||^{T_S}\max\{\|{\lambda_{a}}\|_{C^0(\Omega)}, 1\}^{1/4 - |T_S|}
    \|{\lambda_{a}}\|_{C^{|T_S|}(\Omega)}B_{|T_S|} \\
&\qquad \qquad \qquad\Bigg(\sum_{V_{S,T_S}}(2|V|)!2 \|{\lambda_{a}}(u)\|^2_{C^{|V|}(\Omega)}B_{|V|}
\:(|B|-|\bar{S}|)!(2c)^{-1-(|B|-|\bar{S}|)}C_0\sqrt{(2p)!}\|{\lambda_{a}}\|_{C^{|B|-|\bar{S}|}(\Omega)}B_{|B|-|\bar{S}|}\Bigg)\Bigg)\nonumber.
\end{align*}
Since \(|S|, |T_S|, |V_{T_S}|\leq |\alpha|\), we then see that
\begin{align*}
    \left|D^{\alpha}k(u,u')\right|
\leq 2^{1/2}\s^2&\left(|5/4 - |\a||^{|\a|}\max\{\|{\lambda_{a}}\|_{C^{0}(\Omega)}, 1\}^{1/4 - |\a|}
    \|{\lambda_{a}}\|_{C^{|\a|}(\Omega)}B_{|\a|}\right)^2
\Bigg((2|\a|)!2 \|{\lambda_{a}}(u)\|^2_{C^{|\a|}(\Omega)}B_{|\a|}\\
&\quad(|\a|)!(2c)^{-1-(|\a|)})\|{\lambda_{a}}\|_{C^{|\a|}(\Omega)}B_{|\a|}\Bigg)(B_{|\a|})^3,
\end{align*}
as claimed.
\end{proof}


\end{appendices}
\end{document}